\documentclass{article}
\usepackage{FHTPreamble}
\usepackage{pdflscape}

\newcommand{\PG}{\mathsf{P}} 
\newcommand{\LG}{\mathsf{L}} 
\newcommand{\IG}{\mathsf{I}} 
\newcommand{\NV}{\mathcal{N}} 
\newcommand{\Dl}{\Delta} 
\newcommand{\ecc}{\operatorname{ecc}}
\newcommand{\rad}{\varrho}
\newcommand{\dia}{\delta}
\newcommand{\mec}{\bar{\varepsilon}}

\newcommand{\Blo}{\operatorname{Blo}}
\newcommand{\Abs}{\operatorname{Abs}}
\newcommand{\Pie}{\operatorname{Pie}}
\newcommand{\Rea}{\mathcal{R}} 
\newcommand{\TG}{\mathcal{G}} 
\newcommand{\AH}{\mathcal{H}} 
\newcommand{\Bij}{\mathbb{B}} 
\newcommand{\Pts}{\operatorname{Pts}}
\newcommand{\bP}{\overline{P}}

\newcommand{\Th}{\Theta}
\newcommand{\St}{\mathrm{St}}

\title{Fractal Hyper-Trees I.\\ Combinatorial foundations: definition, enumeration, invariants, morphisms and the algebra $\mathbb C\oplus\mathbb H\oplus M_3(\mathbb C)$}

\author{
    \textbf{Jean-Jacques Salone}\thanks{ORCID: \url{https://orcid.org/0009-0001-1033-5694}} \\
 \small CRREF (Centre de Recherches et de Ressources en Éducation et Formation, EA 4538) \\
 \small University of the Antilles \\
 \small Associate member, LIM (Laboratoire d'Informatique et de Math\'ematiques), University of La R\'eunion \\
 \small \texttt{jean-jacques.salone@univ-antilles.fr}
}
\date{}

\begin{document}
\maketitle

\begin{abstract}
\emph{Fractal Hyper-Trees} (FHTs) are hypergraphs organized in levels: each level consists of disjoint connected clutters, every edge strictly contains each lower component that it meets, and every component of positive level contains a component of the previous level; connected clutters and Schr\"oder hierarchies are the two extreme cases. What the axioms determine is kept apart from what depends on a choice. The axioms determine the hierarchy. The moves are chosen and induce a distance; the twisted construction reconstructs $A_F$ under a precise condition. We prove that every FHT decomposes uniquely into the children of its top component, which carry smaller FHTs, and a host, which is an arbitrary connected clutter on these children. This yields an exact recurrence for the number of FHTs by size and depth, and a species equation for the total; the sequence begins $1, 1, 4, 23, 231, 16947, 490083358, \dots$ The graph of the additions and deletions of one edge, intrinsic to the realization, is already disconnected at three vertices. Adding the fusions and fissions of host edges gives four interior moves (fusion, fission, aggregation, explosion), which connect the FHTs on a given set of vertices; each changes the number of edges $\kappa$ by at most one, and the level moves leave it unchanged. Every FHT on at least two vertices descends to the star by moves located at distance at most one from the root, and $\kappa(\F)-1$ is its distance to the star in the transition graph. For FHTs with automorphism group isomorphic to $S_4$, twisting the real group algebra by the unique non-trivial sign character gives $\mathbb R^{\mathrm{sgn}}[S_4]\cong\mathbb C\oplus\mathbb H\oplus M_3(\mathbb C)$, the finite algebra $A_F$ of Connes' description of the Standard Model. Conversely, $A_F$ determines the abstract group and the cohomology class of the cocycle: if $\mathbb R^{c}[G]\cong A_F$ for a finite group $G$ and a real $2$-cocycle $c$, then $G\cong S_4$ and, after transport, $c$ is cohomologous to $w_1^2$. The FHT axioms alone select neither this group nor this twist. With the twist fixed once for all as the sign of the action on the vertices, $\mathbb R^{\eta_V}[\Aut\F]\cong A_F$ if and only if $\Aut\F\cong S_4$ and some automorphism induces an odd permutation of the vertices; this occurs at every number of vertices $v\ge4$. The algebra is that of Connes; the new fact is its occurrence as the twisted algebra of the automorphism group of FHTs. Every finite group is the automorphism group of a flat FHT. For the homology we separate algebraic and simplicial realisation. The reduced rational homology of a finite simplicial complex with an action of $G$ has an equivariant Euler characteristic in the image of the linearization map $\beta$ from the Burnside ring, and this suffices for a chain complex of permutation modules. Every faithful action of $G$ on a finite connected simplicial complex is rigidified into a flat FHT with the same reduced integral homology. Whether the condition, for modules concentrated in degrees $\ge1$, suffices for a faithful action on a finite connected simplicial complex is left open. The enumeration is exact through nine vertices; the values at eight and nine vertices use external counting data, stated in the text. The computational checks are exhaustive or sampled as specified for each result, and the kind of each check is stated.
\end{abstract}


\tableofcontents

\section{Introduction}\label{sec:intro}

\paragraph{The FHT programme.}
A Fractal Hyper-Tree (FHT) is a hypergraph organized in levels: each level consists of disjoint connected clutters, every edge strictly contains each component of lower level that it meets, and the same rule applies at every scale (Definition~\ref{def:fht}). The programme to which this article belongs asks how much of the finite spectral triple of the Standard Model, in the form given by Connes and by Chamseddine, Connes and Marcolli~\cite{connes96,ccm07,vs15}, can be read on finite FHTs and on their limits, with no input other than their combinatorics, and it says explicitly what has to be added. What the axioms determine is kept apart from what depends on a choice. The axioms determine the hierarchy (Sections~\ref{sec:defsec} and~\ref{sec:enumeration}). The moves are chosen and induce a distance (Section~\ref{sec:morphisms}); the twisted construction reconstructs $A_F$ under a precise condition (Section~\ref{sec:valpha}). This article gathers the combinatorial foundations: the definition of FHTs, their enumeration, their invariants, their elementary morphisms, and the algebra $A_F=\mathbb C\oplus\mathbb H\oplus M_3(\mathbb C)$ obtained as a twisted group algebra of an automorphism group of FHTs.

\paragraph{Context.}
Hypergraphs and clutters (Sperner families)~\cite{berge_hypergraphs,sperner_1928} provide the flat level of FHTs, and hierarchies of subsets, counted by Schr\"oder's fourth problem~\cite{schroder_1870,stanley_ec2}, provide the tree. The invariants come from abstract simplicial complexes, with their automorphism groups and homology~\cite{bjorner,dowker}. The enumeration uses P\'olya theory and combinatorial species~\cite{harary_palmer,bll_species,joyal} together with the Dedekind numbers~\cite{wiedemann_D8}. On the algebraic side, twisted group algebras are the algebras of projective representations of Schur~\cite{schur1904}; the types (real, complex, quaternionic) of their simple factors are given by a twisted Frobenius--Schur indicator~\cite{serre,isaacs,karpilovsky}. That every finite group is the automorphism group of a graph is Frucht's theorem~\cite{frucht39,babai74}; the realisation of prescribed group actions on homology is Steenrod's problem, false over the integers~\cite{carlsson81} and approached rationally through Moore $G$-spaces and the orbit category~\cite{kahn86,tomdieck87,hpy13}, with a recent rigidification theorem for simplicial complexes~\cite{cgv25}. The physical target is the finite algebra of the noncommutative-geometric Standard Model~\cite{connes96,cc97,ccm07,cc08,vs15}, and its extensions: left--right~\cite{ccm07}, Pati--Salam~\cite{ps74,ccvs13ps}, the algebra $M_2(\mathbb H)\oplus M_4(\mathbb C)$ of~\cite{cc08} and the grand symmetry~\cite{dlm14}.

\paragraph{Status conventions.}
Each result of the programme carries two labels. The \emph{epistemic status} is one of
[T] \emph{proved};
[E] \emph{computed} (by a program, numerically or exactly);
[C] \emph{conjecture} (or reading, or agreement, without a rule written in advance);
[P] \emph{postulate}.
The \emph{nature} says where a result comes from:
\emph{derived} (D): obtained from the structure of FHTs, without physical choice, in a form that did not exist before;
\emph{reconstructed} (R): a \emph{known} object is recovered from the structure of FHTs; the novelty is the origin of the object, not the object, and its source is cited where it appears;
\emph{posed} (P): a postulate or a declared choice, on which the results that use it are conditional;
\emph{ad hoc} (A): adjusted or chosen in view of the data, or formulated after the fact;
\emph{borrowed} (B): taken as it stands from elsewhere, without derivation;
\emph{picked} (Pk): a known tool chosen because it fits, without being forced by the structure.
Statuses are written in square brackets and natures in words or with a letter without brackets. Section~\ref{sec:status} gives the table of the results of this article; the method by which computed results are obtained and checked is described in Section~\ref{sec:verification}.

\paragraph{What this article contains.}
Section~\ref{sec:defsec} gives the axioms, the tree of components and its hosts, and the structure lemma (an FHT is a hierarchy with independent admissible hosts). Section~\ref{sec:enumeration} proves the unique decomposition into children and a host and derives an exact recurrence and a species equation for the numbers of FHTs, exact up to nine vertices. Section~\ref{sec:invariants} reduces the automorphism group and the Betti numbers of an FHT to its hosts. It also separates the $260$ FHTs with at most five vertices. Section~\ref{sec:morphisms} introduces the morphisms $\Th_1^{\pm}$, $\Th_2^{\pm}$, $\Omega_1$, which connect all FHTs and descend to the star along moves at distance at most one from the root; they generate a graded poset and groupoids. Section~\ref{sec:valpha} classifies FHTs by the pair $(v,\alpha)$ and proves $\mathbb R^{\mathrm{sgn}}[S_4]\cong\mathbb C\oplus\mathbb H\oplus M_3(\mathbb C)$, with the uniqueness of the pair $(S_4,[w_1^2])$ (the group up to isomorphism, the cohomology class of the cocycle after transport); it adds the closure statement that $A_F$ is the only algebra, among the algebras of the usual ladder listed there, that is a whole twisted group algebra for a sign character, and exhibits some simple factors of the larger algebras as sectors. Section~\ref{sec:representability} proves that every finite group is the automorphism group of a flat FHT and, for the homology, distinguishes an algebraic condition, sufficient for complexes of permutation modules, from the rigidification of faithful actions on finite connected simplicial complexes. Section~\ref{sec:verification} describes the method and gathers the computer checks; Section~\ref{sec:status} gives the status of the results. The appendices contain the tables.

\paragraph{Main results.}
\begin{itemize}
\item \emph{Hosts and complexes.} For $\card V\ge2$, the realization of a flat FHT on $V$ is a connected clutter covering $V$, and, read as a set of facets, it gives a bijection between flat FHTs on $V$ and connected simplicial complexes with vertex set $V$ (Proposition~\ref{prop:recursive_properties} and Section~\ref{sec:effective}); the one-point complex corresponds to the trivial FHT, of depth $0$ and not $1$, which is treated separately. These complexes carry homology and automorphism groups, and the symmetry and homology of any FHT reduce to those of its hosts (Theorem~\ref{thm:autrec}, Proposition~\ref{prop:bettihost}).
\item \emph{The tree of components.} The levels, children, hosts and the structure lemma (Lemma~\ref{lem:structure}) give the hierarchical tree on which the enumeration and the morphisms are built.
\item \emph{The algebra $A_F$.} Theorem~\ref{thm:AF} gives $A_F=\mathbb C\oplus\mathbb H\oplus M_3(\mathbb C)$ as the twisted group algebra $\mathbb R^{\mathrm{sgn}}[\Aut\F]$ of the FHTs with $\Aut\F\cong S_4$, and shows that no other pair (finite group, $\{\pm1\}$-valued $2$-cocycle) gives it, the uniqueness bearing on the abstract group up to isomorphism and on the cohomology class of the cocycle after transport by this isomorphism; Theorem~\ref{thm:hodge} reads the twist as an orientation character and the three factors as the types of Hodge duality. The uniqueness answers the inverse question (which group and which cocycle can produce the whole algebra $A_F$); it does not show that the axioms select $S_4$ or the twist, still less an FHT. With the twist fixed once for all as the sign $\eta_V$ of the action on the vertices, Corollary~\ref{cor:etaV} gives a testable criterion: $\mathbb R^{\eta_V}[\Aut\F]\cong A_F$ if and only if $\Aut\F\cong S_4$ and some automorphism is odd on the vertices. This algebra does not single out an FHT (Remark~\ref{rem:census}).
\item \emph{Sectors and closures.} Proposition~\ref{prop:B1} and Proposition~\ref{prop:spinsector} show that the simple factors of the Pati--Salam algebra and of the algebra $M_2(\mathbb H)\oplus M_4(\mathbb C)$ appear as sectors of twisted algebras of larger groups (the double cover $GL(2,3)$ of $S_4$, and $S_5$), and that, among the algebras listed there, only the bottom of the ladder, $A_F$, is a whole twisted algebra of a sign character (Remark~\ref{rem:closurescope} gives the scope for general cocycles).
\item \emph{Representability.} By Theorem~\ref{thm:repr}, FHTs carry every finite group as automorphism group. The reduced rational homology of a finite simplicial complex with an action of $G$ has an equivariant Euler characteristic in the image of the linearization map $\beta:A(G)\to R_{\mathbb Q}(G)$, which is sufficient for a chain complex of permutation modules. Every faithful action of $G$ on a finite connected simplicial complex is rigidified into a flat FHT with the same reduced integral homology (connectedness is needed to keep the reduced homology in degree $0$). The predictive content of the programme lies in selection principles, not in the bare existence of carriers.
\end{itemize}

\paragraph{Choices and borrowings.} No physical postulate is used. The algebraic construction depends on a declared choice: the algebra attached to an FHT is the twisted group algebra of its automorphism group for the sign $\eta_V$ of its action on the vertices, a rule fixed once for all (Section~\ref{sec:twisted}, Corollary~\ref{cor:etaV}); other characters appear only in computations where they are named. The theorems determine the consequences of that choice, not its necessity; Theorem~\ref{thm:hodge} gives the chosen character a geometric interpretation, as an orientation character, but does not by itself select it. The proofs use classical results, cited where they are used: P\'olya theory and species, Frucht's theorem, the representation theory of finite groups and of twisted group algebras, Hopf's trace formula and Artin's induction theorem. Two ingredients are taken as they stand, without being re-derived: the rigidification theorem of~\cite{cgv25}, in Theorem~\ref{thm:repr}(iv), and the values of the Dedekind numbers $D(7)$, $D(8)$~\cite{wiedemann_D8} and of the terms $a(8)$, $a(9)$ of A003182~\cite{pawelski9}, on which the exact values at eight and nine vertices depend. The identification of $A_F$ is reconstructed (Section~\ref{sec:status}).

\section{Definition and First Properties}\label{sec:defsec}

\paragraph{Motivation.}
Hypergraphs are the standard combinatorial model for interactions involving more than
two elements~\cite{berge_hypergraphs,voloshin_hypergraphs,bretto_hypergraphs}, and they
have become central in the study of higher-order networks~\cite{battiston_beyond,
bianconi_higher,lambiotte_higher}. Two conventions dominate this literature, and they
treat the \emph{inclusion} of one interaction in another in opposite ways. In a
simplicial complex, every subset of an interaction is again an interaction, so
inclusion is automatic. In a clutter, or Sperner family, no interaction may contain
another, so inclusion is forbidden~\cite{sperner_1928}. Many systems are organized in a
third way: a group of elements interacts as a whole, and this group then takes part, as
a single unit, in a larger interaction. Examples are modules inside modules in hierarchical networks~\cite{ravasz_hierarchical} and clusters inside clusters in hierarchical classification~\cite{johnson_clustering}; coalitions that act as single units inside a larger committee~\cite{shapley_simple} are another. In such systems inclusion is meaningful, but constrained: two groups that act as units are either disjoint or nested, whereas the interactions inside a single unit may overlap without being nested.

This paper proposes a class of hypergraphs adapted to this third regime. A \emph{Fractal
Hyper-Tree} ($\FHT$) is organized in levels. Each level consists of disjoint connected
clutters, and every edge of a given level strictly contains each component of a lower
level that it meets. Equivalently, each component is a connected clutter whose
``points'' are the components of lower level that it contains, and the same rule
applies at every scale. The word \emph{fractal} refers to this repetition. The two
extreme cases are classical. On at least two vertices, FHTs of depth one are exactly the connected clutters, i.e.\
the facet sets of connected simplicial complexes, counted by OEIS
A261006~\cite{oeis_A261006}. FHTs in which every component has a single edge are exactly
the hierarchies of subsets, counted by Schr\"oder's fourth problem, i.e.\ by the
series-reduced rooted trees of OEIS A000669~\cite{schroder_1870,stanley_ec2,oeis_A000669}.
General FHTs interpolate between these two families: a Schr\"oder hierarchy carries a clutter at each node.

\subsection{Definition of Fractal Hyper-Trees}\label{sec:definition}

\paragraph{Conventions and inclusion.}
A \emph{hypergraph} is a pair $H=(V_H,E_H)$, where $E_H$ is a finite non-empty set of non-empty finite sets (edges), and $V_H=\bigcup E_H$. It is \emph{connected} if any two vertices are linked by intersecting edges, and \emph{Sperner} (or a \emph{clutter}, or an \emph{antichain}) if no edge contains another. If $\sigma$ is a bijection on $V_H$, we write $\sigma(H)=(\sigma(V_H),\{\sigma(e): e\in E_H\})$. 
The components of an FHT (Definition~\ref{def:fht}) and the admissible hosts (Definition~\ref{def:admissible}) are always hypergraphs in this strict sense: their edge sets are non-empty, their edges are non-empty and they cover their vertex sets. We also use \emph{auxiliary hypergraphs} on a given support $X$: pairs $(X,\mathcal E)$ where $\mathcal E$ is a finite set of non-empty subsets of $X$, possibly empty and not necessarily covering $X$; their connected components are defined as usual, a vertex covered by no edge forming a component by itself. The realization of the trivial FHT, whose edge set is empty, and the residual hypergraphs $(V,E\setminus E_N)$ of Lemma~\ref{lem:grading} are auxiliary hypergraphs.

For two components $H,H'$ we write $H'\subseteq H$ when $V_{H'}\subseteq V_H$, and $H'\subsetneq H$ when the inclusion is strict. This is the inclusion of the supports, not the inclusion of sub-hypergraphs, and it is used only between components; the inclusion of a component in an edge $e$ is written $V_{H'}\subseteq e$.

\begin{definition}[Fractal Hyper-Tree]\label{def:fht}
Let $V$ be a finite non-empty set and $N\in\Nset$. A \emph{Fractal Hyper-Tree} ($\FHT$) of \emph{depth} $N$ on $V$ is a family $\F=(\mathcal{H}_n)_{0\le n\le N}$ of finite sets of hypergraphs whose vertex sets are contained in $V$ (called \emph{components of level $n$}), satisfying:
\begin{enumerate}
\item[(A1)] \textbf{Base level:} $\mathcal{H}_0=\{(\{x\},\{\{x\}\}) : x\in V\}$.
\item[(A2)] \textbf{Connected Sperner components:} For $1\le n\le N$, each $H\in\mathcal{H}_n$ is a connected Sperner hypergraph, and distinct components of level $n$ have disjoint vertex sets.
\item[(A3)] \textbf{Top level:} $\mathcal{H}_N=\{(V,E_N)\}$.
\item[(A4)] \textbf{Fractal inclusion:} For $1\le n\le N$, $H\in\mathcal{H}_n$, $e\in E_H$, $k<n$, and $H'\in\mathcal{H}_k$:
\[ e\cap V_{H'}\neq\emptyset \implies V_{H'}\subsetneq e. \]
\item[(A5)] \textbf{Filiation:} Every component of level $n\ge1$ contains a component of level $n-1$ (that is, its vertex set contains the vertex set of some element of $\mathcal{H}_{n-1}$).
\end{enumerate}
An FHT of depth $1$ is called \emph{flat}.
\end{definition}

\begin{remark}[What (A5) fixes]\label{rem:A5empty}
Without (A5) the grading would be ambiguous. On $V=\{1,2\}$, the family $\mathcal H_0=\{(\{1\},\{\{1\}\}),(\{2\},\{\{2\}\})\}$, $\mathcal H_1=\varnothing$, $\mathcal H_2=\{(V,\{12\})\}$ satisfies (A1)--(A4) and has the realization $\{12\}$ of the flat FHT, with depth $2$ instead of $1$: an empty level has been inserted between the leaves and the root without changing the realization. Axiom (A5) excludes such empty levels, and the realization then determines the levels (Lemma~\ref{lem:grading}).
\end{remark}

\begin{remark}[Filiation is a condition on components]\label{rem:A5}
Axiom (A5) concerns components, not edges. An edge of level $n\ge1$ is only required to contain the lower components that it meets (A4); it need not contain a component of level $n-1$, provided its component does. For instance $\{12,34,123\}$ is an FHT of depth $2$ on $\{1,2,3,4\}$: its level-$1$ components are $(\{1,2\},\{12\})$ and its level-$2$ component is the connected clutter $\{123,34\}$, whose edge $34$ contains no component of level $1$, whereas its edge $123$ contains $12$. By Lemma~\ref{lem:laminar}, the level of a component is determined by the hierarchy: it is $1$ plus the maximal level of its children (Corollary~\ref{cor:levels}), and the levels are recomputed from the leaves whenever an FHT is derived from another one. In particular a component of level $N$ may have children of level strictly smaller than $N-1$ (the singletons $3$ and $4$ in $\{12,123,124\}$), the axiom only requiring that at least one child has level $N-1$.
\end{remark}

\paragraph{Trivial Case and General Assumption.}
We have $N=0$ if and only if $|V|=1$. If $N=0$, then $\mathcal H_0=\mathcal H_N$, and by (A1) and (A3) the level-$0$ components, one for each vertex, reduce to the single component $(V,E_0)$, so $|V|=1$. Conversely, let $V=\{x\}$ and suppose $N\ge1$. By (A2) and (A3) the top component is $(V,E_N)$ with $E_N$ a non-empty set of non-empty subsets of $\{x\}$, so $E_N=\{\{x\}\}$; the edge $\{x\}$ meets the level-$0$ component $(\{x\},\{\{x\}\})$, and (A4) would give $\{x\}\subsetneq\{x\}$, which is absurd. So $|V|=1$ yields the unique \emph{trivial FHT}, $\mathcal{H}_0=\{(\{x\},\{\{x\}\})\}$, whose realization (Section~\ref{sec:realization}) is the auxiliary hypergraph $(\{x\},\varnothing)$. 

In the following subsections we assume $|V|\ge 2$.

\subsection{Structural Lemmas and Tree Structure}\label{sec:structural}

\begin{lemma}[Laminarity]\label{lem:laminar}
The vertex sets of all components form a laminar family: any two of them are either disjoint or nested. More precisely, distinct components of the same level have disjoint vertex sets, and if $H\in\mathcal{H}_n$ and $H'\in\mathcal{H}_k$ with $k<n$ share a vertex, then $V_{H'}\subsetneq e\subseteq V_H$ for some $e\in E_H$. In particular, if $H'\subsetneq H$ then $\lev(H')<\lev(H)$.
\end{lemma}
\begin{proof}
For components of the same level $n$, the lemma follows from axiom (A1) if $n=0$ and from axiom (A2) if $n\ge1$.
If $H\in\mathcal{H}_n$ and $H'\in\mathcal{H}_k$ with $k<n$ and if $V_H \cap V_{H'} \neq \emptyset$, then there exists at least one edge $e \in E_H$ such that $e \cap V_{H'} \neq \emptyset$. By axiom (A4), this implies $V_{H'} \subsetneq e$ and so $V_{H'} \subsetneq V_H$: the vertex sets are nested. Finally, if $H'\subsetneq H$, the levels differ (components of the same level are disjoint), and $\lev(H')>\lev(H)$ would give $V_H\subsetneq V_{H'}$ by what precedes; so $\lev(H')<\lev(H)$.
\end{proof}

\begin{remark}[Level of an edge]\label{rem:level}
Every edge $e$ of a component belongs to the own edges $E_H$ of a single component $H$. If $e\in E_K\cap E_{K'}$ with $K\ne K'$, both contain $e$, so they meet and are nested, say $K'\subsetneq K$, and $\lev(K')<\lev(K)$ (Lemma~\ref{lem:laminar}); since $K'$ meets $e$, axiom (A4) gives $V_{K'}\subsetneq e$, which contradicts $e\subseteq V_{K'}$. The \emph{level} $\lev(e)$ of $e$ is the level of this component.
\end{remark}

\begin{lemma}[Branching]\label{lem:branching}
Let $e$ be an edge of level $n\ge1$, and let $C(e)$ be the set of maximal (for inclusion) elements among the vertex sets of components of level $<n$ that are contained in $e$. Then $C(e)$ is a partition of $e$ into at least two parts (they are children of the component of $e$, see Lemma~\ref{lem:children}).
\end{lemma}
\begin{proof}
First, axiom (A4) implies that every component of level $<n$ that meets $e$ has its vertex set strictly contained in $e$; in particular every element of $C(e)$ is strictly contained in $e$.
Second, if two distinct elements of $C(e)$ intersect, then by Lemma~\ref{lem:laminar} they would be nested; this contradicts their maximality.
Third, for every vertex $x\in e$ the level-$0$ component $(\{x\},\{\{x\}\})$ meets $e$, hence is contained in $e$, so in some element of $C(e)$: thus $C(e)$ covers $e$.
Fourth, if $C(e)$ had a single element $M$, it would cover $e$, so $M=e$, which contradicts $M\subsetneq e$. Hence $|C(e)|\ge2$. (Axiom (A5) is not used.)
\end{proof}

\begin{theorem}[FHT Tree Structure]\label{thm:tree}
An FHT forms a tree structure where:
\begin{itemize}
\item The \textbf{nodes} are all the components of $\F$ across all levels $0 \le n \le N$.
\item The \textbf{root} is the unique top-level component in $\mathcal{H}_N$.
\item The \textbf{leaves} are the level-$0$ components.
\item A \textbf{child relation} is defined by: a component $H'$ of level $k < n$ is a child of $H \in \mathcal{H}_n$ if $H'\subsetneq H$ and if $H'$ is maximal (with respect to this inclusion relation).
\end{itemize}
\end{theorem}
\begin{proof}
Laminarity (Lemma~\ref{lem:laminar}) implies that the inclusion relation between components forms a forest. Axiom (A3) ensures there is a unique root and axiom (A1) that the level-$0$ components are leaves.
\end{proof}

\begin{corollary}[Levels]\label{cor:levels}
Let $H$ be a component of level $n\ge1$. Every child of $H$ has level $<n$ and at least one child has level $n-1$, so that
\[
\lev(H)=1+\max\{\lev(C): C\text{ child of }H\}.
\]
The levels of an FHT are determined by its tree of components, by induction from the leaves.
\end{corollary}
\begin{proof}
A child is strictly contained in $H$, so its level is $<n$ (Lemma~\ref{lem:laminar}). By (A5), $H$ contains a component $H'$ of level $n-1$, and $H'\subsetneq H$. The components strictly contained in $H$ that contain $H'$ are nested (Lemma~\ref{lem:laminar}); the largest of them is maximal among the components strictly contained in $H$, that is, it is a child $C$ of $H$. If $C\ne H'$, then $H'\subsetneq C$ and $\lev(C)>n-1$ (Lemma~\ref{lem:laminar}), which is impossible since $\lev(C)<n$. Hence $H'=C$ is a child of level $n-1$.
\end{proof}

\begin{lemma}[Sub-FHT of a component]\label{lem:subfht}
Let $K$ be a component of level $m$. The components of $\F$ contained in $K$, with their levels, form an FHT $\F_K$ of depth $m$ on $V_K$: its components of level $n$ are the elements of $\mathcal H_n$ contained in $K$, for $0\le n\le m$.
\end{lemma}
\begin{proof}
(A1): for $x\in V_K$ the level-$0$ component $(\{x\},\{\{x\}\})$ is contained in $K$, and these are the only level-$0$ components contained in $K$. (A2) and (A4) are inherited, since they bear on the components of $\F$ and on their edges. (A3): a component contained in $K$ and different from $K$ has a level $<m$ (Lemma~\ref{lem:laminar}), so $K$ is the only component of level $m$, and there is none of higher level. (A5): a component $H\subseteq K$ of level $n\ge1$ contains a component of level $n-1$ by (A5) in $\F$, which is contained in $K$. For $m=0$, $\F_K$ is the trivial FHT on the singleton $V_K$.
\end{proof}

Since distinct components have distinct vertex sets (Lemma~\ref{lem:laminar}: they are disjoint if they have the same level, and strictly nested otherwise), a component is identified with its vertex set when it is considered as a child or as a parent. So we define:
\begin{itemize}
\item The set $\Ch(H)$ of the \emph{children} of a component $H$ of level $n\ge1$, as in Theorem~\ref{thm:tree}. The level of a child is $<n$ but is otherwise independent of the level of its parent; in particular a child may be a leaf.
\item The \emph{parent} $\Par(H)$ of every component $H$ of level $n < N$ as the unique component of level $> n$ whose vertex set contains $V_H$ and is minimal (with respect to this property); thus $H\in\Ch(\Par(H))$.
\item The \emph{ancestors} $\Anc(H)=\{\Par(H),\Par(\Par(H)),\dots\}$ of every component $H$ of level $n < N$: the components in the sequence of parents up to the root. We write $\Par^i(H)$ for the $i$-th ancestor ($\Par^0(H)=H$). The root has no ancestor.
\item The \emph{descendants} $\Desc(H)$ of a component $H$: the components $K$ such that $H\in\Anc(K)$. By laminarity (Lemma~\ref{lem:laminar}) these are exactly the components strictly contained in $H$; the children of $H$ are its descendants of generation one, and the leaves contained in $H$ are its descendants without children.
\item The \emph{lowest common ancestor} (first common ancestor) $\lca(H_1,H_2)$ of two components $H_1, H_2 \in \F$: the smallest component whose vertex set contains $V_{H_1}\cup V_{H_2}$. It exists by laminarity (Lemma~\ref{lem:laminar}); it equals $H_1$ if $H_1$ is an ancestor of $H_2$ or equal to it, and its level is at least $\max(\lev(H_1), \lev(H_2))$.
\item The \emph{siblings} of a component $H$ of level $n<N$: the other children of its parent, $\Sib(H)=\Ch(\Par(H))\setminus\{H\}$.
\item Two distinct components $H_1,H_2$, neither of which is a descendant of the other, are \emph{cousins of type $(i,j)$}, $i,j\ge1$, if $\lca(H_1,H_2)=\Par^i(H_1)=\Par^j(H_2)$. Siblings are the cousins of type $(1,1)$; the cousins of type $(i,j)\neq(1,1)$ are the cousins in the strict sense.
\end{itemize}

\begin{lemma}[Children and host]\label{lem:children}
Let $H$ be a component of level $n\ge1$. Then:
\begin{enumerate}
\item $\Ch(H)$ is a partition of $V_H$;
\item for every edge $e\in E_H$, the set $\hat e:=\{C\in\Ch(H): C\subseteq e\}$ contains every child meeting $e$, satisfies $e=\bigcup\hat e$ and $|\hat e|\ge2$; it is the partition $C(e)$ of Lemma~\ref{lem:branching};
\item the hypergraph $\widehat{H}:=(\Ch(H),\{\hat e: e\in E_H\})$ is a connected clutter, each of whose edges has at least two elements; in particular $|\Ch(H)|\ge2$.
\end{enumerate}
\end{lemma}
\begin{proof}
(i) Two distinct children are nested or disjoint (Lemma~\ref{lem:laminar}); nested is excluded by maximality, so they are disjoint. Let $x\in V_H$ and let $e\in E_H$ contain $x$. The level-$0$ component $\{x\}$ meets $e$, hence $\{x\}\subsetneq e\subseteq V_H$ by (A4): it lies strictly inside $H$, hence inside a maximal such component, that is, a child. So the children cover $V_H$.
(ii) A child $C$ has level $<n$, so if it meets $e$ then $C\subsetneq e$ by (A4). Hence $e$ is the union of the children meeting it, which are the elements of $\hat e$. If $\hat e=\{C\}$, then $C=e$, contradicting $C\subsetneq e$; thus $|\hat e|\ge2$. Finally, an element of $C(e)$ lies strictly inside $H$, hence inside a child, which meets $e$ and so lies inside $e$: by maximality it equals this child. Conversely, a child contained in $e$ is maximal among the components contained in $e$.
(iii) Since the children are disjoint, $\hat e\subseteq\hat e'$ if and only if $e\subseteq e'$, and $\hat e\cap\hat e'\neq\emptyset$ if and only if $e\cap e'\neq\emptyset$. So $\widehat{H}$ is Sperner and connected because $E_H$ is. Its vertex set is $\Ch(H)$ since $\bigcup E_H=V_H$ and the children partition $V_H$.
\end{proof}

\begin{definition}[Host of a component]\label{def:hostcomp}
Let $H$ be a component of level $n\ge1$. The hypergraph $\widehat{H}$ of Lemma~\ref{lem:children} is the \emph{host} of $H$: it is $H$ with its edges rewritten on its children, $e\mapsto\hat e$. Its edge set is $E_{\widehat{H}}=\{\hat e: e\in E_H\}$. For a child $C$, let $\F_C$ be the sub-FHT formed by the components of $\F$ contained in $C$ (an FHT on $C$, Lemma~\ref{lem:subfht}), and let $\tau(C)$ be its isomorphism class. The host whose vertices $C$ are coloured by $\tau(C)$ is the \emph{coloured host} $\widehat{H}^{\tau}$; an isomorphism of coloured hosts is a bijection of the sets of children which preserves the colours and the edges.
\end{definition}

\begin{remark}
Here $E_H$ is the set of the own edges of $H$: the edges of the descendants of $H$ are not in it, and an edge has no level of its own, only the level of the component it belongs to. The host $\widehat{H}$ depends on the FHT $\F$ through $\Ch(H)$, and not only on the hypergraph $H=(V_H,E_H)$.
\end{remark}

\subsection{Realization, Uniqueness and Grading}\label{sec:realization}

\begin{definition}[Realization]
The \emph{realization} of $\F$ is the hypergraph $(V,E)$ where $E=\bigcup_{n=1}^{N}\bigcup_{H\in\mathcal{H}_n}E_H$. 
\end{definition}

\begin{remark}
Level-$0$ singletons never appear in $E$. 
\end{remark}

\noindent We also write $\Rea(\F)=E$ for the set of edges of the realization; it is the union of the sets $E_K$ of own edges of the non-leaf components $K$ of $\F$ (this notation is used from Section~\ref{sec:edges} on).

\paragraph{Examples.}
On $V=\{1,2,3\}$, the realization $\{123\}$ is flat ($\mathcal{H}_1=\{(V,\{123\})\}$). The realization $\{12,123\}$ has depth $2$ ($\mathcal{H}_1=\{(\{1,2\},\{12\})\}$ and $\mathcal{H}_2=\{(V,\{123\})\}$): vertex $3$ enters the structure at level $2$. For FHTs with $\card{V} \ge 10$, dots may be introduced to separate vertex indices.

\begin{lemma}[Realization Uniqueness]\label{lem:grading}
For $|V|\ge2$, an FHT is uniquely determined by its realization.
\end{lemma}
\begin{proof}
We argue by induction on $|V|$. Let $\F$ be an FHT of depth $N\ge1$ on $V$ with realization $(V,E)$.

(a) \emph{The edges of level $N$ are exactly the inclusion-maximal edges of $E$.} Let $e$ be an edge of level $N$ and $e'\in E\setminus E_N$ an edge of a component $H'$ of level $<N$. If $e\cap e'\neq\emptyset$, then $e$ meets $H'$, so $e'\subseteq V_{H'}\subsetneq e$ by (A4); in any case $e\not\subseteq e'$. Two edges of level $N$ are incomparable by (A2). Hence $e$ is maximal. Conversely, let $e'\in E\setminus E_N$ lie in a component $H'$ of level $<N$. By (A3) some edge $e$ of level $N$ contains a vertex of $H'$, hence $e'\subseteq V_{H'}\subsetneq e$ and $e'$ is not maximal.

(b) \emph{The children of the top component $H_N=(V,E_N)$ are the connected components of the auxiliary hypergraph $(V,E\setminus E_N)$.} By Lemma~\ref{lem:children}, $\Ch(H_N)$ is a partition of $V$. Each edge of $E\setminus E_N$ belongs to a component of level $<N$, hence lies in exactly one child, and the edges lying in a child $B$ form a connected hypergraph on $B$ (the component with vertex set $B$ is connected, and contains all the lower components inside $B$). Since no edge of $E\setminus E_N$ meets two children, the children are exactly the connected components of $(V,E\setminus E_N)$, the isolated vertices being singleton children.

(c) \emph{Conclusion.} The realization thus determines $E_N$, the partition $\Ch(H_N)$, and the edges contained in each child. The components of level $<N$ are the components of the sub-FHTs $\F_B$ on the children $B$ (Lemma~\ref{lem:subfht}), whose realizations are determined by $E$, hence, by induction on $|B|<|V|$, are determined by $E$; and the only component of level $N$ is $(V,E_N)$. The levels are then fixed by $\lev(H)=1+\max\{\lev(H'): H'\text{ child of }H\}$ (Corollary~\ref{cor:levels}).
\end{proof}

\begin{definition}[Isomorphism of FHTs]\label{def:fhtiso}
An \emph{isomorphism} from an FHT $\F=(\mathcal H_n)_{n\le N}$ on $V$ to an FHT $\F'=(\mathcal H'_n)_{n\le N'}$ on $V'$ is a bijection $\sigma:V\to V'$ such that $N=N'$ and $H\mapsto\sigma(H)$ maps $\mathcal H_n$ onto $\mathcal H'_n$ for every $n$; it thus transports the components, their own edges and their levels.
\end{definition}

\begin{lemma}\label{lem:isoreal}
FHT isomorphisms coincide with hypergraph isomorphisms of their realizations. 
\end{lemma}
\begin{proof}
An isomorphism of FHTs maps the own edges of each component onto those of its image, hence the realization of $\F$ onto that of $\F'$. Conversely, let $\sigma:V\to V'$ be a bijection mapping the edge set $E$ of the realization of $\F$ onto the edge set $E'$ of that of $\F'$. The family $\sigma(\F)=(\sigma(\mathcal H_n))_n$ is an FHT on $V'$, since the axioms are invariant under bijections, and its realization is $\sigma(E)=E'$. By Lemma~\ref{lem:grading} (directly if $\card V=1$), $\sigma(\F)=\F'$, so $\sigma$ is an isomorphism of FHTs.
\end{proof}

\subsection{Admissible hosts and the structure lemma}\label{sec:setup}

\subsubsection{Hierarchical vocabulary}\label{sec:vocab}

We use the vocabulary of Section~\ref{sec:structural}. A component is identified with its vertex set. For a component $H$ of an FHT $\F$ on $V$:
$\Ch(H)$ is the set of its children, $\Par(H)$ its parent, $\Anc(H)$ its ancestors, $\Desc(H)$ its descendants, $\Sib(H)$ its siblings; $\lca(H_1,H_2)$ is the first common ancestor, and two components are cousins of type $(i,j)$ as defined in Section~\ref{sec:structural}. The \emph{root} is the component $V$ and the \emph{leaves} are the singletons. A leaf $\{x\}$ that is a child is written $x$: for instance the root of $\{12,123,124\}$ has the children $12$, $3$ and $4$. A component is \emph{non-leaf} if it has children. We write $\rho(K)=\card{\Anc(K)}$ for the distance from $K$ to the root ($\rho(V)=0$, the children of the root have $\rho=1$), and $\F_K$ for the sub-FHT formed by the components contained in $K$ (an FHT on $K$).
Only these notions, the own edges $E_H$ of a component, its host $\widehat H$ (Definition~\ref{def:hostcomp}) and the blocks $\Blo(H)$ of Section~\ref{sec:blocks} are used in the definitions of the morphisms.

\subsubsection{Hosts and the structure lemma}\label{sec:structure}

\begin{definition}[Admissible host]\label{def:admissible}
Let $X$ be a finite set with $\card{X}\ge 2$. An \emph{admissible host} on $X$ is a connected clutter $\mathcal E$ of subsets of $X$ with $\bigcup\mathcal E=X$. We write $\AH(X)$ for the set of admissible hosts on $X$.
\end{definition}

\begin{remark}[No singleton edge]\label{rem:nosingleton}
An edge of an admissible host has at least two elements: if $\{x\}\in\mathcal E$, then no other edge contains $x$ (clutter) and $\{x\}$ meets no other edge, which contradicts connectedness since $\card{X}\ge2$.
\end{remark}
By Lemma~\ref{lem:children}, the host $\widehat H$ of a non-leaf component $H$ is an admissible host on $\Ch(H)$; its set of edges is $E_{\widehat H}=\{\hat e:e\in E_H\}$.

\begin{lemma}[Structure lemma]\label{lem:structure}
Let $V$ be a finite set, $\card{V}\ge2$. The FHTs on $V$ are in bijection with the \emph{decompositions} of $V$, that is, the pairs consisting of
\begin{enumerate}
\item a \emph{hierarchy} $\mathcal T$ of $V$: a family of subsets of $V$ containing $V$ and all singletons, two members of which are nested or disjoint, and such that every non-singleton member $K$ has at least two children (maximal proper members of $\mathcal T$ contained in $K$);
\item for every non-singleton $K\in\mathcal T$, an admissible host $\mathcal E_K\in\AH(\Ch(K))$.
\end{enumerate}
The FHT $\F$ associated with $(\mathcal T,(\mathcal E_K))$ has the members of $\mathcal T$ as components, the own edges $E_K=\{\bigcup f: f\in\mathcal E_K\}$ for $K$ non-singleton, and the level of $K$ equals $1+\max$ of the levels of its children (level $0$ for singletons), as it must by Corollary~\ref{cor:levels}. Conversely, $\mathcal T$ is the set of components of $\F$ and $\mathcal E_K=E_{\widehat K}$.
\end{lemma}

\begin{proof}
An FHT gives such a pair by Lemma~\ref{lem:laminar}, Theorem~\ref{thm:tree} and Lemma~\ref{lem:children}. Conversely, let $(\mathcal T,(\mathcal E_K))$ be a decomposition and define the components and levels as stated. Since the level of a member is strictly larger than the level of its children, two distinct members of the same level are not nested, hence disjoint. (A1) holds because the singletons are the members of level $0$; (A3) because $V$ is the unique maximal member. For (A2): since the children of $K$ are non-empty and pairwise disjoint, the map $f\mapsto\bigcup f$ is a bijection from $\mathcal E_K$ onto $E_K$ which preserves and reflects inclusion and non-empty intersection; as $\mathcal E_K$ is a connected clutter covering $\Ch(K)$, $E_K$ is a connected clutter covering $K$. (A5) holds since $K$ has a child of level $\lev(K)-1$. For (A4), let $e=\bigcup f\in E_K$ and let $K'$ be a member of level $<\lev(K)$ meeting $e$. As $K'$ is nested with $K$ or disjoint from it, and meets $K$, either $K'\supseteq K$, which is impossible because $\lev(K')<\lev(K)$, or $K'\subseteq K$. Then $K'$ lies in a child $C$ of $K$ (or is a child), $C$ meets $e$ hence $C\in f$ and $V_{K'}\subseteq C\subsetneq e$ because $\card{f}\ge2$. The two constructions are inverse to each other.
\end{proof}

In particular, replacing the host of one component by any other admissible host on the same children produces an FHT, and the hosts of different components can be chosen independently. This is the substitution property underlying the recursive decomposition (Proposition~\ref{prop:recursive}).

\subsubsection{Edges and realization}\label{sec:edges}

\begin{lemma}[Reading the hierarchy from the edges]\label{lem:read}
Let $\F$ be an FHT on $V$, $\card V\ge2$, and let $E_V$ be the set of own edges of the root (a set of subsets of $V$, not the edge set $E_{\widehat V}$ of the root host, whose elements are sets of children).
\begin{enumerate}
\item An edge belongs to the own edges of exactly one component. Thus $\Rea(\F)$ is the disjoint union of the sets $E_K$.
\item $E_V=\max \Rea(\F)$ (the maximal elements for inclusion); every other edge is strictly contained in an edge of $E_V$.
\item The children of the root are the connected components of the auxiliary hypergraph $(V,\Rea(\F)\setminus E_V)$, where the vertices covered by no edge form singleton components.
\end{enumerate}
So $\Rea(\F)$ determines $\F$, by applying (ii) and (iii) recursively in each child of the root.
\end{lemma}

\begin{proof}
This lemma restates, in the notation used from here on, steps of the proof of Lemma~\ref{lem:grading}: (i) is Remark~\ref{rem:level}; (ii) and (iii) are steps (a) and (b) of that proof, the root being the top component $H_N$ and $E_V=E_N$; the last statement is Lemma~\ref{lem:grading}.
\end{proof}

\section{Enumeration}\label{sec:enumeration}

The enumeration serves two purposes. First, it makes precise how the two classical
problems combine. Every FHT decomposes uniquely into the children of its top component, which carry smaller FHTs,
and a host clutter on these children. This decomposition is related to the
composition of clutters and simple games~\cite{shapley_simple,billera_clutters} and to
substitution decomposition~\cite{gallai_1967,mohring_radermacher} (the precise relation is stated after Proposition~\ref{prop:recursive}), and its counting
follows the classical pattern of compound structures. A well-known instance is the
enumeration of two-terminal series-parallel networks~\cite{riordan_shannon}, built from the two operations of series and parallel connection; for $n\ge2$ edges their number is twice the number of Schr\"oder hierarchies on $n$ points (OEIS A000669), the root of the corresponding series-reduced tree being of series or of parallel type. The difference here is that the hosts are
arbitrary connected clutters rather than two operations. This leads to a two-sort
P\'olya composition~\cite{harary_palmer,bll_species}, in which the host numbers
generalize the numbers of antichains~\cite{oeis_A003182}. Second, the
counts show what the hierarchy adds to the flat structures. On the computed range ($v\le9$), the proportion of
imbricated FHTs decreases rapidly, from $0.39$ at $v=4$ to $1.6\cdot10^{-17}$ at $v=9$; whether it tends to $0$ is not proved here. The imbricated FHTs carry the whole
hierarchical information, and their distribution by depth is non-trivial.

\subsection{Recursive decomposition}

\begin{definition} Let $\mathcal{B}$ be a partition of $V$ whose elements are called \emph{blocks}. (A block is a cell of a partition used to build an FHT; once the FHT is built, the blocks of its decomposition are the children of its top component, Proposition~\ref{prop:recursive_properties}.) 
For a subset
$\hat e\subseteq\mathcal{B}$ we write $\bigcup\hat e=\bigcup_{B\in\hat e}B$.
Conversely, if $e\subseteq V$ is a union of blocks of $\mathcal{B}$, we denote
$\hat e=\{B\in\mathcal{B}: B\subseteq e\}$, and we have $e=\bigcup\hat e$. 

A \emph{host} on $\mathcal{B}$ is a hypergraph $\widehat{H}$ whose edges are subsets $\hat e\subseteq\mathcal{B}$.
\end{definition}

\begin{remark}
Since the blocks are subsets of $V$ (disjoint and non-empty), $\bigcup\hat e$ is also a subset of $V$. But, unlike the partition $\mathcal{B}$, it does not necessarily cover $V$, nor does the host $\widehat{H}$.
\end{remark}

\begin{definition}[Recursive decomposition]\label{def:recursive decomposition}
Let $|V|\ge2$. A \emph{recursive decomposition} of $V$ consists of:
\begin{enumerate}
 \item a partition $\mathcal{B}$ of $V$ into $m\ge2$ blocks, where each block $B$ carries an FHT $\mathcal{F}_B$ on $B$ of depth $d_B$ (the trivial FHT for a singleton block);
 \item a host $\widehat{H}$ on $\mathcal{B}$ which is a connected clutter with vertex set
 $\mathcal{B}$ (i.e.\ $\bigcup_{\hat e\in E_{\widehat{H}}}\hat e=\mathcal{B}$); no other condition is imposed on the host.
\end{enumerate}
We put $L:=\max_{B\in\mathcal{B}}d_B$, the maximal depth of the blocks, and $N:=L+1$. The \emph{realization} of the decomposition is $(V,E)$ with
\[ E=\bigcup_{B\in\mathcal{B}}E_B\ \cup\ \Big\{\textstyle\bigcup\hat e : \hat e\in E_{\widehat{H}}\Big\}, \]
where $E_B$ is the realization of $\mathcal{F}_B$ if $|B|\ge2$, and $E_B=\emptyset$ if
$|B|=1$.
\end{definition}

\begin{proposition}[Recursive decomposition equivalence]\label{prop:recursive}
Let $|V|\ge2$. The map which sends an FHT $\F$ of depth $N$ on $V$, with top component $H_N$, to $\big(\Ch(H_N),(\F_C)_{C\in\Ch(H_N)},\widehat H_N\big)$ is a bijection from the FHTs on $V$ onto the recursive decompositions of $V$. It preserves the realization, and the depth of $\F$ is
\[ N=1+\max_{B\in\mathcal B}d_B . \]
\end{proposition}

The top edges $\bigcup\hat e$ are obtained from the host by substituting for each of its points $B$ the one-edge clutter $\{B\}$: this is a special case of the composition of clutters~\cite{shapley_simple,billera_clutters}. The realization also keeps the edges of the $\F_B$, so that it is not itself a composition of clutters, and it is not a clutter as soon as $N\ge2$. The decomposition is also related to the substitution decomposition~\cite{mohring_radermacher}, the blocks playing the role of modules.

\begin{proof}
\emph{From an FHT to a recursive decomposition.} Let $H_N=(V,E_N)$ be the top component and $\mathcal{B}:=\Ch(H_N)$. By Lemma~\ref{lem:children}, $\mathcal{B}$ is a partition of $V$ into $m\ge2$ sets, every top edge is $e=\bigcup\hat e$, and the host $\widehat{H}_N$ is a connected clutter with vertex set $\mathcal{B}$. For $B\in\mathcal B$, the sub-FHT $\F_B$ is an FHT on $B$ of depth $d_B=\lev(B)$ (Lemma~\ref{lem:subfht}); its realization $E_B$ is empty if $|B|=1$. Every component of level $<N$ is contained in a child (Lemma~\ref{lem:laminar}), so $E$ is the union of the $E_B$ and of the top edges. Finally $N=1+\max_B d_B$ by Corollary~\ref{cor:levels}.

\emph{From a recursive decomposition to an FHT.} For $0\le n\le L$, let $\mathcal{H}_n$ be the union over
$B\in\mathcal{B}$ of the components of level $n$ of $\mathcal{F}_B$ (none if $n>d_B$). Let
$\mathcal{H}_{L+1}=\{(V,E_{\mathrm{top}})\}$ with
$E_{\mathrm{top}}=\{\bigcup\hat e:\hat e\in E_{\widehat{H}}\}$. We now prove that the family $\F=(\mathcal{H}_n)_{0\le n\le N}$, with $N=L+1$, is an FHT by checking the five axioms.
\begin{itemize}
 \item (A1): the level-$0$ components of all the $\mathcal{F}_B$ are the singletons
 of $V$.
 \item (A2): below level $L+1$, components are connected and Sperner. Two of them are
 disjoint if they lie in different
 blocks. By axiom (A2) for $\mathcal{F}_B$, they are also disjoint if they lie in the same block. At level $L+1$, as the correspondence $\hat e\leftrightarrow\bigcup\hat e$
 preserves inclusion and intersection, $E_{\mathrm{top}}$ is a connected
 Sperner hypergraph.
 \item (A3): $\bigcup E_{\mathrm{top}}=\bigcup\mathcal{B}=V$.
 \item (A4): let $e=\bigcup\hat e$ be a top edge and $H'$ a lower component, contained
 in a block $B$. If $e\cap V_{H'}\neq\emptyset$, then $B\in\hat e$. By Remark~\ref{rem:nosingleton},
 $|\hat e|\ge2$, and we get $V_{H'}\subseteq B\subsetneq e$. By axiom (A4) for $\mathcal{F}_B$, it follows that for any edge $e$
 of level $\le L$ contained in a block $B$, lower components inside $B$ also satisfy
 (A4), and those in other blocks are disjoint from $e$.
 \item (A5): the top component $(V,E_{\mathrm{top}})$ contains every block, in particular a block $B$ of depth $L$, which is
 the vertex set of the top component of $\mathcal{F}_B$, a component of level $L=N-1$ (a level-$0$ singleton if
 $L=0$). The containment is strict since $m\ge2$. For components of lower level,
 (A5) is inherited from the $\mathcal{F}_B$. No condition is needed on the individual edges of the host.
\end{itemize}
The realization is by construction $(V,E)$, singleton blocks contributing no edges, and the decomposition of this FHT is $(\mathcal B,(\F_B),\widehat H)$ (Proposition~\ref{prop:recursive_properties}(ii)). Since an FHT is determined by its realization (Lemma~\ref{lem:grading}), the two constructions are inverse to each other.
\end{proof}

\begin{proposition}\label{prop:recursive_properties}
 Let $|V|\ge2$. With the notation of Definition~\ref{def:recursive decomposition}, we have the following properties:
\begin{enumerate}
 \item we have the relation $N=L+1$;
 \item the blocks are the children of the top component $H_N$ (the maximal vertex sets of components of level $<N$), that is $\mathcal{B}=\Ch(H_N)$; the host of the decomposition is the host $\widehat{H}_N$ of $H_N$ and $\F_B$ is the sub-FHT of the components contained in $B$. In particular the recursive decomposition is unique;
 \item since $m\ge2$, every edge of the host contains at least two children;
 \item for flat FHTs, the realizations are exactly the connected clutters with vertex set $V$.
\end{enumerate}
\end{proposition}

\begin{proof}
\begin{enumerate}
 \item \emph{Relation $N = L + 1$:} this is the depth formula of Proposition~\ref{prop:recursive} (Corollary~\ref{cor:levels}).

 \item \emph{Proof of the identification and uniqueness:} Starting from the recursive decomposition, every
 component of level $\le L$ is in a block, and each block is the vertex set of the top
 component of $\mathcal{F}_B$. So the maximal vertex sets of components of level $<L+1=N$ are
 exactly the blocks, that is $\mathcal{B}=\Ch(H_N)$, and restricting to each block gives back $\mathcal{F}_B$ and
 $\widehat{H}$. Starting from an FHT, every component of level $<N$ lies in a child, so the reconstructed family coincides with the original one.
 As the blocks are the maximal vertex sets of components of level $<N$, the recursive decomposition is unique.

 \item \emph{Every edge contains at least two children:} Remark~\ref{rem:nosingleton}, since $m\ge2$.

 \item \emph{Proof for flat FHTs:} For a flat FHT ($N=1$, hence $L=0$) the blocks are the singletons, so the host is the clutter of top edges itself: the realizations of flat FHTs on $V$, $\card V\ge2$, are exactly the connected clutters with vertex set $V$.
\end{enumerate}
\end{proof}

\paragraph{Example.}
Take $V = \{1, 2, 3, 4\}$ and the FHT of depth $N=2$ whose global realization $E$ is:
 \[ E = \{\{1, 2\}, \{1, 2, 3\}, \{1, 2, 4\}\}, \]
and the recursive decomposition $\mathcal{B} = \{B_1, B_2, B_3\}$ where $B_1 = \{1, 2\}$, $B_2 = \{3\}$, $B_3 = \{4\}$ 
with the local FHTs $\mathcal{F}_B$ on each block yielding local realizations $E_{B_1} = \{\{1, 2\}\}$ (depth $d_{B_1} = 1$) 
and $E_{B_2} = E_{B_3} = \emptyset$ (depth $0$), and the host:
 \[ E_{\widehat{H}} = \{\{B_1, B_2\}, \{B_1, B_3\}\}. \]
The maximum block depth is $L = 1$, so $N = L + 1 = 2$.

Starting from the recursive decomposition, we collect the components of levels $\le L$ from each block $\mathcal{F}_B$
and we construct the top-level component at level $L+1 = 2$ by taking the union $\bigcup\hat{e}$ for each edge $\hat{e} \in E_{\widehat{H}}$:
 \begin{align*}
 \bigcup \{B_1, B_2\} &= \{1, 2\} \cup \{3\} = \{1, 2, 3\}, \\
 \bigcup \{B_1, B_3\} &= \{1, 2\} \cup \{4\} = \{1, 2, 4\}.
 \end{align*}
Then, combining the local edge sets $E_{B}$ and the newly constructed top-level edges via the formula:
 \[ E = \bigcup_{B\in\mathcal{B}}E_B \cup \Big\{\textstyle\bigcup\hat e : \hat e\in E_{\widehat{H}}\Big\}, \]
we reconstruct the global realization:
 \[ E = \{\{1, 2\}\} \cup \{\{1, 2, 3\}, \{1, 2, 4\}\} = \{\{1, 2\}, \{1, 2, 3\}, \{1, 2, 4\}\}. \]

\begin{remark}[Single-edge FHTs and Schr\"oder hierarchies]\label{rem:schroder}
An FHT is \emph{single-edge} if every component of level $\ge1$ has exactly one edge.
The realizations of such FHTs on $V$ ($|V|\ge2$) actually are the families
$\mathcal{L}$ of subsets of $V$ such that:
(a) $\mathcal{L}$ is laminar;
(b) every $A\in\mathcal{L}$ has $|A|\ge2$;
(c) $V\in\mathcal{L}$.
Such families are called \emph{hierarchies} on $V$~\cite{johnson_clustering}. Adding
the singletons of $V$, their Hasse diagram is a rooted tree with leaf set $V$ in which
every internal node has at least two children. Up to isomorphism, hierarchies are
counted by Schr\"oder's fourth problem (i.e.\ by the numbers of
series-reduced rooted trees with $v$ leaves)~\cite{schroder_1870,stanley_ec2}:
$1,1,2,5,12,33,90,261,\dots$ (OEIS A000669~\cite{oeis_A000669}).

Thus FHTs contain two classical families: on at least two vertices, the flat FHTs are the connected clutters (or connected Sperner hypergraphs) counted by the A261006 OEIS sequence, and the single-edge FHTs are the Schr\"oder hierarchies counted by the A000669 OEIS sequence. We observe that 
FHTs in both families are those with the single edge $\{V\}$. This gives us a lower bound for the number $f(v)$ of FHTs on $v$ vertices:
\[ f(v)\ge \text{A261006}(v)+\text{A000669}(v)-1. \] 
For $v=3$ this is an equality ($3+2-1=4$). But for
$v=4$ there are five remaining FHTs:
$\{12,13,1234\}$, $\{12,13,23,1234\}$, $\{12,123,124\}$, $\{12,34,123\}$ and $\{12,34,123,124\}$. The third one combines a non-trivial
clutter with a non-trivial hierarchy, and that in the last two the top component has an edge ($34$) which contains no component of level $1$ (Remark~\ref{rem:A5}).
\end{remark}

\begin{proof}
($\Rightarrow$) If every component $H$ of level $\ge1$ has a single edge, this edge is
$V_H$, so $E=\{V_H:\lev(H)\ge1\}$, which is laminar by
Lemma~\ref{lem:laminar}. By Proposition~\ref{prop:recursive}, its elements have at least two vertices, and by axiom (A3) it contains $V$.

($\Leftarrow$) Let $\mathcal{L}$ satisfy (a)--(c), and let $r(A)$ be the rank of
$A\in\mathcal{L}$ defined as follows: $r(A)=1$ if no element of
$\mathcal{L}$ is strictly contained in $A$, and $r(A)=1+\max\{r(A'):A'\in\mathcal{L},\
A'\subsetneq A\}$ otherwise. If $A'\subsetneq A$ then $r(A')<r(A)$. Put $N=r(V)$,
let $\mathcal{H}_0$ be the singletons, and let
$\mathcal{H}_n=\{(A,\{A\}):A\in\mathcal{L},\ r(A)=n\}$ for $1\le n\le N$. We can now check the five FHT axioms.
\begin{itemize}
 \item (A1) holds by construction.
 \item (A2): single-edge hypergraphs are connected and Sperner. Since strict inclusion changes the rank, two distinct elements of
 the same rank are not nested. From (a) it follows that they
 are disjoint.
 \item (A3): every $A\in\mathcal{L}$ with $A\neq V$ satisfies $A\subsetneq V$, hence
 $r(A)<r(V)$. So $V$ is the only element of rank $N$.
 \item (A4): let $A$ have rank $n$, and let $A'$ be the vertex set of a component of
 level $k<n$ with $A\cap A'\neq\emptyset$. If $k=0$, then $A'=\{x\}\subsetneq A$ by
 (b). If $k\ge1$, then $A$ and $A'$ are nested by (a) and, since $r(A')<r(A)$, neither
 $A=A'$ nor $A\subsetneq A'$ is possible, which implies that $A'\subsetneq A$.
 \item (A5): if $r(A)=1$, then $A$ strictly contains a singleton by (b). If $r(A)=n\ge2$,
 the maximum in the definition of $r(A)$ is attained by some $A'\subsetneq A$ of
 rank $n-1$.
\end{itemize}
The realization of this FHT is $\mathcal{L}$. Since, by Lemma~\ref{lem:grading}, an FHT is
determined by its realization, the two constructions are mutually inverse. They are
also compatible with isomorphisms, since isomorphisms of FHTs are exactly the
isomorphisms of their realizations.
\end{proof}

\emph{Hasse diagram.} Let $A\in\mathcal{L}$. The maximal elements of
$\mathcal{L}\cup\{\{x\}:x\in V\}$ strictly contained in $A$ partition $A$: they cover $A$
because they include the singletons, and they are disjoint by laminarity and
maximality. There are at least two of them, since a single set strictly contained in
$A$ cannot cover $A$. The correspondence between hierarchies and series-reduced rooted
trees is then classical~\cite{stanley_ec2}.

\subsection{Enumeration by recurrence}

The Pólya--Burnside method for two-sorted species counts isomorphism classes. The bijection of Proposition~\ref{prop:recursive} and Lemma~\ref{lem:isodec} reduce the isomorphism of FHTs to the colouring of hosts by smaller FHTs. We derive a compact formula and an explicit one, and a species equation for the totals, from the generating functions of the FHTs of depth $L+1$ written with the host cycle index.

\paragraph{Notation.}
For $v\ge1$ and $d\ge0$, let $f(v,d)$ be the number of isomorphism classes of FHTs on
$v$ vertices of depth $d$. Then $f(1,0)=1$, $f(v,0)=0$ for $v\ge2$, and $f(v)=\sum_d f(v,d)$.

We now introduce the two generating functions
\[ F_d(x)=\sum_{v\ge1}f(v,d)\,x^v,\qquad F_{<L}(x)=\sum_{d<L}F_d(x), \]
where $x$ is a formal indeterminate. We immediately have $F_0(x)=x$ and $F_{<0}=0$.

Now let $\lambda \vdash k$ be an integer partition of an integer $k \ge 1$. We can see it as the cycle type of a permutation in the symmetric group $S_k$ acting on $k$ elements (where each part of $\lambda$ corresponds to the length of a cycle). We respectively denote by $\ell(\lambda)$ its number of parts, $|\lambda| = k$ its sum, $r_j(\lambda)$ the number of its parts of length $j$ (representing the number of $j$-cycles), $z_\lambda = \prod_j j^{r_j(\lambda)} r_j(\lambda)!$ the order of its centralizer $\{ \tau \in S_k \mid \tau \sigma = \sigma \tau \}$ in $S_k$, and $p_\lambda = \prod_i p_{\lambda_i}$ the associated power-sum variables, with $z_\emptyset = p_\emptyset = 1$ (convention for $k=0$). Recall that $S_k$ contains exactly $k!/z_\lambda$ permutations of cycle type $\lambda$.

\begin{example} For example let $k = 4$ and $\lambda = (2, 1, 1) \vdash 4$. We have $|\lambda| = 4 = k$, $\ell(\lambda) = 3$, $r_1(\lambda) = 2$ ($1$-cycles), $r_2(\lambda) = 1$ ($2$-cycle), and $r_j(\lambda) = 0$ for $j \ge 3$, $z_\lambda = (1^2 \cdot 2!) \cdot (2^1 \cdot 1!) = 4$ (in $S_4$), $\dfrac{k!}{z_\lambda} = \dfrac{4!}{4} = 6$, and $p_\lambda = p_{\lambda_1} p_{\lambda_2} p_{\lambda_3} = p_2 p_1^2$.
\end{example}

\begin{definition}[Hosts, fixed hosts, host cycle index]\label{def:hosts}
Let $X$ be a finite set with $|X|=m\ge2$. A \emph{host} on $X$ is a connected clutter with vertex set $X$; no other condition is imposed. For a permutation $g$ of $X$, let $\operatorname{Fix}(g)$ denote the number of hosts $\widehat{H}$ with $g(\widehat{H})=\widehat{H}$. Two permutations of the same cycle type $\lambda\vdash m$ are conjugate in $\mathfrak{S}(X)$, so $\operatorname{Fix}(g)$ only depends on $\lambda$ and we write it $\operatorname{Fix}(\lambda)$. The \emph{cycle index of connected clutters} is
\[ Z_{\mathrm{cc}}(p_1,p_2,\dots)=\sum_{m\ge2}\ \sum_{\lambda\vdash m}\frac{\operatorname{Fix}(\lambda)}{z_\lambda}\,p_\lambda . \]
When the points of the host are split into a set $T$ of $t\ge1$ points (standing for the children of maximal depth) and a set $U$ of $u\ge0$ points (standing for the children of smaller depth), and the permutations preserve $T$ and $U$, with cycle types $\mu\vdash t$ and $\nu\vdash u$, we write $\operatorname{Fix}(\mu;\nu):=\operatorname{Fix}(\mu\sqcup\nu)$ and we define the two-sort \emph{host cycle index}
\[ Z(p;q)=\sum_{\substack{t\ge1,\ u\ge0\\ t+u\ge2}}\ \sum_{\mu\vdash t}\ \sum_{\nu\vdash u}
 \frac{\operatorname{Fix}(\mu;\nu)}{z_\mu z_\nu}\,p_\mu\,q_\nu . \]
\end{definition}

\begin{remark}\label{rem:ZZcc}
Let $(p+q)_k:=p_k+q_k$. Then $Z(p;q)=Z_{\mathrm{cc}}(p+q)-Z_{\mathrm{cc}}(q)$. Expanding $\prod_i(p_{\lambda_i}+q_{\lambda_i})$ in $Z_{\mathrm{cc}}(p+q)$, the number of ways to choose which parts of $\lambda$ are $p$-parts (forming $\mu$) is $\prod_j\binom{r_j(\lambda)}{r_j(\mu)}=z_\lambda/(z_\mu z_\nu)$, and the terms with $t=0$ form $Z_{\mathrm{cc}}(q)$. The colouring of the host points by depth plays no role in $\operatorname{Fix}$: in particular no edge of the host has to meet $T$. (Hence the numbers $\operatorname{Fix}$ are the same for all depth profiles; only the weights of the colourings depend on them.)
\end{remark}

\begin{lemma}[Isomorphism of decompositions]\label{lem:isodec}
Let $\F$ and $\F'$ be FHTs on vertex sets $V$ and $V'$ with $|V|,|V'|\ge2$, and let $H_N$ and $H'_{N'}$ be their top components. Then $\F\cong\F'$ if and only if the coloured hosts $\widehat{H}_N^{\tau}$ and $\widehat{H}'_{N'}{}^{\tau'}$ are isomorphic, that is, if there exists a bijection $\pi:\Ch(H_N)\to\Ch(H'_{N'})$ such that $\F_C\cong\F'_{\pi(C)}$ for every $C \in \Ch(H_N)$ and $\pi(\widehat{H}_N)=\widehat{H}'_{N'}$, where $\pi(\hat{e}) = \{\pi(C) : C \in \hat{e}\}$ for each edge $\hat{e}$.
\end{lemma}
\begin{proof}
($\Rightarrow$) As any isomorphism $\sigma$ maps components of level $n$ onto components of
level $n$, it preserves the depth $N$ and maps the maximal vertex sets of components
of level $<N$ onto those of $\F'$. Hence $\pi(C)=\sigma(C)$ is a bijection of
children, and $\sigma|_C$ is an isomorphism $\F_C\to\F'_{\sigma(C)}$. As a top
edge $e=\bigcup\hat e$ is mapped onto the top edge $\sigma(e)=\bigcup\pi(\hat e)$, we also have
$\pi(\widehat{H}_N)=\widehat{H}'_{N'}$.

($\Leftarrow$) Choose isomorphisms $\sigma_C:\F_C\to\F'_{\pi(C)}$ and glue
them into a bijection $\sigma:V\to V'$. By Proposition~\ref{prop:recursive}, the
components of $\F$ are those of the $\F_C$ together with the top
component $(V,\{\bigcup\hat e:\hat e\in E_{\widehat{H}_N}\})$, and similarly for $\F'$.
Since $\sigma(\bigcup\hat e)=\bigcup\pi(\hat e)$, $\sigma$ sends the former onto the latter, level by level.
\end{proof}

The first recursive formula follows from the Pólya--Burnside method for two-sorted species~\cite{bll_species,harary_palmer}.

\begin{theorem}[Recurrence, compact form]\label{thm:recurrence}
For every $L\ge0$,
\[ F_{L+1}(x)=Z\Big(p_k\leftarrow F_L(x^k)\,;\ q_k\leftarrow F_{<L}(x^k)\Big)
 = Z_{\mathrm{cc}}\big[F_{\le L}\big]-Z_{\mathrm{cc}}\big[F_{<L}\big], \]
where $p_k\leftarrow A$ denotes the substitution of the power-sum variable $p_k$ by $A$, and $Z_{\mathrm{cc}}[G]$ stands for $Z_{\mathrm{cc}}(p_k\leftarrow G(x^k))$.

Since every term involves at least two children, the coefficient of $x^v$ on the right
only involves the numbers $f(s,d)$ with $s<v$ and $d\le L$: this is a recurrence. For
$L=0$ it gives $F_1(x)=\sum_{v\ge2}S(v)x^v$, where $S(v)$ is the number of unlabelled
connected clutters (or connected Sperner families) on $v$ vertices (OEIS A261006~\cite{oeis_A261006}, beginning with $S(1)=S(2)=1$). Thus $S(v)=f(v,1)$ for $v\ge2$ only: the term $S(1)=1$ of A261006 counts the one-point complex, which is the trivial FHT, of depth $0$.
\end{theorem}

\begin{proof}
The proof unfolds in two steps. First, we establish a bijection between isomorphism classes of FHTs of depth $L+1$ and the orbits of coloured hosts under the action of the direct product of symmetric groups $G_{t,u}=S_T\times S_U$. Second, we apply the weighted Burnside lemma to sum over all symmetry types.

Fix $L\ge0$. For each $d$, let $\mathcal{T}_d$ be a set of representatives of the
isomorphism classes of FHTs of depth $d$ (of all sizes).

We define a \emph{coloured host} of type
$(t,u)$, $t\ge1$, $t+u\ge2$, as a pair $(\widehat{H},c)$ consisting of a host $\widehat{H}$ on $X=T\sqcup U$, where $T=\{1,\dots,t\}$ and
$U=\{t+1,\dots,t+u\}$, and a map $c$ on $X$
with $c(T)\subseteq\mathcal{T}_L$ and
$c(U)\subseteq\bigcup_{d<L}\mathcal{T}_d$.
Its weight is $x^{\sum_{i\in X}|c(i)|}$, where $|c(i)|$ is the number of vertices of $c(i)$.
The group $G_{t,u}$ acts by
$g\cdot(\widehat{H},c)=(g(\widehat{H}),c\circ g^{-1})$, and this action preserves weights.

The first step shows that the FHT classes of depth $L+1$ on $v$ vertices correspond bijectively to the orbits of coloured hosts of weight $x^v$ (all types $(t,u)$).
Given such an FHT, take its top component $H_N$, its host $\widehat{H}_N$ and its children, of maximal depth
$L$ (Proposition~\ref{prop:recursive_properties}). Let $m=|\Ch(H_N)|$, let $t\ge1$ be the number of children of depth $L$, and $u=m-t$. Choose a
bijection $\beta:\Ch(H_N)\to X$ sending the children of depth $L$ onto $T$, let
$c(\beta(C))$ be the representative of $\F_C$, and transport the host by $\beta$: this is a host on $X$, with no further condition. Two choices
of $\beta$ differ only by an element of $G_{t,u}$. By Lemma~\ref{lem:isodec}, isomorphic FHTs
give the same orbit, since $\pi$ preserves isomorphism classes, hence depths, hence
$T$.
Conversely, two FHTs giving the same orbit are isomorphic by
Lemma~\ref{lem:isodec}. Every coloured host arises, since
Proposition~\ref{prop:recursive}, applied to disjoint copies of the $c(i)$, builds an
FHT of depth $L+1$ (whose top component has $t\ge1$ children of depth $L$). When $L=0$ there is no colour
for $U$, so only $u=0$ occurs.

The second step uses the weighted Burnside lemma. The sum of the weights of the orbits of type $(t,u)$ is
\[ \frac{1}{t!\,u!}\sum_{g\in G_{t,u}}\operatorname{Fix}(g)\,W(g), \qquad
 W(g)=\sum_{c\,:\,c\circ g^{-1}=c}x^{\sum_i|c(i)|}, \]
where $\operatorname{Fix}(g)$ is the number of hosts on $X$ invariant under $g$, which depends only on the cycle type of $g$ on $X$ (Definition~\ref{def:hosts}).
A colouring fixed by $g$ is constant on the cycles of $g$, and a cycle of length $k$
coloured by an FHT with $s$ vertices contributes $x^{ks}$. So
\[ W(g)=\prod_{\gamma\subseteq T}F_L(x^{|\gamma|})\prod_{\gamma\subseteq U}F_{<L}(x^{|\gamma|}), \]
where the products run over the cycles $\gamma$ of $g$.
Grouping the $g$ by cycle types
$(\mu,\nu)$, with $t!u!/(z_\mu z_\nu)$ elements each, and summing over $(t,u)$ yields the first formula of the theorem; the second follows from Remark~\ref{rem:ZZcc}.
\end{proof}

\begin{lemma}[Depth bound]\label{lem:depthbound}
An FHT of depth $d$ has at least $d+1$ vertices. Hence $F_d(x)$ has no term of degree $\le d$, and for each $v$ only the depths $d\le v-1$ contribute to $f(v)$, which is a finite sum.
\end{lemma}
\begin{proof}
By (A5), starting from the top component, there is a chain of components $H_d\supseteq H_{d-1}\supseteq\dots\supseteq H_0$ with $\lev(H_n)=n$. Distinct components have distinct vertex sets (Lemma~\ref{lem:laminar}), so the inclusions are strict and $\card V\ge\card{V_{H_d}}\ge\card{V_{H_0}}+d=d+1$.
\end{proof}

\begin{corollary}[Species equation]\label{cor:species}
Let $F(x)=\sum_{d\ge0}F_d(x)=\sum_{v\ge1}f(v)x^v$. Then
\[ F(x)=x+Z_{\mathrm{cc}}\big(F(x),F(x^2),F(x^3),\dots\big). \]
\end{corollary}
\begin{proof}
By Lemma~\ref{lem:depthbound}, the coefficient of $x^v$ in $F$ is the finite sum $\sum_{d\le v-1}f(v,d)$, so $F$ is a well-defined formal power series. Summing the compact form over $L=0,\dots,M$, the right-hand sides telescope: $\sum_{L=0}^{M}F_{L+1}=Z_{\mathrm{cc}}[F_{\le M}]-Z_{\mathrm{cc}}[F_{<0}]=Z_{\mathrm{cc}}[F_{\le M}]$, since $F_{<0}=0$ and $Z_{\mathrm{cc}}$ has no constant term; with $F_0=x$, this gives $F_{\le M+1}=x+Z_{\mathrm{cc}}[F_{\le M}]$. For $M\ge v-1$, $F_{\le M}$ and $F$ have the same coefficients in degrees $\le v$ (Lemma~\ref{lem:depthbound}), hence so do $Z_{\mathrm{cc}}[F_{\le M}]$ and $Z_{\mathrm{cc}}[F]$, whose coefficient of $x^v$ only involves coefficients of degree $<v$; the same holds for $F_{\le M+1}$ and $F$. The coefficients of $x^v$ of the two sides stabilize, and the identity follows coefficient by coefficient.
\end{proof}

In other words, an FHT is either a point, or a connected clutter whose vertices carry FHTs: the species of FHTs is the solution of $\mathcal{F}=\mathbf{X}+\mathcal{CC}\circ\mathcal{F}$, where $\mathcal{CC}$ is the species of connected clutters with at least two vertices. This is the sense in which FHTs are the iterated composition of clutters.

\begin{corollary}[Recurrence, explicit form]\label{cor:explicit}
For $v\ge2$ and $L\ge0$,
\[ f(v,L+1)=\sum_{(\lambda^{s,d})}\ \Big[\prod_{s,d}w\big(f(s,d),\lambda^{s,d}\big)\Big]\
 H(\alpha;\beta), \]
where:
\begin{itemize}
 \item the sum runs over families of partitions $\lambda^{s,d}$ ($1\le s<v$,
 $0\le d\le L$) with $\sum_{s,d}s\,|\lambda^{s,d}|=v$ and
 $\sum_{s,d}|\lambda^{s,d}|\ge2$, and with $\lambda^{s,L}\neq\emptyset$ for some $s$.
 Here $|\lambda^{s,d}|$ is the sum of the parts of $\lambda^{s,d}$ (the number of children with $s$ vertices and depth $d$),
 and $\lambda^{s,d}$ records how often the chosen classes are repeated;
 \item $w(F,\lambda)=\dfrac{F(F-1)\cdots(F-\ell(\lambda)+1)}{\prod_j r_j(\lambda)!}$ is
 the number of multisets of $|\lambda|$ classes, chosen among $F$ classes, with
 multiplicity partition $\lambda$;
 \item $\alpha=\bigsqcup_s\lambda^{s,L}$ and $\beta=\bigsqcup_s\bigsqcup_{d<L}\lambda^{s,d}$
 are the concatenations of parts;
 \item and $H(\alpha;\beta)$ is the number of orbits of hosts on the $|\alpha|+|\beta|$ points (the children), under the Young subgroup
 $\prod_iS_{\alpha_i}\times\prod_jS_{\beta_j}$ (permutations of children with the same class). By Burnside's lemma, we have
 \[ H(\alpha;\beta)=\sum_{\mu^{i}\vdash\alpha_i}\ \sum_{\nu^{j}\vdash\beta_j}
 \frac{\operatorname{Fix}\big(\bigsqcup_i\mu^{i}\sqcup\bigsqcup_j\nu^{j}\big)}
 {\prod_i z_{\mu^{i}}\prod_j z_{\nu^{j}}}. \]
\end{itemize}
\end{corollary}

\begin{proof}
In the first step of the proof of Theorem~\ref{thm:recurrence}, we group the orbits of coloured
hosts according to the orbit of the colouring $c$. An orbit of colourings corresponds to a
multiset of classes, and the stabilizer of $c$ is the Young subgroup of the permutations
preserving the colours. So the orbits of pairs lying over $c$ correspond to the
orbits of hosts under this stabilizer, whose number is $H(\alpha;\beta)$.
Finally, the number of multisets of classes of type $(s,d)$ with multiplicity
partition $\lambda^{s,d}$ is $w(f(s,d),\lambda^{s,d})$: choose an
injective assignment of classes to the parts and divide by the symmetries of equal parts.
\end{proof}

It follows from this corollary that the explicit computation of the host orbit numbers $H(\alpha;\beta)$ relies on the fixed-host numbers $\operatorname{Fix}(\lambda)$ (Appendix~\ref{sec:fix}).

\begin{example}[The case $v=4$, $L=1$]\label{ex:v4}
We compute $f(4,2)$ by Corollary~\ref{cor:explicit}. The profiles of children with at least one child of depth $L=1$ are the following. Recall $f(2,1)=1$, $f(3,1)=3$ and $f(1,0)=1$.
\begin{itemize}
 \item One child of size $3$ and depth $1$ and one singleton: $w=3\cdot1=3$ choices of the child, and a host on two points, which is unique ($\{B_1,B_2\}$). Contribution $3$.
 \item Two children of size $2$ (depth $1$): $\lambda^{2,1}=(2)$, $w=1$, and a host on two points, unique. Contribution $1$ (the FHT $\{12,34,1234\}$).
 \item One child of size $2$ and two singletons ($\alpha=(1)$, $\beta=(2)$): the hosts are the connected clutters on $\{B,x,y\}$. Of the $5$ labelled ones, $3$ are invariant under the exchange of $x$ and $y$ ($\{Bxy\}$, $\{Bx,By\}$ and $\{Bx,By,xy\}$), so $H=(5+3)/2=4$. With $B=\{1,2\}$, $x=3$, $y=4$ the four orbits give the FHTs $\{12,1234\}$, $\{12,123,124\}$, $\{12,34,123,124\}$ and $\{12,34,123\}$ (the last from $\{Bx,xy\}\sim\{By,xy\}$). Contribution $4$.
\end{itemize}
Altogether $f(4,2)=3+1+4=8$, in agreement with Table~\ref{tab:fht_depth}. The edge $\{x,y\}$ of the host $\{Bx,xy\}$ (giving $\{12,34,123\}$) does not meet $B$: such hosts are allowed by the component-wise axiom (A5).
\end{example}

\subsection{Effective enumeration}\label{sec:effective}

The values below were obtained, and cross-checked, by the methods of Section~\ref{sec:verif_enum}.

\paragraph{Flat structures.}
By Proposition~\ref{prop:recursive_properties}(iv), flat FHTs on $v\ge2$ vertices are the connected
clutters with vertex set $V$. A simplicial complex is determined by its set of facets,
which is a clutter covering its vertices, and the complex is connected if and only if
this clutter is connected. Hence, for $v\ge2$, $S(v)$ is the number of unlabelled connected simplicial
complexes on $v$ nodes, OEIS A261006~\cite{oeis_A261006}.
The exact values for $v=8,9$ are not those printed in A261006. The unlabelled antichains of subsets of a $v$-set, counted by A003182~\cite{oeis_A003182}, are the multisets of connected components, and a one-point component is either an uncovered point or a point carrying the edge $\{x\}$. Let $a(n)$ be the term of A003182 itself, and $b(n)=a(n)-1$ the number obtained after removing the antichain $\{\varnothing\}$. Then
\[ \sum_{n\ge0}b(n)\,x^n=\prod_{k\ge1}(1-x^k)^{-c_k},\qquad c_1=2,\quad c_k=S(k)\ (k\ge2), \]
which gives $S(v)$ by the inverse Euler transform. The terms $a(8)=1392195548889993358$ and $a(9)=789204635842035040527740846300252680$ are exact (B.~Pawelski~\cite{pawelski9}), and they give $S(7)=489980450$, $S(8)=1392195547909966876$, $S(9)=789204635842035037743349748520248426$. The values printed in A261006 for $n=8,9$ ($1392195547909966848$ and $789204635842035012967539870068113408$) are the nearest IEEE double-precision numbers to $S(8)$ and $S(9)$ (they are multiples of $2^{8}$ and $2^{67}$ respectively, the spacing of double-precision numbers at these magnitudes); an earlier version of this work used them as they stand, which shifted $f(8)$ by $28$ and $f(9)$ by about $2.5\cdot10^{19}$. Independently of A003182, the Burnside sum over $\mathfrak{S}_8$ of the table of Appendix~\ref{sec:fix} gives exactly $S(8)=1392195547909966876$.

\paragraph{Results.}
Table~\ref{tab:fht_counts} gives the number $f(v)$ of FHTs on $v$ vertices up to
isomorphism, split into flat structures $S(v)=f(v,1)$ ($v\ge2$; for $v=1$, $S(1)=1$ counts the trivial FHT) and imbricated structures
$I(v)=f(v)-S(v)$. Table~\ref{tab:fht_depth} gives the distribution $f(v,d)$ by depth.

\begin{table}[H]
\centering\small
\resizebox{\textwidth}{!}{%
\begin{tabular}{r r r r}
\toprule
$v$ & $S(v)$ & $I(v)$ & $f(v)$ \\
\midrule
1 & 1 & 0 & 1 \\
2 & 1 & 0 & 1 \\
3 & 3 & 1 & 4 \\
4 & 14 & 9 & 23 \\
5 & 157 & 74 & 231 \\
6 & 15\,942 & 1\,005 & 16\,947 \\
7 & 489\,980\,450 & 102\,908 & 490\,083\,358 \\
8 & 1\,392\,195\,547\,909\,966\,876 & 3\,894\,873\,270 & 1\,392\,195\,551\,804\,840\,146 \\
9 & 789\,204\,635\,842\,035\,037\,743\,349\,748\,520\,248\,426 & 12\,529\,614\,498\,185\,160\,229 & 789\,204\,635\,842\,035\,050\,272\,964\,246\,705\,408\,655 \\
\bottomrule
\end{tabular}}
\caption{Number of FHTs on $v$ vertices up to isomorphism. For $v=1$ the trivial FHT
(depth $0$) is counted in $S(1)$ (convention of A261006). The values $S(8)$ and $S(9)$ are derived from A003182 (see the text).}
\label{tab:fht_counts}
\end{table}

\begin{table}[H]
\centering\small
\resizebox{\textwidth}{!}{%
\begin{tabular}{r r r r r r r r r }
\toprule
$v\backslash d$ & 1 & 2 & 3 & 4 & 5 & 6 & 7 & 8 \\
\midrule
2 & 1 & & & & & & & \\
3 & 3 & 1 & & & & & & \\
4 & 14 & 8 & 1 & & & & & \\
5 & 157 & 60 & 13 & 1 & & & & \\
6 & 15\,942 & 850 & 136 & 18 & 1 & & & \\
7 & 489\,980\,450 & 100\,637 & 2\,010 & 237 & 23 & 1 & & \\
8 & $S(8)$ & 3\,894\,674\,643 & 194\,495 & 3\,740 & 363 & 28 & 1 & \\
9 & $S(9)$ & 12\,529\,614\,490\,884\,393\,704 & 7\,300\,459\,865 & 299\,957 & 6\,155 & 514 & 33 & 1 \\
\bottomrule
\end{tabular}}
\caption{Number $f(v,d)$ of FHTs on $v$ vertices of depth $d$ (with $f(1,0)=1$).}
\label{tab:fht_depth}
\end{table}

\begin{remark}[Growth]\label{rem:growth}
\leavevmode
\begin{enumerate}
 \item On the computed range the flat structures dominate: $I(v)/f(v)$ decreases from $0.39$ at $v=4$ to $1.6\cdot10^{-17}$ at $v=9$. Whether this persists is a conjecture (Section~\ref{sec:perspectives}).
 \item Among imbricated structures, for $v\le9$ the dominant contribution comes from one child of
 $2$ vertices and $v-2$ isolated vertices: its hosts are the connected clutters on $v-1$ points, one of which (the child) is marked, so it contributes $H((1);(v-2))$, roughly $S(v-1)$ times the mean number of vertex orbits. It accounts for 88.9\% of $I(9)$. The configuration made of one child of $v-1$ vertices and one isolated vertex, which was dominant when every host edge had to meet a deepest block, still contributes exactly $f(v-1)$ (a single host).
 \item The computation of $f(10)$ would require $S(10)$, which is not known, and hosts on $9$ points.
\end{enumerate}
\end{remark}

\section{Hierarchical hosts and invariants}\label{sec:invariants}

This section is about what
can be \emph{measured} on an FHT. It does not use the enumeration.

The invariants are the automorphism group, with the orbits and the invariant vertices it defines (Section~\ref{sec:sym}), and shape data of two kinds: the eccentricities and radii of graphs attached to a hypergraph (Sections~\ref{sec:hosts} and~\ref{sec:ecc}), and the Betti numbers of the simplicial complex generated by the edges (Section~\ref{sec:betti}).
What distinguishes an FHT from a flat hypergraph is its hierarchy. Each non-leaf component has its own
\emph{host}, a connected clutter on its children (Definition~\ref{def:hostcomp}), and the invariants can be computed for the FHT as a
whole and for each host. Appendix~\ref{sec:annex} gives both: one line of global parameters for each FHT, then one line for each
non-leaf component.

\paragraph{Conventions.} We use the notation of Sections~\ref{sec:defsec} and~\ref{sec:enumeration}. As in Definition~\ref{def:hostcomp}, the host of a non-leaf component $H$ is the connected clutter $\widehat H$ on $\Ch(H)$; in this section and in Appendices~\ref{sec:annex} and~\ref{sec:annexB} it is designated by its set of edges $\widehat{E}_H=E_{\widehat H}=\{\hat e:e\in E_H\}$, as in the tables. An edge $\{i,j,k\}$ is written $ijk$. The FHTs with at most
five vertices are numbered $\F_1,\dots,\F_{260}$ as in Table~\ref{tab:fht_list} (Appendix~\ref{sec:FHT_list}); $\F_1$ is the FHT on one vertex, which
exists by convention (axiom (A1) at level $0$) and has no edge structure. Groups: $\{e\}$ is the trivial group,
$C_n$ the cyclic group of order $n$, $S_n$ the symmetric group, $V_4=C_2\times C_2$, $D_n$ the dihedral group of order $2n$, $\mathrm{Dic}_n$ the dicyclic group of order $4n$, $A_4$ the alternating group, and $\wr$ is the wreath product.

\subsection{Hierarchical hosts and their graphs}\label{sec:hosts}

\begin{definition}[Hierarchical host]\label{def:host}
Let $H$ be a non-leaf component of an FHT $\F$. Applying Proposition~\ref{prop:recursive} and Lemma~\ref{lem:children} to $\F_H$ we get the set $\Ch(H)$ of the children of $H$, with $m_H=\lvert \Ch(H)\rvert{}\ge2$, and the \emph{host} $\widehat{E}_H$ of $H$: the connected clutter, with vertex set $\Ch(H)$, of the sets $\hat e=\{c\in\Ch(H):c\subseteq e\}$, $e\in E_H$. The host is the hypergraph $\widehat{E}_H$ itself: it depends on $\F$ through $\Ch(H)$, and not only on the hypergraph $H=(V_H,E_H)$. We give to each child $c$ the \emph{colour} $\tau(c)$, the isomorphism class of the FHT $\F_c$
($\tau(c)=\F_1$ if $c$ is a leaf). The \emph{hierarchical host} of $H$ is the coloured host $\widehat{E}_H^{\tau}$. The \emph{host of $\F$} is the host of its root, and the hosts of the other non-leaf components are the \emph{local} hosts.
\end{definition}

A host can itself be regarded as a flat FHT whose vertices are the children: since it is a connected clutter on $\Ch(H)$, it is the realization of a flat FHT with $m_H$ vertices,
its \emph{flat type}, which is one of the $\F_n$ of the dictionary when $m_H\le5$. The host is simple (it is a clutter) and connected, by Lemma~\ref{lem:children}. The depth of $\F_H$ is
$1+\max_c\operatorname{depth}(\F_c)$ (Proposition~\ref{prop:recursive_properties}). The coloured host $\widehat{E}_H^{\tau}$ depends only on the isomorphism class of $\F_H$.

\begin{definition}[Associated graphs]\label{def:graphs}
Let $K=(X,\mathcal{E})$ be a connected hypergraph (for instance a host $\widehat{E}_H$, or the realization of an FHT). We associate to it three graphs.
\begin{enumerate}
 \item The \emph{primal graph} $\PG(K)$ on $X$: $a\sim b$ if $a\ne b$ and some edge contains both.
 \item The \emph{edge graph} $\LG(K)$ on $\mathcal{E}$: $e\sim f$ if $e\ne f$ and $e\cap f\ne\emptyset$.
 \item The \emph{incidence graph} $\IG(K)$, bipartite on $X\sqcup\mathcal{E}$: $a\sim e$ if $a\in e$.
\end{enumerate}
\end{definition}

\begin{remark}[Nerve]\label{rem:nerve}
Let $\Dl(K)$ be the simplicial complex whose faces are the subsets of edges (for an auxiliary hypergraph $(X,\mathcal E)$, the singletons of $X$ are also faces, so that every point of $X$ is a vertex of $\Dl$), and $\NV(K)$ its nerve, the simplicial complex on $\mathcal{E}$
whose faces are the sets of edges with a common point. The primal graph is the $1$-skeleton of $\Dl(K)$ (an edge has at least two points) and the
edge graph is the $1$-skeleton of $\NV(K)$. The incidence graph is the graph of the relation $a\in e$ between points and edges. The two complexes built from this relation, $\Dl(K)$ on the points and $\NV(K)$ on the edges, have isomorphic homology (Dowker's theorem~\cite{dowker}; for an auxiliary hypergraph this holds for the subcomplex of $\Dl(K)$ on the covered support $\bigcup\mathcal E$, an uncovered point being an extra isolated vertex of $\Dl(K)$ and not a vertex of $\NV(K)$), and the incidence graph is the graph on which both are read. We call it the \emph{nerve graph} when a single name for the three graphs is needed.
\end{remark}

\paragraph{The trivial FHT.} The realization of $\F_1$ on $\{x\}$ is the auxiliary hypergraph $(\{x\},\varnothing)$. By the convention above, $\Dl(\F_1)$ is the one-point complex, so that $\beta(\F_1)=(1)$. Its primal graph and its incidence graph are the graph with one vertex and no edge, and its edge graph is the empty graph. The metric parameters of $\F_1$ are not used; they are written -- in the tables.

\begin{lemma}\label{lem:conn}
If $K$ is connected then $\PG(K)$, $\LG(K)$ and $\IG(K)$ are connected. The realization of an FHT on at least two vertices is connected.
\end{lemma}
\begin{proof}
Two points of a connected hypergraph are linked by a chain of intersecting edges, which gives a path in $\PG$, in $\IG$, and in $\LG$ between the first and last
edges of the chain. For the second statement, by Proposition~\ref{prop:recursive} every vertex lies in a top edge, the top edges form a connected clutter on the children of the root, and a top edge contains whole children.
\end{proof}

For an FHT $\F$ we write $\PG(\F)$, $\LG(\F)$, $\IG(\F)$ for the graphs of its realization, and $\PG(\widehat{E}_H)$, $\LG(\widehat{E}_H)$, $\IG(\widehat{E}_H)$ for those of the host of a component $H$.

\begin{proposition}[Distances in the incidence graph]\label{prop:dI}
Let $K=(X,\mathcal{E})$ be connected, $a,b\in X$, $e,f\in\mathcal{E}$, and $d(a,e)=\min_{y\in e}d_{\PG}(a,y)$. Then
\[ d_{\IG}(a,b)=2\,d_{\PG}(a,b),\qquad d_{\IG}(e,f)=2\,d_{\LG}(e,f),\qquad d_{\IG}(a,e)=2\,d(a,e)+1. \]
\end{proposition}
\begin{proof}
A path in $\IG$ alternates between points and edges. A path $a=x_0,e_1,x_1,\dots,e_\ell,x_\ell=b$ of length $2\ell$ has $x_{i-1},x_i\in e_i$, so $x_0,\dots,x_\ell$ is a walk of length $\ell$
in $\PG$ (steps with $x_{i-1}=x_i$ are shortened); conversely a path of $\PG$ of length $\ell$ lifts to a path of $\IG$ of length $2\ell$ by choosing for each step an edge containing both ends.
The statement for two edges is the same argument with the roles exchanged. A path from $a$ to $e$ has odd length $2\ell+1$ and has the form $a=x_0,e_1,\dots,x_\ell,e_{\ell+1}=e$, with $x_\ell\in e$, so it exists if and only if some $y\in e$ satisfies $d_{\PG}(a,y)\le\ell$.
\end{proof}

\begin{proposition}[The primal graph of an FHT]\label{prop:blowup}
Let $\F$ have $\lvert V\rvert{}\ge2$, with root $H_N$, children $\Ch(H_N)$ and root host $\widehat{E}_{H_N}$. Then $\PG(\F)$ is obtained from $\PG(\widehat{E}_{H_N})$ by replacing every point $C$ by a clique
on $C$, two vertices of different children being adjacent if and only if their children are adjacent in $\PG(\widehat{E}_{H_N})$. In particular, writing $C_x$ for the child containing $x$,
\[ d_{\PG(\F)}(x,y)=\begin{cases}1&\text{if }x\ne y\text{ lie in the same child},\\ d_{\PG(\widehat{E}_{H_N})}(C_x,C_y)&\text{otherwise,}\end{cases}\qquad
\ecc_{\PG(\F)}(x)=\ecc_{\PG(\widehat{E}_{H_N})}(C_x). \]
\end{proposition}
\begin{proof}
Two vertices of a child $C$ lie in every top edge $\bigcup\hat e$ with $C\in\hat e$, and there is one since $\widehat{E}_{H_N}$ covers $\Ch(H_N)$. Two vertices of different children $C,C'$ lie in a common edge of $\F$ if and only if they lie in a common top edge
(the other edges are inside a child), if and only if $C,C'$ lie in a common edge of $\widehat{E}_{H_N}$. A shortest path of $\PG(\F)$ projects onto a walk of $\PG(\widehat{E}_{H_N})$, and a path of $\PG(\widehat{E}_{H_N})$ lifts, since all vertices of a child are adjacent to all vertices of the adjacent children. The eccentricity follows because $m\ge2$ implies $\ecc_{\PG(\widehat{E}_{H_N})}\ge1$.
\end{proof}

\subsection{Invariants}\label{sec:inv}

\subsubsection{Automorphisms and invariant vertices}\label{sec:sym}

\begin{definition}[Automorphism group, orbits, invariant vertices]\label{def:aut}
Let $\F$ be an FHT on $V$.
\begin{enumerate}
 \item An \emph{automorphism} of $\F$ is an isomorphism $\F\to\F$. By Lemma~\ref{lem:isoreal} the automorphisms form the group $\Aut(\F)=\{\sigma\in\mathfrak{S}(V):\sigma(E)=E\}$ of the realization.
 \item The \emph{orbit} of $x\in V$ is $\Orb(x)=\{\sigma(x):\sigma\in\Aut(\F)\}$. The orbits partition $V$; their number is $k(\F)$ and the multiset of their sizes is the \emph{orbit-length tuple}, written in decreasing order.
 \item A vertex $x$ is \emph{invariant} if $\sigma(x)=x$ for all $\sigma\in\Aut(\F)$. The set of invariant vertices is $\Inv(\F)$.
\end{enumerate}
\end{definition}

\begin{proposition}\label{prop:autbasic}
\leavevmode
\begin{enumerate}
 \item Every automorphism maps the components of level $n$ onto the components of level $n$, hence induces an automorphism of the component tree, and maps the host of a component $H$ onto the host of $\sigma(H)$.
 \item $\lvert \Orb(x)\rvert{}=\lvert \Aut(\F)\rvert{}/\lvert \operatorname{Stab}(x)\rvert{}$, and $\Inv(\F)$ is the union of the orbits of size $1$. Hence $\lvert \Inv(\F)\rvert{}\le k(\F)$, with equality if and only if $\Aut(\F)$ is trivial.
 \item $k(\F)=\frac1{\lvert \Aut(\F)\rvert{}}\sum_{\sigma\in\Aut(\F)}\lvert \{x:\sigma(x)=x\}\rvert{}$ (Burnside).
\end{enumerate}
\end{proposition}
\begin{proof}
(1) Isomorphisms preserve levels and inclusion, hence the tree and the hosts. (2) is the orbit--stabilizer theorem; since $\Aut(\F)\subseteq\mathfrak{S}(V)$ acts faithfully, a non-trivial automorphism moves a vertex, whose orbit has size $\ge2$, so $\lvert \Inv\rvert{}<k$; for the trivial group $\lvert \Inv\rvert{}=k=\lvert V\rvert{}$. (3) is Burnside's lemma.
\end{proof}

\begin{definition}[Host group]\label{def:gamma}
For a non-leaf component $H$, let $\Gamma(H)=\Aut(\widehat{E}_H^{\tau})$ be the group of the automorphisms of the coloured host, that is, of the permutations $\pi$ of $\Ch(H)$ such that $\tau\circ\pi=\tau$ and $\pi(\widehat{E}_H)=\widehat{E}_H$.
We write $\Gamma(\F)=\Gamma(\text{root})$.
\end{definition}

\begin{theorem}[Recursive structure of the automorphism group]\label{thm:autrec}
Let $\lvert V\rvert{}\ge2$, with root $H_N$, children $\Ch(H_N)$, sub-FHTs $\F_C$ and root host $\widehat{E}_{H_N}$ (Proposition~\ref{prop:recursive}). The map $\rho:\sigma\mapsto(C\mapsto\sigma(C))$ is a surjective homomorphism $\Aut(\F)\to\Gamma(\F)$ and there is an exact sequence
\[ 1\longrightarrow\prod_{C\in\Ch(H_N)}\Aut(\F_C)\longrightarrow\Aut(\F)\xrightarrow{\ \rho\ }\Gamma(\F)\longrightarrow1, \]
with $\Aut(\F_C)=\{e\}$ for a leaf child. So $\lvert \Aut(\F)\rvert{}=\lvert \Gamma(\F)\rvert{}\prod_{C}\lvert \Aut(\F_C)\rvert{}$.
\end{theorem}
\begin{proof}
$\Gamma(\F)$ contains the identity and is closed under composition and inverses. Let $\sigma\in\Aut(\F)$. By the proof of Lemma~\ref{lem:isodec} (direction $\Rightarrow$, with $\F'=\F$), $\sigma$ permutes the children, $\sigma|_C:\F_C\to\F_{\sigma(C)}$ is an isomorphism and $\rho(\sigma)$ preserves $\widehat{E}_{H_N}$. So $\rho(\sigma)\in\Gamma(\F)$, and $\rho$ is a homomorphism.
It is surjective: for $\pi\in\Gamma(\F)$, choose isomorphisms $\F_C\to\F_{\pi(C)}$ and glue them as in the proof of Lemma~\ref{lem:isodec} (direction $\Leftarrow$). The kernel consists of the $\sigma$ with $\sigma(C)=C$ for all $C$, that is, of the families $(\sigma|_C)_C\in\prod\Aut(\F_C)$, and any such family glues to an automorphism by Lemma~\ref{lem:isodec} with $\pi=\mathrm{id}$.
\end{proof}

\begin{corollary}[Order of the group]\label{cor:order}
$\lvert \Aut(\F)\rvert{}=\prod_H\lvert \Gamma(H)\rvert{}$, where $H$ runs over all the non-leaf components of $\F$ (each component counted once, even if several are isomorphic).
\end{corollary}
\begin{proof}
Induction on $\lvert V\rvert{}$ with Theorem~\ref{thm:autrec}: the non-leaf components of $\F$ are its root and the non-leaf components of the $\F_C$, $C\in\Ch(H_N)$.
\end{proof}

\begin{remark}[Flat FHTs and wreath products]\label{rem:wreath}
If $\F$ is flat, all children are leaves and $\Aut(\F)=\Gamma(\F)$ is the automorphism group of the clutter $(V,E)$. If the host is a single edge and the children have isomorphism types $\GGroup_1,\dots,\GGroup_r$ with multiplicities $m_1,\dots,m_r$, then $\Gamma(\F)=\prod_iS_{m_i}$ and, fixing isomorphisms between children of the same type, $\Aut(\F)\cong\prod_i\Aut(\GGroup_i)\wr S_{m_i}$. For instance $\Aut(\{12,34,1234\})\cong C_2\wr S_2\cong D_4$.
For a general host, $\Gamma(\F)$ is the stabilizer of $\widehat{E}_{H_N}$ in the group of the colour-preserving permutations of the children of the root. The FHTs $\{12,1234\}$ and $\{12,123,124\}$ both have $\Gamma\cong C_2$ (the exchange of $3$ and $4$) and $\Aut\cong V_4$, although their hosts differ.
\end{remark}

\subsubsection{Eccentricities and radii}\label{sec:ecc}

Let $G$ be a connected graph. The \emph{eccentricity} of a vertex is $\ecc(x)=\max_yd(x,y)$. We use
\begin{align*}
 \rad(G)&=\min_x\ecc(x) &&\text{(radius)},\\
 \dia(G)&=\max_x\ecc(x) &&\text{(maximal eccentricity, i.e.\ the diameter)},\\
 \mec(G)&=\frac1{\lvert V(G)\rvert{}}\sum_x\ecc(x) &&\text{(mean eccentricity)}.
\end{align*}
We always have $\rad\le\mec\le\dia\le2\rad$, and the mean eccentricity is a rational number, which the annex gives in exact form. The triple $(\rad,\dia,\mec)$ is invariant under graph isomorphism; so by Lemma~\ref{lem:isoreal}, the triples of $\PG(\F)$, $\LG(\F)$, $\IG(\F)$ are invariants of the FHT $\F$, and the triples of the graphs of a host are invariants of the corresponding component. A one-vertex graph has the triple $(0,0,0)$, which happens for $\LG$ when the host has a single edge.

Not all of these numbers are independent, and we state what is proved.

\begin{proposition}[Redundancies]\label{prop:redund}
\leavevmode
\begin{enumerate}
 \item (\emph{Primal graph of an FHT versus its host}.) For $\lvert V\rvert{}\ge2$, with $\widehat{E}_{H_N}$ the root host,
\[
\rad(\PG(\F))=\rad(\PG(\widehat{E}_{H_N})),\qquad \dia(\PG(\F))=\dia(\PG(\widehat{E}_{H_N})),
\]
\[
\mec(\PG(\F))=\frac1{\lvert V\rvert{}}\sum_{C\in\Ch(H_N)}\lvert C\rvert{}\,\ecc_{\PG(\widehat{E}_{H_N})}(C).
\]
 \item (\emph{Incidence graph versus primal and edge graphs.}) For any connected hypergraph $K$,
 \[ 2\max\{\dia(\PG),\dia(\LG)\}\ \le\ \dia(\IG)\ \le\ 2\max\{\dia(\PG),\dia(\LG)\}+1. \]
 \item (\emph{Flat FHTs.}) If $\F$ is flat, then $\widehat{E}_{H_N}=(V,E)$ and the global triples coincide with the triples of the host.
\end{enumerate}
\end{proposition}
\begin{proof}
(1) is Proposition~\ref{prop:blowup}. (2) By Proposition~\ref{prop:dI}, the distances in $\IG$ between two points (resp.\ two edges) are twice those of $\PG$ (resp.\ $\LG$), which gives the lower bound; the distance between a point and an edge is $2d(a,e)+1\le2\dia(\PG)+1$ since $d(a,e)\le d_{\PG}(a,y)$ for any $y\in e$, which gives the upper bound. (3) is clear.
\end{proof}

The metric summaries of the incidence graph are not determined by the corresponding metric summaries of the primal and edge graphs. Among the $259$ root hosts of the FHTs with $2\le v\le5$ vertices (one for each FHT, counted with repetition; the local hosts are not counted), there are pairs with the same $(\rad,\dia)$ for $\PG$ and for $\LG$ and different $\rad(\IG)$, and pairs with the same triples for $\PG$ and $\LG$ and different mean eccentricity of $\IG$. For $\dia(\IG)$ both values allowed by Proposition~\ref{prop:redund}(2) occur (in $174$ and $85$ of the $259$ cases). We keep the three graphs in the annex; the primal graph of the whole FHT adds, with respect to the root host, only the weighted mean of Proposition~\ref{prop:redund}(1).

\subsubsection{Betti numbers}\label{sec:betti}

\begin{definition}\label{def:betti}
The \emph{complex} of $\F$ is the simplicial complex $\Dl(\F)$ on $V$ whose faces are the subsets of edges of the realization. The \emph{Betti numbers} of $\F$ are $\beta_i(\F)=\dim H_i(\Dl(\F);\mathbb{F}_2)$, $i\ge0$.
\end{definition}

Since $\lvert V\rvert{}\ge2$ and $V=\bigcup E$, every vertex is a face. The lower edges lie inside top edges, so $\Dl(\F)$ is generated by the top edges alone. The dimension of $\Dl(\F)$ is $\max_e\lvert e\rvert{}-1$, so $\beta_i=0$ for $i\ge\max_e\lvert e\rvert{}$; and $\beta_0=1$ since the realization is connected (Lemma~\ref{lem:conn}). We list $(\beta_0,\dots,\beta_d)$ up to the last non-zero Betti number.

\begin{proposition}[Betti numbers are those of the host]\label{prop:bettihost}
Let $\lvert V\rvert{}\ge2$. The projection $\pi:\Dl(\F)\to\Dl(\widehat{E}_{H_N})$, which sends each vertex to the child of the root that contains it, is a homotopy equivalence. In particular $\beta_i(\F)=\beta_i(\Dl(\widehat{E}_{H_N}))$ for all $i$.
\end{proposition}
\begin{proof}
Every edge of $\Rea(\F)$ lies in a top edge $e=\bigcup\hat e$ (Lemma~\ref{lem:read}), so every face of $\Dl(\F)$ lies in a top edge. The map $\pi$ is simplicial: a face contained in $e$ is sent into $\hat e$, a face of $\Dl(\widehat{E}_{H_N})$. Choose a vertex $r(C)\in C$ for each child $C$ of the root. The map $s:C\mapsto r(C)$ is simplicial, since a face $\sigma\subseteq\hat e$ is sent into $\bigcup\hat e=e$, and $\pi\circ s$ is the identity. For a face $\tau\subseteq e$ of $\Dl(\F)$, the set $\tau\cup s(\pi(\tau))$ is contained in $e$, hence is a face: $s\circ\pi$ is contiguous to the identity, hence homotopic to it. So $\pi$ is a homotopy equivalence with homotopy inverse $s$. The topology of $\F$ forgets the interior of the children of the root.
\end{proof}

\begin{remark}\label{rem:betti_vanishing}
The Betti numbers do not stop at the depth: the boundary $\{123,124,134,234\}$ of the tetrahedron is a flat FHT ($N=1$) with $\beta_2=1$. Over $\mathbb{F}_2$ and over $\mathbb{Q}$ the Betti numbers of the $260$ complexes with at most five vertices coincide (checked by computing both ranks); for more vertices, torsion may occur (the smallest complex with $2$-torsion has six vertices), and the Betti numbers over $\mathbb{F}_2$ and $\mathbb{Q}$ may differ.
\end{remark}

\subsubsection{Parameters on the children and on the host edges}\label{sec:childedge}

The host $\widehat{E}_H$ of a component $H$ has two kinds of elements, the \emph{children} (its points) and the \emph{host edges}, and each kind carries its own parameters. We write $\lvert c\rvert{}$ for the number of vertices of the child $c$ (so $\lvert c\rvert{}=1$ for a leaf).
\begin{itemize}
 \item \emph{Children.} The \emph{degree} of a child $c$ is $\deg(c)=\lvert \{\hat e\ni c\}\rvert{}$ and its \emph{incidence sizes} are the sizes $\lvert \hat e\rvert{}$ of the host edges containing it, as a multiset $\operatorname{sz}(c)$. The \emph{child profile} of the host is the multiset of the pairs $(\tau(c),\operatorname{sz}(c))$, that is, of the type of each child together with the sizes of the edges around it.
 \item \emph{Host edges.} The \emph{size} of $\hat e$ is $\lvert \hat e\rvert{}\ge2$ (number of children) and its \emph{weight} is $w(\hat e)=\sum_{c\in\hat e}\lvert c\rvert{}$, the number of vertices of the corresponding edge $\bigcup\hat e$ of $\F_H$. The \emph{edge-size profile} and the \emph{weight profile} are the multisets of the sizes and of the weights.
 \item \emph{Sizes of the graphs.} $\lvert E(\PG)\rvert{}$ is the number of pairs of children lying in a common host edge, $\lvert E(\LG)\rvert{}$ the number of pairs of intersecting host edges, and $\lvert E(\IG)\rvert{}=\sum_{\hat e}\lvert \hat e\rvert{}$.
 \item \emph{Orbits on the edges.} $\Gamma(H)$ acts on the host edges, and $\Aut(\F)$ acts on the edges of the realization (Proposition~\ref{prop:autbasic}(1)). The \emph{edge-orbit tuple} is the multiset of the orbit sizes, in decreasing order.
\end{itemize}
The first three kinds of parameters are related by the double counting $\sum_c\deg(c)=\sum_{\hat e}\lvert \hat e\rvert{}=\lvert E(\IG)\rvert{}$ and by $\lvert E(\PG)\rvert{}\le\sum_{\hat e}\binom{\lvert \hat e\rvert{}}{2}$, with equality when two host edges share at most one child. The weight profile is redundant with the size profile for a flat FHT, where $\lvert c\rvert{}=1$, and is then omitted in the annex. The edge orbits of $\Aut(\F)$ are a global parameter of $\F$ and complete the vertex orbits of Definition~\ref{def:aut}; they are given in the last column of Annex~\ref{sec:annexB}.

\subsection{The dictionary of invariants}\label{sec:dictionary}

Appendix~\ref{sec:annex} gives the invariants of the $260$ FHTs with at most five vertices, numbered as in Table~\ref{tab:fht_list}, and Appendix~\ref{sec:annexB} gives the parameters on the children and on the host edges of Section~\ref{sec:childedge}. The realizations are not repeated.
For each FHT $\F_n$ there is first a line, in bold, of \emph{global parameters}:
\begin{itemize}
 \item the group $\Aut(\F)$, its order, the number $k$ of vertex orbits, the orbit-length tuple and $\lvert \Inv(\F)\rvert{}$;
 \item the Betti numbers $(\beta_0,\beta_1,\dots)$;
 \item for each of $\PG(\F)$, $\LG(\F)$, $\IG(\F)$ of the realization, the triple $(\rad,\dia,\mec)$.
\end{itemize}
The leaves have no parameter. If $\F$ is not flat, the following lines give the \emph{parameters of the hosts}: first the line \textit{root} of the root host, then one line for each type of non-leaf proper component (child, grandchild, \dots), with its type $\F_k$, its distance $d$ to the root in the upper index when $d>1$ and the multiplicity ${\times}\mu$ when several components of the same type occur at the same distance to the root. For a host line the columns mean:
the group $\Gamma$ of Definition~\ref{def:gamma} and its order; the number and lengths of the orbits of $\Gamma$ on the children; the number of children fixed by $\Gamma$; in the column of the Betti numbers, the triple $m\,|\,\lvert \widehat{E}_H\rvert{}\,|\,\F_f$ (number of children, number of host edges, flat type of the host); and the triples of the three graphs of the host. For a flat FHT the root host is the realization (Proposition~\ref{prop:redund}(3)) and its line is omitted. The line of a local host of type $\F_k$ repeats the root host line of $\F_k$; it is kept so that each FHT can be read without the others.

\subsubsection{How much do these invariants distinguish?}

We compare the FHTs with at most five vertices by signatures of increasing length. A signature is a tuple of numbers and of isomorphism classes of abstract groups; it contains no isomorphism type of an FHT or of a host. Groups are compared up to isomorphism.
\begin{enumerate}
\item[(i)] $(\lvert V\rvert{},N,\lvert E\rvert{})$;
\item[(ii)] (i), the isomorphism class of $\Aut(\F)$, $\lvert \Aut\rvert{}$, $k$, the orbit-length tuple and $\lvert \Inv\rvert{}$;
\item[(iii)] (ii) and the Betti numbers;
\item[(iv)] (iii) and the three global triples $(\rad,\dia,\mec)$;
\item[(v)] (iv) and the host lines of Appendix~\ref{sec:annex} with all types removed (the type $\F_k$ of the component and the flat type $\F_f$ of its host): the line of the root host, and the multiset, with multiplicities, of the lines of the other non-leaf components, their distance to the root not being recorded. A host line consists of the isomorphism class of $\Gamma$, $\lvert\Gamma\rvert$, the number and lengths of its orbits on the children, the number of fixed children, $m$, $\lvert \widehat{E}_H\rvert{}$ and the three triples of the host graphs. For a flat FHT the root host line repeats the global data;
\item[(vi)] (v) and the edge-orbit tuple of $\Aut(\F)$ (Appendix~\ref{sec:annexB});
\item[(vii)] (vi) and the multiset of the sizes of the edges of the root host (Appendix~\ref{sec:annexB}).
\end{enumerate}
The number of distinct values taken on the FHTs is:

\begin{center}
\begin{tabular}{@{}lrrrrrrrr@{}}
\toprule
 & number of FHTs & (i) & (ii) & (iii) & (iv) & (v) & (vi) & (vii)\\
\midrule
$\lvert V\rvert=3$ & 4 & 4 & 4 & 4 & 4 & 4 & 4 & 4 \\
$\lvert V\rvert=4$ & 23 & 10 & 21 & 21 & 23 & 23 & 23 & 23 \\
$\lvert V\rvert=5$ & 231 & 20 & 79 & 144 & 224 & 226 & 229 & 231 \\
\midrule
$\lvert V\rvert\le5$ & 260 & 36 & 106 & 171 & 253 & 255 & 258 & 260 \\
\bottomrule
\end{tabular}
\end{center}

Each added family separates more FHTs. The global eccentricity data raise the number of classes at five vertices from $144$ to $224$. After (v), five pairs are not separated: four pairs of flat FHTs, $(\F_{75},\F_{77})$, $(\F_{82},\F_{102})$, $(\F_{96},\F_{108})$, $(\F_{106},\F_{174})$, and the pair $\F_{232}=\{12,123,345,1245\}$, $\F_{234}=\{12,34,1235,1245\}$. The last two have the same hosts up to the position of the only non-trivial child: the host of both is a connected clutter on four points with edges of sizes $2,3,3$, and the child of type $\F_2$ lies in the edge of size $2$ for $\F_{232}$ and not for $\F_{234}$. This feature of the coloured host is recorded neither by the host lines nor by the global line of Appendix~\ref{sec:annex}. It is recorded by the child profile of Appendix~\ref{sec:annexB} (the child of type $\F_2$ has incidence sizes $\{2,3\}$ in $\F_{232}$ and $\{3,3\}$ in $\F_{234}$), and by the edge orbits of the automorphism group (the tuples $(1,1,1,1)$ and $(2,1,1)$). The edge orbits (vi) also separate $\F_{75}$ from $\F_{77}$ and $\F_{96}$ from $\F_{108}$. Two pairs of flat FHTs remain: $\F_{82}=\{12,13,24,145\}$ and $\F_{102}=\{12,13,234,245\}$, and $\F_{106}=\{12,13,2345\}$ and $\F_{174}=\{123,124,345\}$. They have the same symmetry, homology, metric data and edge orbits, and differ by the sizes of their edges ($2,2,2,3$ against $2,2,3,3$, and $2,2,4$ against $3,3,3$). With (vii) the $260$ FHTs are pairwise separated.

If the flat types $\F_f$ are kept in the host lines, as printed in Appendix~\ref{sec:annex}, (v) gives $230$ and $259$ classes and (vi) separates all FHTs. For a flat FHT, however, the flat type of its root host is its own isomorphism type, so this variant is not a separation by numerical invariants. (The counts are recomputed from the tables of Appendices~\ref{sec:FHT_list}--\ref{sec:annexB} by a program of the repository.)

\section{Elementary morphisms and the morphic groupoid}\label{sec:morphisms}

A Fractal Hyper-Tree (FHT) $\F$ on a finite set $V$ is a hierarchy of connected clutters (Section~\ref{sec:defsec}): each component $H$ carries a connected clutter $\widehat H$ on its children, called its host. The enumeration of Section~\ref{sec:enumeration} counts these objects up to isomorphism,
and Section~\ref{sec:invariants} measures them. This section studies how one FHT is changed into another one on the same vertices by elementary local changes.

The realization suggests an intrinsic neighbourhood, the addition or deletion of one edge (Section~\ref{sec:intrinsic}); the resulting graph is already disconnected at three vertices. We use four \emph{interior morphisms}, which keep the number $v=\card{V}$ of vertices, and one \emph{exterior morphism} $\Omega_1$, which changes $v$ by one (Section~\ref{sec:omega}):
\begin{itemize}
\item $\Th_1^{+}$ (\emph{fusion}) and $\Th_1^{-}$ (\emph{fission}) reorganize the edges of the host $\widehat H$ of one component $H$, without changing its children;
\item $\Th_2^{+}$ (\emph{aggregation}) regroups children of $H$ into a new child, and $\Th_2^{-}$ (\emph{explosion}) releases the children of a child of $H$ into $H$;
\item $\Omega_1^{+}$ and $\Omega_1^{-}$ add a leaf as a new child of $H$ (twin of a block of children), or remove a leaf from a block with at least two members.
\end{itemize}
Each of them is \emph{located} at a component $H$: $\Th_1^{\pm}$ and $\Th_2^{+}$ act on $H$ with their data (an aggregation creates a new child of $H$), $\Th_2^{-}$ acts on the pair $(H,c)$ formed by $H$ and one of its non-leaf children, and $\Omega_1^{\pm}$ act on $H$ and a leaf; the starred versions $\Th_i^{\pm *}$ are located at $H$ or at one of its descendants.
The word \emph{morphism} is used here in the sense of the programme: an elementary transformation of FHTs. It should not be confused with an \emph{arrow} of a category.

\paragraph{Plan of the section.} After the intrinsic graph of edge toggles (Section~\ref{sec:intrinsic}), the section has three levels. \emph{Elementary moves and their data} (Sections~\ref{sec:blocks}--\ref{sec:functor}): blocks, the four interior morphisms, their validity, inverses, effects, locality and commutation, their reading on the realization (only some moves add or delete a single edge), the exterior morphism $\Omega_1$, and naturality. \emph{Graphs and order} (Sections~\ref{sec:transgraphs}--\ref{sec:connect}): the transition graphs, the complexity, the descent to the star at distance at most one from the root, the distance to the star, the order, and connectivity. \emph{Groupoids} (Section~\ref{sec:groupoid}): the free groupoid, which keeps all the paths, the thin groupoid, which keeps one arrow between two objects, and the action of the symmetric group, which adds the symmetries.

\subsection{The intrinsic graph of edge toggles}\label{sec:intrinsic}

An FHT is determined by its realization (Lemma~\ref{lem:grading}). The neighbourhood intrinsic to this description is the addition or the deletion of one edge.

\begin{definition}[Toggle graph]\label{def:toggle}
The \emph{toggle graph} $\TG_{\mathrm{tog}}(V)$ has the FHTs on $V$ as vertices, and an edge between $\F$ and $\F'$ when $\Rea(\F)$ and $\Rea(\F')$ differ by exactly one set.
\end{definition}

\begin{example}\label{ex:v3}
For $v=3$ there are $8$ labelled FHTs: the star $\{123\}$, the three nested FHTs $\{ij,123\}$, the three paths $\{ij,jk\}$ and the triangle $\{12,13,23\}$. The star is adjacent to the three nested FHTs and the triangle to the three paths, and there are no other adjacencies (for instance $\{12,23,123\}$ is not an FHT: the root would have a single child). So the toggle graph has two components. They are joined by the replacement of the two edges $12$, $23$ by the edge $123$, which is not a toggle.
\end{example}

The toggle graph is not enough. We add the fusions and fissions of the edges of a host, which replace at most two edges by one, or one edge by at most two (Section~\ref{sec:moves}). This extension is a choice, motivated by connectivity. The interior morphisms contain the toggles exactly as the fusions without absorption, the fissions without fragment, the aggregations and the explosions (Corollary~\ref{cor:cube}). The distance formula $d(\F,\St_V)=\kappa(\F)-1$ (Corollary~\ref{cor:diststar}) is a theorem about this dynamics.

\paragraph{Boundary conventions.} Throughout this section $V$ is a finite \emph{non-empty} set: no FHT on the empty set is defined (Definition~\ref{def:fht}); in the language of species, $\FHT(\varnothing)=\varnothing$ (Corollary~\ref{cor:species}). For $\card V\ge2$, the \emph{star} $\St_V$ is the FHT whose root has all the vertices as children and whose host has the single edge $V$; for $\card V=1$ we set $\St_V=\F_1$, the only FHT on $V$, which has no non-leaf component. The \emph{distance to the root} of a component $K$ is $\rho(K)=\card{\Anc(K)}$ (Section~\ref{sec:setup}); the word \emph{depth} is reserved for the number $N$ of levels of an FHT (Definition~\ref{def:fht}). A component is identified with its vertex set (Section~\ref{sec:structural}).

\subsection{Twin classes and twin blocks}\label{sec:blocks}

\begin{definition}[Twin blocks]\label{def:blocks}
Let $H$ be a non-leaf component. Two children $x,y\in\Ch(H)$ are \emph{twins} if they lie in exactly the same edges of $\widehat H$. The classes of this equivalence relation form the partition $\Blo(H)$ of $\Ch(H)$ into \emph{twin blocks}, called \emph{blocks} in the proofs of this section when no confusion is possible. Two twin blocks are \emph{adjacent} if some edge of $\widehat H$ contains both.
\end{definition}

No condition on the types of the children is imposed: two children are twins as soon as they lie in the same edges of $\widehat H$, whatever the sub-FHTs they carry. (A twin block is not a block of a recursive decomposition, Definition~\ref{def:recursive decomposition}, which are the children of the top component.)

\begin{lemma}[Properties of twin blocks]\label{lem:blocks}
Let $H$ be non-leaf.
\begin{enumerate}
\item Every edge of $\widehat H$ is a union of twin blocks.
\item $\Blo(H)$ has one twin block if and only if $\widehat H$ has a single edge, namely $\Ch(H)$.
\item Let $B$ be a twin block. Every edge of $\widehat H$ meeting $B$ contains $B$.
\end{enumerate}
\end{lemma}

\begin{proof}
(i) and (iii) follow from the definition. (ii) If there is a single block, all children lie in the same edges; since they are covered, every edge contains all children. The converse is clear.
\end{proof}

\subsection{The four interior morphisms}\label{sec:moves}

In this section $\F$ is an FHT on $V$, $H$ is a non-leaf component, $X=\Ch(H)$ and $\mathcal E=E_{\widehat H}$. For each morphism we state what it \emph{acts on}, its \emph{choices}, its \emph{conditions}, its \emph{action}, and what it \emph{leaves unchanged}.

\begin{definition}[Move]\label{def:move}
A \emph{move} is a quadruple $(\F;\Th;K;d)$ where $\F$ is an FHT on $V$, $\Th$ is one of $\Th_1^{\pm}$, $\Th_2^{\pm}$, $K$ is a non-leaf component of $\F$, at which the move is \emph{located}, and $d$ is a datum (a \emph{choice}) admissible for $\Th$ at $K$ in the sense of Definitions~\ref{def:th1p}--\ref{def:th2p} below. Its result is an FHT $\Th(\F;K,d)$ on $V$ (Theorem~\ref{thm:valid}). For $\Th_1^{\pm}$ the datum is a new edge, or an edge and its fragments. An \emph{explosion} $\Th_2^{-}$ is applied to the pair $(K,c)$, where $c$ is a non-leaf child of $K$ present before the move, with the datum $d=(c,\varepsilon)$; an \emph{aggregation} $\Th_2^{+}$ is applied to $K$ with the datum $d=(B,S,D,\varepsilon)$, and it creates a new child $c$ of $K$. In all cases the move is located at $K$, and its distance to the root is $\rho(K)$.
\end{definition}

\subsubsection{Fusion and fission of edges: $\Th_1^{\pm}$}\label{sec:th1}

\begin{definition}[Fusion $\Th_1^{+}$]\label{def:th1p}
\leavevmode
\begin{description}
\item[Acts on] a non-leaf component $H$, through its host $\widehat H$.
\item[Choice] a set $e\subseteq\Ch(H)$ of children (the new edge).
\item[Conditions] (a) $e$ is contained in no edge of $\widehat H$ (in particular $e\notin\mathcal E$ and $\card e\ge2$, since every child lies in some edge); (b) the set of edges \emph{absorbed} by $e$, $\Abs(e)=\{f\in\mathcal E: f\subseteq e\}$, has at most two elements.
\item[Action] $\mathcal E\mapsto(\mathcal E\setminus\Abs(e))\cup\{e\}$: the new edge $e$ replaces the edges it absorbs.
\item[Unchanged] $\Ch(H)$, the hierarchy $\mathcal T$, the levels, and the hosts of all other components.
\end{description}
\end{definition}

\begin{definition}[Fission $\Th_1^{-}$]\label{def:th1m}
\leavevmode
\begin{description}
\item[Acts on] a non-leaf component $H$, through its host $\widehat H$.
\item[Choices] an edge $e\in\mathcal E$, and a set $B$ of at most two subsets of $\Ch(H)$ (the \emph{fragments}).
\item[Conditions] every $b\in B$ is a proper subset of $e$, and $(\mathcal E\setminus\{e\})\cup B$ is an admissible host on $\Ch(H)$.
\item[Action] $\mathcal E\mapsto(\mathcal E\setminus\{e\})\cup B$.
\item[Unchanged] $\Ch(H)$, the hierarchy, the levels, the hosts of all other components.
\end{description}
\end{definition}

The three values of $\card{\Abs(e)}\in\{0,1,2\}$ for a fusion and of $\card B\in\{0,1,2\}$ for a fission give the elementary rearrangements of a host: addition ($\Abs(e)=\varnothing$), growth ($\card{\Abs(e)}=1$) and merging ($\card{\Abs(e)}=2$) of edges for $\Th_1^{+}$; deletion ($B=\varnothing$), shrinking ($\card B=1$) and splitting ($\card B=2$) of an edge for $\Th_1^{-}$. A fragment is never a singleton, since the result must be an admissible host (Remark~\ref{rem:nosingleton}). Within the family of rules that bound $\card{\Abs(e)}$ and $\card B$, the bound $2$ is the smallest one that makes the spaces of admissible hosts connected uniformly in $X$; it is not minimal among all conceivable dynamics. For $\card X=2$ the space $\AH(X)$ is a single point and any bound will do; but for every $X$ with $\card X\ge3$ and the bound $1$, the host $\{X\}$ is isolated: no edge can be added, since it would be contained in $X$, no edge can grow, and a single fragment $b\subsetneq X$ does not cover $X$. More generally, the edge toggles of Corollary~\ref{cor:cube} do not connect the FHTs already for $v=3$ (Example~\ref{ex:v3}).

\begin{example}\label{ex:th1}
Let $H$ be the root of the flat FHT on $\{1,2,3\}$ with host $\{123\}$. Then $\Th_1^{-}$ with $e=123$ and $B=\{12,23\}$ gives the host $\{12,23\}$; conversely $\Th_1^{+}$ on $\{12,23\}$ with $e=123$ absorbs both edges and gives $\{123\}$. On $\{12,23\}$, the choice $e=13$ with $\Abs=\emptyset$ adds an edge and gives the triangle $\{12,13,23\}$.
\end{example}

\subsubsection{Explosion of a child: $\Th_2^{-}$}\label{sec:th2m}

Let $c$ be a non-leaf child of $H$, with children $Y=\Ch(c)$ and host $\widehat c$ with edge set $\mathcal F=E_{\widehat c}$.

\begin{definition}[Pieces]\label{def:pieces}
For $\varepsilon\in\mathcal F$, the \emph{pieces} of $(\widehat c,\varepsilon)$ are the classes of the equivalence relation on $Y$ generated by ``$y,y'$ lie in a common edge of $\mathcal F\setminus\{\varepsilon\}$''. We write $\Pie(\widehat c,\varepsilon)$ for the set of pieces. A piece is \emph{trivial} if it is a single point (a child of $c$ lying in no edge other than $\varepsilon$), and \emph{non-trivial} otherwise. For a non-trivial piece $P$, $\widehat c|_P=(P,\{f\in\mathcal F\setminus\{\varepsilon\}:f\subseteq P\})$ is an admissible host on $P$.
\end{definition}

The pieces are the connected components of the hypergraph $(Y,\mathcal F\setminus\{\varepsilon\})$. $\mathcal F\setminus\{\varepsilon\}$ is an admissible host on $Y$ if and only if there is exactly one piece.

\begin{definition}[Explosion $\Th_2^{-}$]\label{def:th2m}
\leavevmode
\begin{description}
\item[Acts on] a non-leaf component $H$ and a \emph{non-leaf} child $c$ of $H$ (the pair $(H,c)$).
\item[Choice] an edge $\varepsilon$ of the host $\widehat c$ (even when $\widehat c$ has a single edge, in which case it is forced).
\item[Condition] $\widehat c$ minus $\varepsilon$ has at least two pieces. (If it has exactly one piece, the deletion of $\varepsilon$ is a move $\Th_1^{-}$ at $c$, with $B=\emptyset$.)
\item[Action] The child $c$ disappears and is replaced in $\Ch(H)$ by its pieces: for each piece $P$, let $\bP$ be the child $y$ if $P=\{y\}$, and the new component with children $P$ and host $\widehat c|_P$ otherwise. Thus $\Ch(H)$ becomes $(\Ch(H)\setminus\{c\})\cup\{\bP:P\in\Pie(\widehat c,\varepsilon)\}$, and the edges of $\widehat H$ containing $c$ now contain \emph{all} the $\bP$:
\[
f\ni c\ \longmapsto\ (f\setminus\{c\})\cup\{\bP: P\in\Pie(\widehat c,\varepsilon)\}.
\]
The grandchildren (the children of $c$) rise to children of $H$ when their piece is trivial, and become children of the new components otherwise; their own sub-FHTs are unchanged.
\item[Unchanged] the edges of $H$ as subsets of $V$, all other hosts, and the sub-FHT $\F_y$ of every child $y$ of $c$. The only edge of the realization that disappears is $\bigcup\varepsilon$ (Theorem~\ref{thm:deletion}).
\end{description}
\end{definition}

The pieces are twins in the new host $\widehat H$: they lie in exactly the edges that contained $c$.

\begin{example}\label{ex:th2m}
Let $\F=\{12,34,123,12345\}$ on $\{1,\dots,5\}$: the root has children $c=\{1,2,3,4\}$ and $5$; $c$ has children $A=\{1,2\}$, $3$, $4$ and host $\{34,A3\}$. For $\varepsilon=\{3,4\}$ the pieces are $\{A,3\}$ and $\{4\}$: the result has root children $[123]$, $4$, $5$, with $[123]=(\{1,2,3\},\{123\})$ over $A,3$, and the realization $\{12,123,12345\}$. For $\varepsilon=\{A,3\}$ the pieces are $\{3,4\}$ and $\{A\}$ and the result is $\{12,34,12345\}$.
\end{example}

\subsubsection{Aggregation of children: $\Th_2^{+}$}\label{sec:th2p}

\begin{definition}[Aggregation $\Th_2^{+}$]\label{def:th2p}
\leavevmode
\begin{description}
\item[Acts on] a non-leaf component $H$, together with a set $S$ of its children.
\item[Choices] (1) a twin block $B\in\Blo(H)$; (2) a set $S\subseteq B$ of children to be aggregated, with $2\le\card S\le\card{\Ch(H)}-1$; (3) a subset $D\subseteq S$ of \emph{non-leaf} members, the members that are \emph{dissolved} (their children are released into the new child); (4) a set $\varepsilon$ of \emph{points}, where the set of points is
\[
\Pts(S,D)=\bigsqcup_{s\in D}\Ch(s)\ \sqcup\ (S\setminus D).
\]
\item[Conditions] (a) $\varepsilon\subseteq\Pts(S,D)$ meets every member of $S$: $\varepsilon\cap\Ch(s)\ne\emptyset$ for $s\in D$, and $s\in\varepsilon$ for $s\in S\setminus D$; (b) $\varepsilon$ contains no whole edge of a dissolved member: for $s\in D$ and $f\in E_{\widehat s}$, $f\not\subseteq\varepsilon$.
\item[Action] The members of $S$ are replaced in $\Ch(H)$ by a single new child $c$ with $\Ch(c)=\Pts(S,D)$ and host
\[
E_{\widehat c}=\{\varepsilon\}\ \cup\ \bigcup_{s\in D}E_{\widehat s}.
\]
In every edge of $\widehat H$ containing the members of $S$ (all of them, by Lemma~\ref{lem:blocks}), $S$ is replaced by $c$. The dissolved members $s\in D$ disappear as components; their children become children of $c$.
\item[Unchanged] the edges of $H$ as subsets of $V$, all other hosts, the sub-FHTs of the children of the dissolved members and of the non-dissolved members of $S$. The only edge of the realization that appears is $\bigcup\varepsilon$ (Theorem~\ref{thm:addition}).
\end{description}
\end{definition}

The block is chosen because the aggregated children are replaced by one child in \emph{every} edge, so they must lie in the same edges. The upper bound $\card S\le\card{\Ch(H)}-1$ keeps at least two children in $H$; when $\Blo(H)$ has at least two blocks it holds automatically.

\begin{example}[Canonical aggregation]\label{ex:canonical}
If $D=\emptyset$, then $\Pts=S$ and condition (a) forces $\varepsilon=S$: the new child is $c=(\bigcup S,\{\bigcup S\})$, the flat component over $S$ with a single edge.
\end{example}

\begin{example}[Aggregation with dissolution]\label{ex:agg}
Let $\F=\{12,34,12345\}$: the root has children $A=[12]$, $B=[34]$, $5$ and a single edge. The block is the whole set of children; take $S=\{A,B\}$ (so $\card S=2\le 3-1$), $D=\{B\}$. Then $\Pts=\{A,3,4\}$, and $\varepsilon=\{A,3\}$ meets $A$ (non-dissolved member) and $\Ch(B)=\{3,4\}$ (at $3$) and contains no whole edge of $\widehat B=\{34\}$. The result has the new child $c$ with children $A,3,4$ and host $\{34,A3\}$, root children $c,5$, and realization $\{12,34,123,12345\}$: the depth goes from $2$ to $3$. By Example~\ref{ex:th2m}, $\Th_2^{-}$ with $\varepsilon=\{A,3\}$ gives $\F$ back.
\end{example}

\subsubsection{Descending versions}\label{sec:star}

\begin{definition}[Starred morphisms]\label{def:star}
Let $H$ be a component. The starred morphism $\Th^{*}$ (for $\Th\in\{\Th_1^{\pm},\Th_2^{\pm}\}$) \emph{acts on $H$ or on one of its descendants}: its choices are those of $\Th$ plus the choice of a non-leaf component $K\in\{H\}\cup\Desc(H)$ at which $\Th$ is located. For $\Th_1^{\pm}$ and $\Th_2^{+}$ the move acts on $K$ (an aggregation creating a new child of $K$); for $\Th_2^{-}$ it acts on a pair $(K,c)$ where $c$ is a non-leaf child of $K$. The action is that of $\Th$ at $K$. Thus $\Th^{*}$ with $H$ the root can reach every non-leaf component of $\F$.
\end{definition}

Locality does not imply commutation (Proposition~\ref{prop:locality}). Moves located at components that are pairwise incomparable in the hierarchy (that is, pairwise disjoint) can be performed independently and commute; moves $\Th_1$ located at distinct components commute even when these components are nested; more generally, moves with pairwise disjoint footprints (Definition~\ref{def:footprint}) commute. For general moves at nested components, a composition requires an order of application and the successive verification that each datum is still admissible (Example~\ref{ex:noncommute}). Descending actions below the first generation are not generators of new moves, since $\Th$ at $K$ is just $\Th$ applied to the FHT $\F_K$ and substituted back (Proposition~\ref{prop:locality}).

\subsubsection{Summary}

\begin{table}[ht]
\centering
\small
\begin{tabular}{@{}p{1.6cm}p{2.9cm}p{3.1cm}p{3.4cm}p{3.1cm}@{}}
\toprule
 & acts on & choices & conditions & changes \\
\midrule
$\Th_1^{+}$ & $H$ & new edge $e\subseteq\Ch(H)$ & $e$ in no edge; $\card{\Abs(e)}\le2$ & host edges (absorbed $\to e$) \\
$\Th_1^{-}$ & $H$ & edge $e$, fragments $B$, $\card B\le2$ & $B$ proper subsets; result admissible & host edges ($e\to B$) \\
$\Th_2^{-}$ & $H$ and non-leaf child $c$ & edge $\varepsilon$ of $\widehat c$ & $\ge2$ pieces & $c\to$ its pieces; $\Ch(H)$ \\
$\Th_2^{+}$ & $H$ & block $B$, $S\subseteq B$, $D\subseteq S$, $\varepsilon$ & $2\le\card S<\card{\Ch(H)}$; (a), (b) & $S\to$ new child $c$; $\Ch(H)$ \\
\bottomrule
\end{tabular}
\caption{The four interior morphisms, each located at $H$. The starred versions apply them at a chosen $K\in\{H\}\cup\Desc(H)$.}
\end{table}

\subsection{Inverses and effects of the moves}\label{sec:basic}

\begin{theorem}[Validity]\label{thm:valid}
Every move $\Th_1^{\pm},\Th_2^{\pm}$ (and every starred move) applied to an FHT on $V$ with an admissible choice produces an FHT on $V$.
\end{theorem}

\begin{proof}
By the structure lemma (Lemma~\ref{lem:structure}), it suffices to check that the new pair $(\mathcal T',(\mathcal E'_K))$ is a decomposition. The starred versions are the same moves applied at another component and need no further proof.

\emph{$\Th_1^{-}$.} The hierarchy is unchanged and the new host is admissible by hypothesis.

\emph{$\Th_1^{+}$.} Let $\mathcal E'=(\mathcal E\setminus\Abs(e))\cup\{e\}$. It is a clutter: $e$ is contained in no edge of $\mathcal E$ hence in no edge of $\mathcal E\setminus\Abs(e)$ and $e\notin\mathcal E$, and no edge of $\mathcal E\setminus\Abs(e)$ is contained in $e$ by definition of $\Abs(e)$. It covers $X$: a vertex lying only in absorbed edges lies in $e$. It is connected: given $x,y\in X$, take a chain of edges $f_1,\dots,f_k$ of $\mathcal E$ with $x\in f_1$, $y\in f_k$ and $f_i\cap f_{i+1}\ne\emptyset$, and replace each $f_i\in\Abs(e)$ by $e\supseteq f_i$: consecutive sets still meet, and the ends still contain $x$ and $y$ (if $x\in f_1\in\Abs(e)$ then $x\in e$). So $\mathcal E'\in\AH(X)$.

\emph{$\Th_2^{-}$.} Notation as in Definition~\ref{def:th2m}. Let $\mathcal T'=(\mathcal T\setminus\{c\})\cup\{\bigcup P:P\text{ non-trivial piece}\}$ (here $\bigcup P=\bigcup_{y\in P}y$ is the union of the vertex sets). The new sets are unions of members, contained in $c$ and pairwise disjoint, so $\mathcal T'$ is a hierarchy. The children are as announced: $\Ch'(H)=(\Ch(H)\setminus\{c\})\cup\{\bP\}$ has at least $\card{\Ch(H)}-1+2\ge2$ elements; the new component with children $P$ has $\card P\ge2$ children; the children of all other members are unchanged. The host of a non-trivial piece is admissible: $P$ is a connected component of $(Y,\mathcal F\setminus\{\varepsilon\})$ with at least two points, hence covered, connected and a clutter. It remains to check that the new host $\mathcal E'$ of $H$, obtained by replacing $c$ in each edge $f\ni c$ by the set $\Pi$ of all the $\bP$, is admissible. Let $\phi(f)=f$ if $c\notin f$, $\phi(f)=(f\setminus\{c\})\cup\Pi$ otherwise. For $f,g\in\mathcal E$ one has $f\subseteq g\iff\phi(f)\subseteq\phi(g)$ (if $c\in f$, then $\phi(f)\subseteq\phi(g)$ forces $\Pi\subseteq\phi(g)$, i.e. $c\in g$), so $\mathcal E'$ is a clutter. It covers $\Ch'(H)$: the points of $\Ch(H)\setminus\{c\}$ as before, and $\Pi$ by the edges containing $c$, of which there is at least one. Connectedness follows by replacing $c$ in a chain of edges by $\Pi$, since all edges of the chain that contain $c$ contain all of $\Pi$, and every element of $\Pi$ belongs to one such edge.

\emph{$\Th_2^{+}$.} Let $\mathcal F_0=\bigcup_{s\in D}E_{\widehat s}$ and $\mathcal F=\mathcal F_0\cup\{\varepsilon\}$ on $\Pts$. Edges of $\widehat s$ and $\widehat{s'}$ for $s\ne s'$ lie in disjoint sets $\Ch(s),\Ch(s')$ and so are incomparable. An edge $f\in E_{\widehat s}$ is not contained in $\varepsilon$ by (b); $\varepsilon$ is not contained in $f\subseteq\Ch(s)$ because by (a) $\varepsilon$ meets at least two members of $S$, hence meets at least two of the disjoint sets $\Ch(s)$ or $\{s\}$. So $\mathcal F$ is a clutter. It covers $\Pts$: $\Ch(s)$ is covered by $E_{\widehat s}$, and $s\in S\setminus D$ lies in $\varepsilon$ by (a). It is connected: each $E_{\widehat s}$ is connected and $\varepsilon$ meets each of the parts. So $\mathcal F\in\AH(\Pts)$ ($\card{\Pts}\ge\card S\ge2$). The hierarchy $\mathcal T'=(\mathcal T\setminus\{s:s\in D\})\cup\{\bigcup S\}$ is a hierarchy, with $\Ch'(c)=\Pts$ and $\Ch'(H)=(\Ch(H)\setminus S)\cup\{c\}$ of cardinality $\card{\Ch(H)}-\card S+1\ge2$. Let $\psi(f)=(f\setminus S)\cup\{c\}$ if $f\supseteq S$ and $\psi(f)=f$ otherwise; as $S$ lies in a block, every edge meeting $S$ contains $S$ (Lemma~\ref{lem:blocks}), so $f\subseteq g\iff\psi(f)\subseteq\psi(g)$, and the same arguments as above show that $\psi(\mathcal E)$ is a clutter that covers $\Ch'(H)$ and is connected (an edge containing $S$ becomes an edge containing $c$ and the other vertices unchanged). Its edges have at least two elements by Remark~\ref{rem:nosingleton}.
\end{proof}

\begin{proposition}[Inverses]\label{prop:inverse}
\leavevmode
\begin{enumerate}
\item If $\F'$ is the result of $\Th_1^{+}$ with the edge $e$ from $\F$, then $\F$ is the result of $\Th_1^{-}$ from $\F'$ with the edge $e$ and the fragments $B=\Abs(e)$; and conversely every $\Th_1^{-}$ is inverted this way. In particular $\Th_1^{+}$ and $\Th_1^{-}$ are inverse of each other.
\item If $\F'$ is the result of $\Th_2^{-}$ from $\F$ with data $(H,c,\varepsilon)$, then $\F$ is the result of $\Th_2^{+}$ from $\F'$ with $S=\{\bP\}$ the set of the images of the pieces, $D=\{\bP:P\text{ non-trivial}\}$, and the same set $\varepsilon$. Conversely, if $\F'$ is the result of $\Th_2^{+}$ from $\F$ with data $(B,S,D,\varepsilon)$, then $\F$ is the result of $\Th_2^{-}$ from $\F'$ with $(H,c,\varepsilon)$ where $c$ is the new child. the whole datum of the move, $(H,B,S,D,\varepsilon)$ for $\Th_2^{+}$ and $(H,c,\varepsilon)$ for $\Th_2^{-}$, is determined by the pair $(\F,\F')$.
\end{enumerate}
\end{proposition}

\begin{proof}
(i) Let $\mathcal E'=(\mathcal E\setminus A)\cup\{e\}$, $A=\Abs(e)$. The edge $e$ is a proper superset of each $a\in A$, so $A$ consists of proper subsets of $e$; $|A|\le2$; and $(\mathcal E'\setminus\{e\})\cup A=\mathcal E$ is admissible. So $\Th_1^{-}$ with $(e,A)$ is defined on $\mathcal E'$ and returns $\mathcal E$. Conversely, let $\mathcal E''=(\mathcal E\setminus\{e\})\cup B$ be the result of $\Th_1^{-}$, with $e\in\mathcal E$, $B$ proper subsets of $e$. Then $e$ is contained in no edge of $\mathcal E''$ (edges of $\mathcal E\setminus\{e\}$ do not contain $e$ since $\mathcal E$ is a clutter, and elements of $B$ are proper subsets); the edges of $\mathcal E''$ contained in $e$ are exactly the elements of $B$ (an edge of $\mathcal E\setminus\{e\}$ contained in $e$ would contradict that $\mathcal E$ is a clutter). So $\Th_1^{+}$ with $e$ is defined on $\mathcal E''$, absorbs exactly $B$, and returns $\mathcal E$.

(ii) In this part a subscript $\F$ or $\F'$ indicates the FHT in which the hierarchy $\mathcal T$, the children, the parents, the blocks and the hosts are read; $\widehat K_{\F}$ is the host of $K$ in $\F$. Suppose $\F'=\Th_2^{-}(\F;H,c,\varepsilon)$. In $\F'$ the set $S=\{\bP:P\in\Pie(\widehat c_{\F},\varepsilon)\}$ is contained in a block of $\widehat H_{\F'}$ (the $\bP$ lie in exactly the edges that contained $c$), $\card S\ge2$ and $\card S\le\card{\Ch_{\F'}(H)}-1$ because $\Ch_{\F'}(H)$ contains $\Ch_{\F}(H)\setminus\{c\}\ne\emptyset$. The set $D$ of non-trivial pieces' images consists of non-leaf children of $H$ in $\F'$, with $\Pts(S,D)=\bigsqcup_{P}P=Y=\Ch_{\F}(c)$. The set $\varepsilon$ meets every piece: otherwise a piece not meeting $\varepsilon$ would be a union of connected components of $\widehat c_{\F}$ itself, contradicting that $\widehat c_{\F}$ is connected and has at least two pieces; a trivial piece $\{y\}$ lies in $\varepsilon$ since $y$ is covered by $E_{\widehat c_{\F}}$ and by no other edge. Finally $\varepsilon$ contains no whole edge of $\widehat{\bP}_{\F'}$ since these edges belong to $E_{\widehat c_{\F}}\setminus\{\varepsilon\}$ and $E_{\widehat c_{\F}}$ is a clutter. So $\Th_2^{+}$ is defined on $\F'$ with this data, and it produces the child with children $Y$ and host $\{\varepsilon\}\cup\bigcup_{P\text{ non-trivial}}E_{\widehat c_{\F}|_P}=E_{\widehat c_{\F}}$, and puts back $c$ in the edges of $\widehat H_{\F'}$: this is $\widehat H_{\F}$ and the same sub-FHTs. Hence $\F$ is recovered.

Conversely, suppose $\F'=\Th_2^{+}(\F;B,S,D,\varepsilon)$, at the component $H$. In $\F'$ the host of the new child $c$ has the edge set $E_{\widehat c_{\F'}}=\{\varepsilon\}\cup\bigcup_{s\in D}E_{\widehat s_{\F}}$. Deleting $\varepsilon$ leaves $\bigcup_s E_{\widehat s_{\F}}$, whose pieces are the sets $\Ch_{\F}(s)$ ($s\in D$, each connected and covered by the admissible host $\widehat s_{\F}$ and not linked to the others) and the singletons $\{s\}$ ($s\in S\setminus D$, covered by $\varepsilon$ only). There are $\card S\ge2$ pieces, so $\Th_2^{-}(\F';H,c,\varepsilon)$ is defined; for $s\in D$ the piece $\Ch_{\F}(s)$ gives a component with children $\Ch_{\F}(s)$ and host $\widehat s_{\F}$, that is, $s$ again; for $s\in S\setminus D$ the piece gives $s$. Replacing $c$ by $S$ in the edges of $\widehat H_{\F'}$ gives $\widehat H_{\F}$. So $\F$ is recovered.

Finally, the datum is read on the pair. Keep $\F'=\Th_2^{+}(\F;B,S,D,\varepsilon)$ at $H$. The hierarchies satisfy $\mathcal T_{\F'}=(\mathcal T_{\F}\setminus D)\cup\{\bigcup S\}$, and $\bigcup S\notin\mathcal T_{\F}$: it is strictly contained in $H$ since $\card S\le\card{\Ch_{\F}(H)}-1$, and it strictly contains every member of $S$ since $\card S\ge2$, so the members of $S$ would not be maximal in $H$ (Theorem~\ref{thm:tree}). Hence the new child is given by $\mathcal T_{\F'}\setminus\mathcal T_{\F}=\{c\}$, and then
\[
H=\Par_{\F'}(c),\qquad S=\Ch_{\F}(H)\setminus\Ch_{\F'}(H),\qquad D=S\setminus\Ch_{\F'}(c):
\]
the first difference gives the former children of $H$ that were aggregated, the second distinguishes those that were dissolved from those that remain children of $c$. The block $B$ is the block of $\Blo_{\F}(H)$ that contains $S$. The set $\varepsilon$ is the unique edge of $\widehat c_{\F'}$ that is not an edge of $\widehat s_{\F}$ for some $s\in D$: it meets every member of $S$, hence at least two of the disjoint sets $\Ch_{\F}(s)$ ($s\in D$) and $\{s\}$ ($s\in S\setminus D$), and so it is contained in none of them. If $\F'=\Th_2^{-}(\F;H,c,\varepsilon)$, then $\F=\Th_2^{+}(\F';\cdot)$ with the same $H$ and $\varepsilon$, $c$ being the new child, so $(H,c,\varepsilon)$ is read in the same way with the roles of $\F$ and $\F'$ exchanged.
\end{proof}

\begin{proposition}[Effects of the moves]\label{prop:effects}
Let $\F'$ be the result of a move from $\F$, both on $V$.
\begin{enumerate}
\item The number of vertices is unchanged, $\card V=v$.
\item $\Th_1^{\pm}$ do not change the hierarchy $\mathcal T$, the children, the levels and the depth. In particular $\Th_1^{\pm *}$ preserve the hierarchy of $\F$.
\item $\Th_2^{-}$ never increases the level of a component common to $\F$ and $\F'$, nor the depth; $\Th_2^{+}$ never decreases them. Both can change the depth, strictly (Example~\ref{ex:agg}).
\item $\Th_2^{\pm}$ preserve the \emph{reduced host} of the component $H$ at which they are located (the other hosts that they produce, the host of the new child $c$ for an aggregation and the hosts of the non-trivial pieces for an explosion, are new): the hypergraph whose vertices are the twin blocks and whose edges are the sets of twin blocks contained in the edges of $\widehat H$ is unchanged, in particular the number of twin blocks, the number of edges of $\widehat H$ and the adjacency of twin blocks. Only the sizes of twin blocks change: for $\Th_2^{+}$ the twin block $B\supseteq S$ of $\Ch(H)$ is replaced by $(B\setminus S)\cup\{c\}$, of cardinality $\card B-\card S+1$; for $\Th_2^{-}$ the twin block containing $c$ is replaced by $(B\setminus\{c\})\cup\{\bP\}$.
\end{enumerate}
\end{proposition}

\begin{proof}
(i) is clear and (ii) follows from the definition of $\Th_1$ (only $\mathcal E_H$ changes). (iii) Consider $\Th_2^{+}$. The new child $c$ has level $1+\max\{\lev(y):y\in\Pts\}$; for $s\in D$, $\lev(s)=1+\max_{y\in\Ch(s)}\lev(y)\le\lev(c)$ and for $s\in S\setminus D$, $\lev(s)<\lev(c)$. So $\lev(c)\ge\lev(s)$ for all $s\in S$. The components not on the path from $c$ to the root keep their level; going up from $c$, each level is $1+\max$ over the children (Corollary~\ref{cor:levels}), so it does not decrease. Thus the depth does not decrease; $\Th_2^{-}$ is the inverse (Proposition~\ref{prop:inverse}), so it does not increase it. (iv) For $\Th_2^{+}$ with $S\subseteq B$: let $x,y\in\Ch(H)\setminus S$; they are twins in $\widehat H'$ iff in $\widehat H$ (the edges containing them are the images $\psi(f)$); an element $x\in B\setminus S$ is a twin of $c$ in $\widehat H'$ since $c$ lies in the images of the edges containing $S$, which contain $x$; and no element outside $B$ is a twin of $c$. The edges of the reduced host are the same sets of blocks. The case of $\Th_2^{-}$ is the inverse.
\end{proof}

\begin{definition}[Footprint]\label{def:footprint}
The \emph{footprint} of a move is the set of the components whose children or host the move reads or modifies, together with the components that it deletes or creates:
\begin{itemize}
\item for $\Th_1^{\pm}$ located at $K$: the component $K$;
\item for an explosion $\Th_2^{-}$ on $(K,c)$: $K$, the child $c$ (deleted), and the new components carried by the non-trivial pieces;
\item for an aggregation $\Th_2^{+}$ located at $K$ with the datum $(B,S,D,\varepsilon)$: $K$, the dissolved members $s\in D$ (deleted), and the new child $c$.
\end{itemize}
The data actually read are: for $\Th_1^{\pm}$, the children and the host of $K$; for $\Th_2^{-}$, the host of $K$, the children of $c$ and the host of $c$; for $\Th_2^{+}$, the host of $K$ and its twin blocks, the members of $S$, and the children and the hosts of the dissolved members.
\end{definition}

\begin{proposition}[Locality and commutation]\label{prop:locality}
Let $K$ be a non-leaf component of $\F$.
\begin{enumerate}
\item (\emph{Locality.}) A move located at $K$ only modifies the coordinates of the structure lemma (Lemma~\ref{lem:structure}) attached to components contained in $K$: the members of the hierarchy that are not contained in $K$ and their uncoloured hosts are unchanged (the colour of the child of an ancestor that contains $K$, an isomorphism type, may change), and the restriction of $\Th(\F;K,d)$ to $K$ is $\Th(\F_K;K,d)$. The levels of the ancestors of $K$, the depth, and global invariants such as $\kappa(\F)$ (Section~\ref{sec:kappa}) or $\Aut(\F)$ may change (Proposition~\ref{prop:effects}(iii)): they are derived from the hierarchy and the hosts, and are recomputed.
\item (\emph{Footprints.}) Moves of $\F$ with pairwise disjoint footprints can be performed in any order with the same data, and the result does not depend on the order.
\item In particular, moves located at pairwise incomparable (that is, pairwise disjoint) components can be performed independently and commute, and two moves $\Th_1$ located at distinct components commute, even when these components are nested.
\end{enumerate}
\end{proposition}

\begin{proof}
(i) By Definition~\ref{def:footprint}, the move only reads the children and hosts of members of its footprint, and it only deletes or creates members of its footprint and changes the children and the hosts of members of its footprint; all of them are contained in $K$. The other coordinates are unchanged by definition, and the levels, which are determined by the hierarchy, are recomputed (Proposition~\ref{prop:effects}(iii)).

(ii) By the structure lemma, an FHT is the family of its coordinates: the hierarchy $\mathcal T$ and the hosts $\mathcal E_L$. A move with footprint $\Phi$ changes the set of children and the host only of members of $\Phi$: the members of $S\setminus D$, the children of the dissolved members (for an aggregation) and the children of $c$ (for an explosion) are moved as elements, but their own children and hosts are unchanged. Let two moves have footprints $\Phi_1\cap\Phi_2=\varnothing$. The footprint $\Phi_2$ contains components of $\F$, which the second move reads or modifies, and components that it creates. The former are not in $\Phi_1$: after the first move they still exist, with the same children and the same hosts. The latter do not exist in $\F$; the first move does not create them, since $\Phi_1\cap\Phi_2=\varnothing$, and the second move creates them from data that the first move leaves unchanged. The conditions of admissibility of the second move are expressed on these data only (the levels play no role in them), so its datum is still admissible, and it modifies the same coordinates in the same way. The two orders produce the same family of coordinates, hence the same FHT. With more than two moves, argue by induction.

(iii) The footprint of a move located at $K$ consists of components contained in $K$; so moves located at disjoint components have disjoint footprints. Moves $\Th_1$ located at distinct components $K\ne K'$ have the footprints $\{K\}$ and $\{K'\}$.
\end{proof}

\begin{example}[Non-commuting moves]\label{ex:noncommute}
On $V=\{1,2,3,4\}$, let the root have the children $K=\{1,2,3\}$ and $4$, with $\widehat K=\{12,23\}$, so that $\Rea(\F)=\{12,23,1234\}$. The explosion $\Th_2^{-}$ on (root, $K$) with $\varepsilon=\{1,2\}$ is defined: the pieces are $\{1\}$ and $\{2,3\}$, and the result is $\{23,1234\}$. The fusion $\Th_1^{+}$ at $K$ with $e=\{1,2,3\}$, which absorbs $12$ and $23$, is also defined. But after this fusion $\{1,2\}$ is no longer an edge of $\widehat K$, so the datum of the explosion is no longer admissible; and after the explosion $K$ is no longer a component. The footprints $\{\text{root},K,[23]\}$ and $\{K\}$ meet.
\end{example}

\subsection{Edge toggles among the moves}\label{sec:toggle}

By Lemma~\ref{lem:read}, an FHT is determined by its realization $\Rea(\F)$, a family of subsets of $V$. This gives a canonical characterization of the moves that add or delete exactly one edge, the \emph{edge toggles}. They are only part of the moves (Corollary~\ref{cor:cube}).

\begin{theorem}[Deletion of one edge]\label{thm:deletion}
Let $\F$ be an FHT on $V$, $e\in\Rea(\F)$, $K$ the component with $e\in E_K$ (Lemma~\ref{lem:read}(i)) and $\hat e\in E_{\widehat K}$ the corresponding edge of its host. The following are equivalent:
\begin{enumerate}
\item $\Rea(\F)\setminus\{e\}$ is the realization of an FHT $\F'$ on $V$;
\item $K$ is not the root, or $K$ is the root and $E_{\widehat K}\setminus\{\hat e\}$ is an admissible host on $\Ch(K)$.
\end{enumerate}
In that case $\F'$ is unique, and it is the result of $\Th_1^{-}$ at $K$ with $(\hat e,B=\emptyset)$ if $E_{\widehat K}\setminus\{\hat e\}$ is admissible, and of $\Th_2^{-}$ with $(\Par(K),K,\hat e)$ otherwise. In both cases $\Rea(\F')=\Rea(\F)\setminus\{e\}$ is obtained without touching any other edge.
\end{theorem}

\begin{proof}
(ii)$\Rightarrow$(i). If $E_{\widehat K}\setminus\{\hat e\}$ is admissible, $\Th_1^{-}$ with $B=\emptyset$ is defined at $K$; otherwise it has at least two pieces and $K$ is not the root, so $\Th_2^{-}$ is defined for the pair $(\Par(K),K)$. In both cases, by Theorem~\ref{thm:valid} the result is an FHT. Its edges are those of $\F$ except $e$: for $\Th_1^{-}$ this is clear; for $\Th_2^{-}$ the edges of $H$, as unions of children, are unchanged (the union of $\Pi$ is $c$), the edges of the new components are the edges of $\mathcal F\setminus\{\varepsilon\}$, and all other components are untouched. Uniqueness comes from Lemma~\ref{lem:read}.

(i)$\Rightarrow$(ii). Suppose that $K$ is the root, that $E_{\widehat K}\setminus\{\hat e\}$ is not admissible, and that $\F'$ exists. The pieces of $(\widehat K,\hat e)$ are defined as in Definition~\ref{def:pieces}, which applies to any admissible host: they are the connected components of the hypergraph $(\Ch(K),E_{\widehat K}\setminus\{\hat e\})$, a child lying in no remaining edge forming a trivial piece. As noted after Definition~\ref{def:pieces}, the remaining host is admissible if and only if there is exactly one piece; so there are $r\ge2$ pieces $P_1,\dots,P_r$ (if $E_{\widehat K}=\{\hat e\}$, they are the $\card{\Ch(K)}\ge2$ singletons). Put $V_i=\bigcup_{C\in P_i}C$. Since the children of the root partition $V$ (Lemma~\ref{lem:children}), the sets $V_1,\dots,V_r$ form a partition of $V$ into non-empty parts. Every edge of $\Rea(\F)\setminus\{e\}$ is contained in a single $V_i$: an own edge $f\ne e$ of the root is the union of a remaining edge $\hat f$ of the host, which lies in a single piece; an edge which is not an own edge of the root is contained in a child of the root (Lemma~\ref{lem:read}), hence in a single $V_i$. So no edge of $\Rea(\F')=\Rea(\F)\setminus\{e\}$ meets two distinct parts, and the hypergraph $(V,\Rea(\F'))$ is disconnected. But the realization of an FHT on at least two vertices contains the own edges of its root, which form a connected hypergraph covering $V$ (axioms (A2), (A3)). This is a contradiction.
\end{proof}

\begin{theorem}[Addition of one edge]\label{thm:addition}
Let $\F$ be an FHT on $V$ and $e\subseteq V$, $\card e\ge2$, $e\notin\Rea(\F)$. Then $\Rea(\F)\cup\{e\}$ is the realization of an FHT on $V$ if and only if it is obtained from $\F$ by a move $\Th_1^{+}$ with $\Abs=\emptyset$ or a move $\Th_2^{+}$ with $\bigcup\varepsilon=e$. The FHT is then unique.
\end{theorem}

\begin{proof}
If $\Rea(\F)\cup\{e\}=\Rea(\F'')$ for an FHT $\F''$, then $\Rea(\F'')\setminus\{e\}=\Rea(\F)$, and Theorem~\ref{thm:deletion} applied to $\F''$ shows that $\F$ is the result of $\Th_1^{-}(B=\emptyset)$ or of $\Th_2^{-}$ from $\F''$ with the edge $e$. By Proposition~\ref{prop:inverse}, $\F''$ is the result of $\Th_1^{+}$ with $\Abs=\emptyset$ or of $\Th_2^{+}$ from $\F$. The converse is Theorem~\ref{thm:valid} with the computation of the realization in the proof of Theorem~\ref{thm:deletion}.
\end{proof}

\begin{corollary}\label{cor:cube}
Identify an FHT with the indicator vector of its realization in $\{0,1\}^{\{e\subseteq V:\card e\ge2\}}$. The graph $\TG_{\mathrm{tog}}(V)$ of the \emph{edge toggles}, namely the moves $\Th_2^{\pm}$, the fusions with $\Abs(e)=\varnothing$ and the fissions with $B=\varnothing$, is the subgraph of the hypercube induced on the realizations of FHTs. A fusion with $\card{\Abs(e)}=1$ or a fission with $\card B=1$ changes two coordinates, and a fusion with $\card{\Abs(e)}=2$ or a fission with $\card B=2$ changes three.
\end{corollary}

\subsection{The exterior morphism \texorpdfstring{$\Omega_1$}{Omega1}: adding and removing a leaf}\label{sec:omega}

The interior morphisms keep $v$. A single further morphism $\Omega_1$, in two directions, changes $v$ by one. It acts directly on the component where the leaf is added or removed; there is no starred version, since choosing the component $H$ is part of the choice.

\begin{definition}[Leaf addition $\Omega_1^{+}$]\label{def:om1p}
\leavevmode
\begin{description}
\item[Acts on] an FHT $\F$ on $V$ and a non-leaf component $H$ of $\F$.
\item[Choices] (1) a new element $x\notin V$ (the label of the new leaf); (2) a twin block $B\in\Blo(H)$.
\item[Conditions] none.
\item[Action] The leaf $x$ becomes a new child of $H$, twin of the members of $B$: every edge of $\widehat H$ containing $B$ receives $x$,
\[
f\supseteq B\ \longmapsto\ f\cup\{x\},\qquad f\cap B=\emptyset\ \longmapsto\ f.
\]
The vertex $x$ is added to the vertex sets of $H$ and of all its ancestors; so it is added to every own edge of $H$ whose host edge contains $B$, and to every own edge of an ancestor that contains $H$.
\item[Unchanged] all the other components and hosts, the levels, and the reduced host of $H$.
\end{description}
\end{definition}

\begin{definition}[Leaf removal $\Omega_1^{-}$]\label{def:om1m}
\leavevmode
\begin{description}
\item[Acts on] an FHT $\F$ on $V$ and a leaf $x$ of $\F$, child of $H=\Par(x)$.
\item[Choice] the leaf $x$ (and thus $H$).
\item[Conditions] (a) the twin block of $x$ in $\Blo(H)$ has at least two members: \emph{removing $x$ does not empty a twin block}; (b) $\card{\Ch(H)}\ge3$ (equivalently, $\widehat H$ is not a single edge on two children).
\item[Action] $x$ is removed from $\Ch(H)$, from the edges of $\widehat H$ and from the vertex sets of $H$ and of its ancestors.
\item[Unchanged] all the other components and hosts, the levels, and the reduced host of $H$.
\end{description}
\end{definition}

\begin{definition}[Degenerate morphisms $\Omega_1^{\pm,0}$]\label{def:degenerate}
The FHT on a one-element set $\{y\}$, denoted $\F_1$, has no non-leaf component, so $\Omega_1^{+}$ does not act on it; and the FHT on a two-element set $W$ is the star $\St_W$, from which no leaf can be removed because of condition (b). We complete $\Omega_1$ by two \emph{degenerate morphisms}, which create or delete the root:
\begin{description}
\item[$\Omega_1^{+,0}$] acts on $\F_1$ on $\{y\}$ and a new element $x\notin\{y\}$ (no other choice, no condition); its result is the star $\St_{\{y,x\}}$: the two leaves $y$ and $x$ are the children of a new root component, whose host is the single edge $\{y,x\}$.
\item[$\Omega_1^{-,0}$] acts on the star $\St_W$, $W=\{y,x\}$, and a leaf $x$ (two choices); its result is $\F_1$ on $\{y\}$.
\end{description}
The morphisms of Definitions~\ref{def:om1p} and~\ref{def:om1m} are called \emph{regular}. The degenerate morphisms are not instances of the twin block rule (there is no twin block to join, and the twin block of $x$ in $\St_W$ has two members but $\widehat H$ would be left with one child); they are the only moves which create or delete the root component. (Recall the convention $\St_V=\F_1$ for $\card V=1$, stated at the beginning of this section.)
\end{definition}

Condition (b) is only needed when $\widehat H$ has a single block of size $2$; otherwise (a) implies (b) by Lemma~\ref{lem:blocks}. The leaf $x$ added or removed may be a child of the root or of any deeper component: the distance to the root of the leaf is $\rho(H)+1$, where $H$ is the chosen component.

\begin{theorem}[Validity and inverses of $\Omega_1$]\label{thm:omegavalid}
Fix an infinite set $U$ of labels; the FHTs considered are on finite non-empty subsets of $U$. Components are identified across the morphisms $\Omega_1$ by the natural correspondence: after the addition of $x$ at $H$, every component $L\supseteq H$ becomes $L\cup\{x\}$, the other components are unchanged, and $\{x\}$ is a new leaf; after the removal of $x$, every component $L\ni x$ with $L\ne\{x\}$ becomes $L\setminus\{x\}$, and the leaf $\{x\}$ disappears.
\begin{enumerate}
\item (\emph{Regular morphisms.}) $\Omega_1^{+}(\F;x,H,B)$ is an FHT on $V\sqcup\{x\}$ in which $x$ lies in the twin block $B\cup\{x\}$ of $H\cup\{x\}$, of size $\ge2$; $\Omega_1^{-}(\F;x)$ is an FHT on $V\setminus\{x\}$.
\item (\emph{Bijection.}) $\Omega_1^{-}(\Omega_1^{+}(\F;x,H,B);x)=\F$; and if $\F'=\Omega_1^{-}(\F;x)$, $H=\Par_{\F}(x)$ and $B'$ is the twin block of $x$ in $\Blo_{\F}(H)$, then $\Omega_1^{+}(\F';x,H\setminus\{x\},B'\setminus\{x\})=\F$. Hence
\[
(\F,x,H,B)\longmapsto\big(\Omega_1^{+}(\F;x,H,B),\,x\big)
\]
is a bijection from the set of quadruples $(\F,x,H,B)$, where $\F$ is an FHT on a finite set $V\subseteq U$, $x\in U\setminus V$, $H$ is a non-leaf component of $\F$ and $B\in\Blo_{\F}(H)$, onto the set of pairs $(\F'',x)$, where $\F''$ is an FHT on a finite subset of $U$ and $x$ is a leaf of $\F''$ satisfying the conditions (a) and (b) of Definition~\ref{def:om1m}. The inverse map is $(\F'',x)\mapsto\big(\Omega_1^{-}(\F'';x),x,\Par_{\F''}(x)\setminus\{x\},B''\setminus\{x\}\big)$, where $B''$ is the twin block of $x$ in $\Blo_{\F''}(\Par_{\F''}(x))$.
\item (\emph{Effects of the regular morphisms.}) The levels of the components and the depth are unchanged; the hierarchy changes only by the presence of $x$ in $H$ and its ancestors; the reduced host of $H$ is unchanged and the size of one twin block changes by $\pm1$.
\item (\emph{Degenerate morphisms.}) $\Omega_1^{+,0}(\F_1;x)=\St_{\{y,x\}}$ and $\Omega_1^{-,0}(\St_{\{y,x\}};x)=\F_1$ on $\{y\}$ are FHTs, and they are inverse to each other. They create or delete the root: the depth passes from $0$ to $1$ or from $1$ to $0$, and the complexity $\kappa$ (Section~\ref{sec:kappa}) from $0$ to $1$ or from $1$ to $0$. The bijection of (ii) extends to all the morphisms $\Omega_1$, regular or degenerate, by adding the pairs $(\F_1,x,\varnothing,\varnothing)\mapsto(\St_{\{y,x\}},x)$, for $\F_1$ on $\{y\}$ and $x\in U\setminus\{y\}$ (the choice of $(H,B)$ being empty on the side of $\F_1$).
\end{enumerate}
\end{theorem}

\begin{proof}
(i) We use the structure lemma. For $\Omega_1^{+}$ let $\mathcal T'=\{L\cup\{x\}:L\in\mathcal T,\,L\supseteq H\}\cup\{L\in\mathcal T:L\not\supseteq H\}\cup\{\{x\}\}$. It is a hierarchy; the children of $H$ become $\Ch(H)\cup\{\{x\}\}$, those of each ancestor of $H$ are the same as before with the child containing $H$ enlarged by $x$, and the other children are unchanged. So all the hosts except $\widehat H$ are admissible after renaming. For $\widehat H$ let $T$ be the set of edges containing $B$ (non-empty, since the children are covered) and $\mathcal E'=\{f\cup\{x\}:f\in T\}\cup(\mathcal E\setminus T)$. It is a clutter: $f\cup\{x\}\subseteq g\cup\{x\}$ iff $f\subseteq g$; $f\cup\{x\}\not\subseteq g$ for $g\in\mathcal E\setminus T$ since $x\notin g$; and $g\subseteq f\cup\{x\}$ with $g\in\mathcal E\setminus T$ would give $g\subseteq f$, hence $g=f\in T$. It covers $\Ch(H)\cup\{x\}$, since $T\ne\emptyset$, and it is connected, since the chains of $\mathcal E$ are chains of $\mathcal E'$ and $x$ lies in $f\cup\{x\}$. The edges containing $x$ are exactly the images of the edges containing $B$, so the block of $x$ is $B\cup\{x\}$.
For $\Omega_1^{-}$ let $y\ne x$ in the block of $x$ (condition (a)): every edge containing $x$ contains $y$, and conversely. Then $\mathcal E'=\{f\setminus\{x\}:f\in\mathcal E\}$ consists of non-empty sets (an edge $\{x\}$ is excluded by Remark~\ref{rem:nosingleton}) and is a clutter: if $f\setminus\{x\}\subseteq g\setminus\{x\}$ then either $x$ is in both or in none, and $f\subseteq g$; or $x\in g\setminus f$ and $f\subseteq g$ again; or $x\in f\setminus g$, but then $y\in f\setminus\{x\}\subseteq g$, so $x\in g$, a contradiction. It covers $\Ch(H)\setminus\{x\}$ and is connected, since $y$ lies in all the edges containing $x$ (a chain through $x$ is rerouted through $y$). By (b), $\Ch(H)\setminus\{x\}$ has at least two elements, so $\mathcal E'$ is an admissible host, whose edges have at least two elements (Remark~\ref{rem:nosingleton}). The hierarchy loses $x$ and the vertex sets of $H$ and of its ancestors shrink by $x$.
(ii) With the identification of components of the statement, the edges of the host of $H$ after $\Omega_1^{+}$ are the $f\cup\{x\}$ and the $g$; removing $x$ gives back $\mathcal E$, and the same holds for the hierarchy. Conversely, after the removal of $x$ from its block $B'$, $B'\setminus\{x\}$ is a block of $H$ of $\F'$ (the members of the block are twins before and after) and adding $x$ to it restores the edges.
(iii) The leaf has level $0$, hence it does not change the maximum of the levels of the children of $H$ as soon as $H$ has another child (which holds before the addition, and after the removal by (b)). So the level of $H$, and then those of its ancestors, are unchanged. The statement on the reduced host follows from the description of the blocks in (i) and (ii). The inverse map of the bijection is read on $(\F'',x)$: the component $H$ corresponds to $\Par_{\F''}(x)\setminus\{x\}$, and $B$ to the block of $x$ minus $x$. (iv) The star on two elements is an FHT with a single non-leaf component, its root, of level $1$, whose host is the single edge $\{y,x\}$ on its two children; $\F_1$ has no non-leaf component, depth $0$ and complexity $0$. The two morphisms are obviously inverse.
\end{proof}

\begin{remark}[Why a block must not be emptied]
Removing a leaf that is alone in its block is not excluded by validity alone. The condition (a) has two consequences: it makes $\Omega_1^{\pm}$ exact inverses of each other with the data $(H,B)$, and it makes $\Omega_1$ behave like $\Th_2$ with respect to the reduced host: only multiplicities change. All changes of the reduced host are left to $\Th_1$.
\end{remark}

\begin{example}\label{ex:om1}
Let $\F$ be the path $\{12,23\}$ on $\{1,2,3\}$; its blocks are $\{1\},\{2\},\{3\}$. Adding a leaf $4$ to the block $\{2\}$ gives the host $\{12,23\}\to\{124,234\}$ (the leaf $4$ is a twin of $2$). Adding $4$ to the block $\{1\}$ gives $\{124,23\}$. The leaf $4$ can then be removed again, because it is in a block of size $2$; but the leaves $1,2,3$ of $\{12,23\}$ cannot be removed, since they are alone in their blocks. On the star $\{123\}$, a single block of size $3$, every leaf can be removed and the result is the star on two vertices; from this one, no leaf can be removed by a regular $\Omega_1^{-}$ ($\widehat H$ is a single edge on two children), but the degenerate $\Omega_1^{-,0}$ removes either leaf and gives the FHT on one vertex.
\end{example}

\subsection{Naturality}\label{sec:functor}

The word \emph{morphism} now has two meanings: the morphisms $\Th$ of this article (elementary transformations) and the arrows of a category. In this section we show that the former are natural with respect to bijections of the support, which is what is needed to pass to isomorphism classes and to let the symmetric group act.

\subsubsection{Naturality: species}

Let $\Bij$ be the category whose objects are the finite sets (of all cardinalities, including $1$) and whose arrows are the bijections. A \emph{species} is a functor $\Bij\to\mathbf{Set}$~\cite{joyal,bll_species}. For a bijection $\sigma:V\to V'$ and an FHT $\F$ on $V$, $\sigma(\F)$ is the FHT on $V'$ whose components are the images $\sigma(K)$ with own edges $\sigma(E_K)$. All the axioms are invariant under bijections, so $V\mapsto\FHT(V)$, $\sigma\mapsto\sigma(\cdot)$ is a species, which we denote $\FHT$; it is the species enumerated in Section~\ref{sec:enumeration}.

\begin{theorem}[Naturality of the morphisms]\label{thm:natural}
Let $\sigma:V\to V'$ be a bijection and $\Th$ one of $\Th_1^{\pm},\Th_2^{\pm}$. If $d$ is an admissible datum for $\Th$ at the component $K$ of $\F$, then $\sigma(d)$ is an admissible datum for $\Th$ at $\sigma(K)$ in $\sigma(\F)$, and
\[
\sigma\big(\Th(\F;K,d)\big)=\Th\big(\sigma(\F);\sigma(K),\sigma(d)\big).
\]
It follows that:
\begin{enumerate}
\item the sets $\mathfrak M_\Th(V)=\{(\F,K,d)\}$ of moves form a species, and the source and target maps $s,t:\mathfrak M_\Th\to\FHT$ are natural transformations; thus each morphism $\Th$ is a \emph{span of species}, that is, a morphism of species $\mathfrak M_\Th\to\FHT\times\FHT$;
\item $V\mapsto\TG(V)$ is a functor $\Bij\to\mathbf{Graph}$ (the graph isomorphisms $\TG(\sigma)$ preserving the types of the moves and the distance to the root $\rho$);
\item the neighbours of $\F$ in $\TG(V)$ form a set invariant under $\Aut(\F)$, so that the moves descend to the graph $\TG(V)/\!\!\sim$ on the isomorphism classes, on which the enumeration of Section~\ref{sec:enumeration} is made.
\end{enumerate}
\end{theorem}

\begin{proof}
All the notions occurring in the definitions (children, host, edges, blocks, twins, leaf, inclusion of sets, connectedness, pieces, union of sets) are defined from the set structure of $\F$ and are transported by $\sigma$. The statements (i)--(iii) follow.
\end{proof}

The morphism $\Omega_1$ changes the support, so it is not a natural transformation of the species $\FHT$ to itself; its naturality is stated with a distinguished leaf, or through the derivative species.

\begin{proposition}[Naturality of $\Omega_1$]\label{prop:naturalomega}
Let $\tilde\sigma:V\sqcup\{x\}\to V'\sqcup\{x'\}$ be a bijection with $\tilde\sigma(x)=x'$, and let $\sigma=\tilde\sigma|_V$. With the identification of components of Theorem~\ref{thm:omegavalid},
\[
\tilde\sigma\big(\Omega_1^{+}(\F;x,H,B)\big)=\Omega_1^{+}\big(\sigma(\F);x',\sigma(H),\sigma(B)\big),\qquad
\sigma\big(\Omega_1^{-}(\F'';x)\big)=\Omega_1^{-}\big(\tilde\sigma(\F'');x'\big)
\]
for every admissible datum, and the same relations hold for the degenerate morphisms. Equivalently, let $\FHT'$ be the derivative species, $\FHT'(V)=\FHT(V\sqcup\{\star\})$, and let $\mathfrak M_{\Omega}(V)$ be the set of the triples $(\F,H,B)$ with $\F\in\FHT(V)$, $H$ a non-leaf component of $\F$ and $B\in\Blo_{\F}(H)$, completed by $\F_1$ when $\card V=1$. Then $(\F,H,B)\mapsto\Omega_1^{+}(\F;\star,H,B)$ (and $\F_1\mapsto\St_{V\sqcup\{\star\}}$) is an isomorphism of species from $\mathfrak M_{\Omega}$ onto the subspecies of $\FHT'$ formed by the FHTs in which the leaf $\star$ satisfies the conditions (a) and (b) of Definition~\ref{def:om1m}, together with the stars on two points; its inverse is $\Omega_1^{-}$ (regular or degenerate). The bijection of Theorem~\ref{thm:omegavalid}(ii) is the unfolded form of this isomorphism, the label $\star$ being replaced by an arbitrary $x\in U\setminus V$.
\end{proposition}

\begin{proof}
The morphism $\Omega_1$ is defined from children, blocks, edges and inclusions of sets, which are transported by bijections, and the label of the new leaf is transported by $\tilde\sigma$. Naturality in $V$ for the derivative species is the case $x=x'=\star$, and bijectivity is Theorem~\ref{thm:omegavalid}(ii) and (iv).
\end{proof}

\subsection{The transition graphs}\label{sec:transgraphs}

\begin{definition}[Transition graph]\label{def:transition}
The \emph{transition graph} $\TG(V)$ has the FHTs on $V$ as vertices, and an edge between $\F$ and $\F'$ if $\F'$ is the result of a move from $\F$ (Definition~\ref{def:move}). For $r\ge0$, $\TG_{\le r}(V)$ is the subgraph of the moves located at components $K$ with distance to the root $\rho(K)\le r$; $\TG_{0}(V)=\TG_{\le0}(V)$ is the \emph{root-only graph}.
\end{definition}

By Proposition~\ref{prop:inverse}, the relation ``is the result of a move from'' is symmetric, so $\TG(V)$ is an undirected graph. It is a \emph{simple} graph for two further reasons. First, the multiplicity of the choices is forgotten: there is one edge as soon as one move joins the two FHTs (by Proposition~\ref{prop:inverse}(ii), a move $\Th_2$ is in fact determined by its two ends). Second, no move fixes its object, since every interior move changes the realization, which determines the FHT (Lemma~\ref{lem:read}): $\Th_1^{+}$ introduces an edge $e\notin\mathcal E$, $\Th_1^{-}$ replaces an edge by proper subsets of it, and $\Th_2^{\pm}$ add or delete exactly one edge (Theorems~\ref{thm:deletion} and~\ref{thm:addition}). We write $\TG_{\Th}$ for the subgraph of the moves of the type $\Th$ and its inverse.
\subsection{Complexity}\label{sec:kappa}

\begin{definition}[Complexity]
The \emph{complexity} of an FHT $\F$ is $\kappa(\F)=\card{\Rea(\F)}$, the number of its edges.
\end{definition}

By Lemma~\ref{lem:read}(i), $\kappa(\F)=\sum_{K}\card{E_{\widehat K}}$, the sum running over the non-leaf components $K$ of $\F$ (for $\F_1$ the sum is empty). The complexity is invariant under relabelling. For $\card V\ge2$ one has $\kappa(\F)\ge1$, and $\kappa(\F)=1$ if and only if $\Rea(\F)=\{V\}$, that is, $\F$ is the star $\St_V$; for $\card V=1$, $\kappa(\St_V)=\kappa(\F_1)=0$ (boundary conventions at the beginning of this section).

\begin{proposition}[Variation of the complexity]\label{prop:kappa}
Let $\F'$ be the result of a move from $\F$.
\begin{enumerate}
\item $\Th_2^{+}$ increases $\kappa$ by one, and $\Th_2^{-}$ decreases it by one.
\item $\Th_1^{+}$ with $\card{\Abs(e)}=j$ changes $\kappa$ by $1-j\in\{+1,0,-1\}$; $\Th_1^{-}$ with $\card B=j$ changes $\kappa$ by $j-1\in\{-1,0,+1\}$.
\item The regular $\Omega_1^{\pm}$ do not change $\kappa$; the degenerate $\Omega_1^{\pm,0}$ change it by $\pm1$, since $\kappa(\F_1)=0$ and $\kappa(\St_2)=1$.
\end{enumerate}
\end{proposition}

\begin{proof}
Use $\kappa=\sum_K\card{E_{\widehat K}}$. A move $\Th_1$ at $K$ changes the host of $K$ only, and the number of its edges as stated. For $\Th_2^{-}$, the edge $\varepsilon$ of $\widehat c$ disappears; every edge $f\ne\varepsilon$ of $\widehat c$ is contained in a single piece (its elements are linked by $f$ itself), and the hosts of the pieces share the remaining $\card{E_{\widehat c}}-1$ edges, while the host of the parent only changes by the re-indexing of its children; hence $\kappa$ decreases by one. $\Th_2^{+}$ is the inverse. For a regular $\Omega_1$, the edges of every host are the same sets of children, with the new child added to those containing the chosen block, so every $\card{E_{\widehat K}}$ is unchanged. The degenerate case is a direct count.
\end{proof}

A move is called \emph{ascending} if it increases $\kappa$ by one, \emph{descending} if it decreases it by one, and \emph{level} otherwise. By Proposition~\ref{prop:kappa}, the ascending interior moves are $\Th_1^{+}$ with $\Abs(e)=\varnothing$, $\Th_1^{-}$ with $\card B=2$, and $\Th_2^{+}$; their inverses are the descending ones, namely $\Th_1^{-}$ with $B=\varnothing$, $\Th_1^{+}$ with $\card{\Abs(e)}=2$, and $\Th_2^{-}$. The level moves are $\Th_1^{\pm}$ in which one edge is replaced by one edge, and the regular $\Omega_1^{\pm}$; the degenerate $\Omega_1^{+,0}$ is ascending and $\Omega_1^{-,0}$ descending. By Theorems~\ref{thm:deletion} and~\ref{thm:addition}, the moves $\Th_2^{\pm}$, the fusions with $\Abs(e)=\varnothing$ and the fissions with $B=\varnothing$ are exactly the additions and deletions of a single edge.

\subsection{Descent to the star}\label{sec:descent}

\begin{lemma}[Descent in a host]\label{lem:hostdescent}
Let $\widehat H$ be an admissible host with at least two edges. Then there is a descending move $\Th_1$ at $H$: either $\Th_1^{-}$ with $B=\varnothing$ (the deletion of an edge, the result being an admissible host), or $\Th_1^{+}$ with $\card{\Abs(e)}=2$.
\end{lemma}

\begin{proof}
Since $\widehat H$ is connected with at least two edges, it has two intersecting edges $a,b$. Put $e=a\cup b$ and let $\mathcal S$ be the set of edges contained in $e$; it contains $a$ and $b$. Since $e\supsetneq a$ and $\widehat H$ is a clutter, $e$ is not an edge and is contained in no edge.

\emph{Case $\mathcal S=\{a,b\}$.} Then $e$ is contained in no edge and $\card{\Abs(e)}=2$, so $\Th_1^{+}$ with this $e$ applies and descends.

\emph{Case $\card{\mathcal S}\ge3$.} Let $c\in\mathcal S\setminus\{a,b\}$ and let $\mathcal E'=\widehat H\setminus\{c\}$. It is a clutter, and it covers $\Ch(H)$ because every point of $c$ lies in $a\cup b$. It is connected: given two edges, join them by a chain of consecutive intersecting edges in $\widehat H$; if $c$ occurs in the chain between $f$ and $g$, then $f$ and $g$ meet $c\subseteq a\cup b$, so each of them meets $a$ or $b$, and $a$ meets $b$; replace $c$ by the corresponding path through $a$ and $b$, which avoids $c$. Hence $\mathcal E'$ is an admissible host and $\Th_1^{-}$ with $e=c$ and $B=\varnothing$ applies.
\end{proof}

\begin{theorem}[Descent]\label{thm:descent}
Let $\F$ be an FHT on $V$ with $\card V\ge2$ and $\F\ne\St_V$ (equivalently, $\kappa(\F)\ge2$). Then there is a descending move located at the root or at one of its children: a fusion with $\card{\Abs(e)}=2$ or a fission with $B=\varnothing$ at the root or at a child of the root, or an explosion $\Th_2^{-}$ on a pair (root, $c$). So every maximal sequence of descending moves, whether or not it is restricted to moves located at distance at most one from the root, ends at the star $\St_V$ and has length exactly $\kappa(\F)-1$.
\end{theorem}

\begin{proof}
Let $R$ be the root. If $\widehat R$ has at least two edges, Lemma~\ref{lem:hostdescent} gives a descending move $\Th_1$ at $R$. Otherwise $\widehat R=\{\Ch(R)\}$, and since $\F\ne\St_V$ the root has a non-leaf child $c$ (if all its children were leaves, $\F$ would be the star). If $\widehat c$ has at least two edges, Lemma~\ref{lem:hostdescent} gives a descending move $\Th_1$ at $c$. Otherwise $\widehat c=\{\Ch(c)\}$; for $\varepsilon=\Ch(c)$ the pieces of $(\widehat c,\varepsilon)$ are the $\card{\Ch(c)}\ge2$ singletons, so the explosion $\Th_2^{-}$ on $(R,c)$ with $\varepsilon$ is defined, and it decreases $\kappa$ by one (Proposition~\ref{prop:kappa}). All these moves are located at $R$ or at $c$, that is, at distance at most one from the root. By Proposition~\ref{prop:kappa}, each descending move decreases $\kappa$ by exactly one; since a descending move located at distance at most one from the root exists as long as $\kappa\ge2$, and $\kappa=1$ only for the star, the maximal sequences, restricted or not, end there after $\kappa(\F)-1$ steps.
\end{proof}

\begin{corollary}[Distance to the star]\label{cor:diststar}
Let $\card V\ge2$. For every FHT $\F$ on $V$,
\[
d_{\TG(V)}(\F,\St_V)=d_{\TG_{\le1}(V)}(\F,\St_V)=\kappa(\F)-1,
\]
and for all FHTs $\F,\F'$ on $V$,
\[
\lvert\kappa(\F)-\kappa(\F')\rvert\le d_{\TG(V)}(\F,\F')\le d_{\TG_{\le1}(V)}(\F,\F')\le\kappa(\F)+\kappa(\F')-2.
\]
The complexity is thus, up to a shift by one, the distance to the star.
\end{corollary}

\begin{proof}
Every interior move changes $\kappa$ by at most one (Proposition~\ref{prop:kappa}), so\[d_{\TG(V)}(\F,\F')\ge\lvert\kappa(\F)-\kappa(\F')\rvert,\]and in particular $d_{\TG(V)}(\F,\St_V)\ge\kappa(\F)-1$ since $\kappa(\St_V)=1$. Since $\TG_{\le1}(V)$ is a subgraph of $\TG(V)$ on the same vertices, $d_{\TG(V)}\le d_{\TG_{\le1}(V)}$. The restricted descent of Theorem~\ref{thm:descent} is a path of length $\kappa(\F)-1$ from $\F$ to $\St_V$ in $\TG_{\le1}(V)$, which gives the equalities. Going down from $\F$ to $\St_V$ and up to $\F'$ (Proposition~\ref{prop:inverse}) gives the last inequality.
\end{proof}

\begin{corollary}[Order and connectivity]\label{cor:poset}
Let $\card V\ge2$, and let $\F\sqsubseteq\F'$ mean that $\F'$ is obtained from $\F$ by a (possibly empty) sequence of ascending moves.
\begin{enumerate}
\item $\sqsubseteq$ is a partial order on $\FHT(V)$, graded by $\kappa$, whose covering relations are the ascending moves; its minimum is $\St_V$. We denote the poset by $\mathcal P(V)$.
\item The graph $\TG_{\le1}(V)$, hence $\TG(V)$, is connected.
\end{enumerate}
\end{corollary}

\begin{proof}
(i) Antisymmetry follows from the strict increase of $\kappa$ along ascending moves, and a relation $\F\sqsubset\F'$ with $\kappa(\F')=\kappa(\F)+1$ is a single ascending move. By Theorem~\ref{thm:descent}, reversing a maximal descending sequence from $\F$ (Proposition~\ref{prop:inverse}) shows $\St_V\sqsubseteq\F$. (ii) follows from Corollary~\ref{cor:diststar}.
\end{proof}

For $\card V=1$ there is a single FHT, $\F_1=\St_V$, of complexity $0$: the poset $\mathcal P(V)$ and the graphs $\TG(V)$ and $\TG_{\le1}(V)$ are reduced to a point, and the statements of Section~\ref{sec:groupoid} hold trivially in this case.

This proof of connectivity is independent of the one of Section~\ref{sec:connect}: it does not use the reduction Lemma~\ref{lem:H}, and, like it, it only uses moves located at distance at most one from the root. It strengthens connectivity: it gives a path to the star along which the complexity decreases by one at each step, and it identifies $\kappa-1$ with the distance to the star. By the classes computed in Section~\ref{sec:verif}, the maximal complexity is $\kappa_{\max}=\binom{v}{\lfloor v/2\rfloor}$ for $3\le v\le6$, i.e.\ $3,6,10,20$; Corollary~\ref{cor:diststar} then gives $\kappa_{\max}-1\le\operatorname{diam}\TG(V)\le2(\kappa_{\max}-1)$, that is $2\le\operatorname{diam}\TG(V)\le4$, $5\le\operatorname{diam}\TG(V)\le10$, $9\le\operatorname{diam}\TG(V)\le18$ and $19\le\operatorname{diam}\TG(V)\le38$ for $v=3,\dots,6$ (isomorphism classes counted by complexity in the table below).

The numbers of isomorphism classes of each complexity, computed from the classes of Section~\ref{sec:verif}, are the following.
\begin{center}
\small
\begin{tabular}{@{}lcp{0.72\textwidth}@{}}
\toprule
$v$ & $\kappa_{\max}$ & number of classes with $\kappa=1,2,\dots,\kappa_{\max}$ \\
\midrule
$3$ & $3$ & $1,\,2,\,1$ \\
$4$ & $6$ & $1,\,4,\,10,\,6,\,1,\,1$ \\
$5$ & $10$ & $1,\,7,\,31,\,66,\,60,\,36,\,19,\,7,\,2,\,2$ \\
$6$ & $20$ & $1,\allowbreak\,10,\allowbreak\,78,\allowbreak\,350,\allowbreak\,975,\allowbreak\,1846,\allowbreak\,2742,\allowbreak\,3215,\allowbreak\,2906,\allowbreak\,2161,\allowbreak\,1329,\allowbreak\,718,\allowbreak\,352,\allowbreak\,160,\allowbreak\,65,\allowbreak\,25,\allowbreak\,9,\allowbreak\,3,\allowbreak\,1,\allowbreak\,1$ \\
\bottomrule
\end{tabular}
\end{center}
These numbers are computed; they are not used below.

\subsection{Connectivity}\label{sec:connect}

Corollary~\ref{cor:poset} proves the connectivity of $\TG_{\le1}(V)$ by descent. We give a second, independent proof, by reduction of the hosts, which does not use the complexity and also describes the classes of the moves $\Th_1^{*}$.

\begin{lemma}[Reduction of a host]\label{lem:H}
Every admissible host $\mathcal E$ on $X$ can be transformed into $\{X\}$ by a finite sequence of moves $\Th_1^{-}$ and $\Th_1^{+}$. So any two admissible hosts on $X$ are related by such a sequence.
\end{lemma}

\begin{proof}
We argue by induction on $\card{\mathcal E}$. If $\card{\mathcal E}=1$, then $\mathcal E=\{X\}$ by covering. Otherwise, since $\mathcal E$ is connected, there are two edges $e_1\ne e_2$ with $e_1\cap e_2\ne\emptyset$. Set $U=e_1\cup e_2$; by the clutter property, $e_1\subsetneq U$ and $e_2\subsetneq U$. Let $G=\{g\in\mathcal E\setminus\{e_1,e_2\}: g\subseteq U\}$.

\emph{Deletion of the edges of $G$.} Let $g\in G$. Then $\mathcal E\setminus\{g\}$ is an admissible host: it is a clutter; it covers $X$ since $g\subseteq e_1\cup e_2$; and it is connected because any edge meeting $g$ meets $e_1$ or $e_2$, and $e_1,e_2$ meet, so a chain of edges through $g$ can be rerouted through $e_1,e_2$. Hence $\Th_1^{-}$ with $(g,B=\emptyset)$ is defined. We delete the elements of $G$ one after the other; each deletion leaves $e_1,e_2$ and $U$ unchanged. We obtain $\mathcal E_1$ in which $e_1,e_2$ are the only edges contained in $U$.

\emph{Fusion.} The set $U$ is contained in no edge of $\mathcal E_1$ (such an edge would strictly contain $e_1$), and $\Abs(U)=\{e_1,e_2\}$ has two elements. So $\Th_1^{+}$ with $e=U$ is defined, and gives an admissible host $\mathcal E_2$ with $\card{\mathcal E_2}=\card{\mathcal E_1}-1<\card{\mathcal E}$. Apply the induction hypothesis. The last statement follows because the moves are invertible (Proposition~\ref{prop:inverse}).
\end{proof}

\begin{theorem}[Connectivity]\label{thm:connect}
For every finite non-empty set $V$, the graph $\TG_{\le1}(V)$ is connected. In particular $\TG(V)$ is connected: any two FHTs on $V$ are related by a finite sequence of moves $\Th_1^{\pm},\Th_2^{\pm}$ located at the root and at its children.
\end{theorem}

\begin{proof}
For $\card V\ge2$, Theorem~\ref{thm:descent} gives more: a path of length $\kappa(\F)-1$ from $\F$ to $\St_V$ in $\TG_{\le1}(V)$ (Corollary~\ref{cor:diststar}). For $\card V=1$ there is a single FHT.

\emph{Alternative proof.} It uses moves $\Th_1$ inside one host, then an explosion, and shows which kinds of moves are needed; this is the point of Proposition~\ref{prop:necessity}. If $\card V\le2$ there is a single FHT. Let $\card V\ge3$. We show that every $\F$ is connected to the star $\mathrm{St}_V$ within $\TG_{\le1}(V)$; since the moves are invertible, the theorem follows. We argue by induction on the number $m(\F)$ of non-root non-leaf components.

If $m(\F)>0$, the root has a non-leaf child $c$ (otherwise all its children would be leaves and there would be no non-root non-leaf component). By Lemma~\ref{lem:H} applied to $\widehat c$, a sequence of moves $\Th_1^{\pm}$ at $c$ (located at $c$, at distance $\rho(c)=1$ from the root) brings the host of $c$ to the single edge $Y=\Ch(c)$. Now $\Th_2^{-}$ at the root and $c$ with $\varepsilon=Y$ is defined: $\mathcal F\setminus\{\varepsilon\}=\emptyset$ has $\card Y\ge2$ trivial pieces. The result has $c$ replaced in $\Ch(\text{root})$ by the points of $Y$, which keep their sub-FHTs, so $m$ decreased by one. By induction we reach an FHT $\F_0$ with $m(\F_0)=0$, that is, flat. By Lemma~\ref{lem:H} applied to its root host, a sequence of $\Th_1^{\pm}$ at the root reaches $\mathrm{St}_V$.
\end{proof}

\begin{proposition}[What does not suffice]\label{prop:necessity}
\leavevmode
\begin{enumerate}
\item (\emph{$\Th_2$ alone}.) For every $v\ge3$ the graph $\TG_{\Th_2}(V)$ is disconnected: the flat FHT whose host is the path $\{12,23,\dots,(v-1)v\}$ is isolated.
\item (\emph{$\Th_1$ alone}.) The moves $\Th_1^{\pm *}$ preserve the hierarchy; for $v\ge3$ there are several hierarchies, so $\TG_{\Th_1}(V)$ is disconnected.
\item (\emph{The root alone}.) For $v\ge4$ the graph $\TG_0(V)$ is disconnected: the FHT whose root has two children, a leaf $v$ and a component with the complete graph $\{ij:1\le i<j\le v-1\}$ on $\{1,\dots,v-1\}$ as host, is isolated. For $v=4$ it is $\{12,13,23,1234\}$.
\item (\emph{Edge toggles alone}.) The edge toggles of Corollary~\ref{cor:cube} do not connect the FHTs already for $v=3$ (Example~\ref{ex:v3}).
\end{enumerate}
\end{proposition}

\begin{proof}
(i) In the path host, the children $1,\dots,v$ lie in the distinct sets of edges
\[\{12\},\ \{12,23\},\ \{23,34\},\ \dots,\ \{(v-1)v\},\]
so all blocks are singletons: $\Th_2^{+}$ (which needs $\card S\ge2$ in a block) is not defined, and $\Th_2^{-}$ needs a non-leaf child.

(ii) is Proposition~\ref{prop:effects}(ii).

(iii) The root has two children, so its host is the single edge $\Ch$ on two points. $\Th_1^{+}$ needs a subset of size $\ge2$ contained in no edge, but the only such subset is $\Ch$ itself, which is an edge. $\Th_1^{-}$ would need fragments which are proper subsets of $\Ch$, that is singletons, impossible in an admissible host (Remark~\ref{rem:nosingleton}). $\Th_2^{+}$ needs $2\le\card S\le1$. $\Th_2^{-}$ is only possible at the non-leaf child, with an edge $ij$; but deleting one edge of the complete graph on $\ge3$ points leaves it connected and covering, so there is a single piece. Hence no move at the root exists.

(iv) is Example~\ref{ex:v3}.
\end{proof}

\begin{proposition}[Classes of $\Th_1^{*}$]\label{prop:orbits}
Two FHTs on $V$ are related by a sequence of moves $\Th_1^{\pm *}$ if and only if they have the same hierarchy. The graph of $\Th_1^{\pm*}$ on the FHTs with a given hierarchy $\mathcal T$ is the Cartesian product, over the non-leaf members $K$ of $\mathcal T$, of the graphs on $\AH(\Ch(K))$ whose edges are the moves $\Th_1^{\pm}$, and each factor is connected. Hence the connected components of the graph $\TG_{\Th_1}(V)$ correspond to the hierarchies of $V$; their number is $1,1,4,26,236,2752$ for $v=1,\dots,6$, and $1,1,2,5,12,33$ up to isomorphism.
\end{proposition}

\begin{proof}
By the structure lemma, the FHTs with hierarchy $\mathcal T$ are the families $(\mathcal E_K)_K\in\prod_K\AH(\Ch(K))$, and a $\Th_1$ move at $K$ modifies only the coordinate $K$. So the graph is the Cartesian product of the host graphs, which are connected by Lemma~\ref{lem:H}. The hierarchies of $V$ are the total partitions (Schr\"oder's fourth problem~\cite{schroder_1870,oeis}), whose number is $1,1,4,26,236,2752$ for $v=1,\dots,6$, and $1,1,2,5,12,33$ unlabelled.
\end{proof}

\subsubsection{The transition graphs for $v\le4$}\label{sec:figs}

Figures~\ref{fig:th1}, \ref{fig:th2} and~\ref{fig:th12} show the transition graphs on isomorphism classes for $v=2,3,4$ (the $1+4+23=28$ classes; the trivial FHT with $v=1$ has no move). A vertex is a class, labelled by its number $F_k$ in the dictionary of Table~\ref{tab:fht_list} and by its realization; the colour of a vertex depends on $v$. An edge joins two classes if some move of the type considered at some component (that is, a $\Th^{*}$) transforms a representative of the first into a representative of the second. Moves $\Th_1^{\pm}$ are drawn in grey, moves $\Th_2^{\pm}$ in red. No two classes are joined by both types. On labelled FHTs this holds because $\Th_1$ preserves the hierarchy and $\Th_2$ changes it. On isomorphism classes, a move $\Th_1$ preserves the isomorphism type of the hierarchy, as a rooted tree, while a move $\Th_2$ changes it: an aggregation located at $H$ decreases the number of children of $H$ by $\card S-1\ge1$, so the rooted type of the subtree at $H$ changes; at each ancestor of $H$, the multiset of the types of the children changes in exactly one entry, $X\mapsto X'$ with $X'\not\cong X$, hence it changes, since $M\uplus\{X\}=M\uplus\{X'\}$ forces $X\cong X'$; so the type of the whole tree changes. An explosion is the inverse case. (A coarser invariant would not do: the number of components changes by $1-\card D$ under an aggregation, which vanishes when $\card D=1$.)

\begin{itemize}
\item $\Th_1^{\pm *}$ (Figure~\ref{fig:th1}): $35$ edges and $8$ connected components, which are the hierarchies up to isomorphism (Proposition~\ref{prop:orbits}): $1$ for $v=2$, $2$ for $v=3$, $5$ for $v=4$. The large component of $v=4$ consists of the $14$ flat FHTs; the four others have $3$, $4$, $1$ and $1$ classes.
\item $\Th_2^{\pm *}$ (Figure~\ref{fig:th2}): $10$ edges and $19$ components ($1+3+15$); most FHTs are isolated (in particular all flat FHTs without twins); for $v=4$ there is one component with $6$ classes and three with $2$ classes, and for $v=3$ one with $2$ classes.
\item $\Th_1^{\pm *}\cup\Th_2^{\pm *}$ (Figure~\ref{fig:th12}): $45$ edges and $3$ connected components, one for each $v$ (Theorem~\ref{thm:connect}).
\end{itemize}

\begin{figure}[p]
\centering
\includegraphics[width=\textwidth,height=0.82\textheight,keepaspectratio]{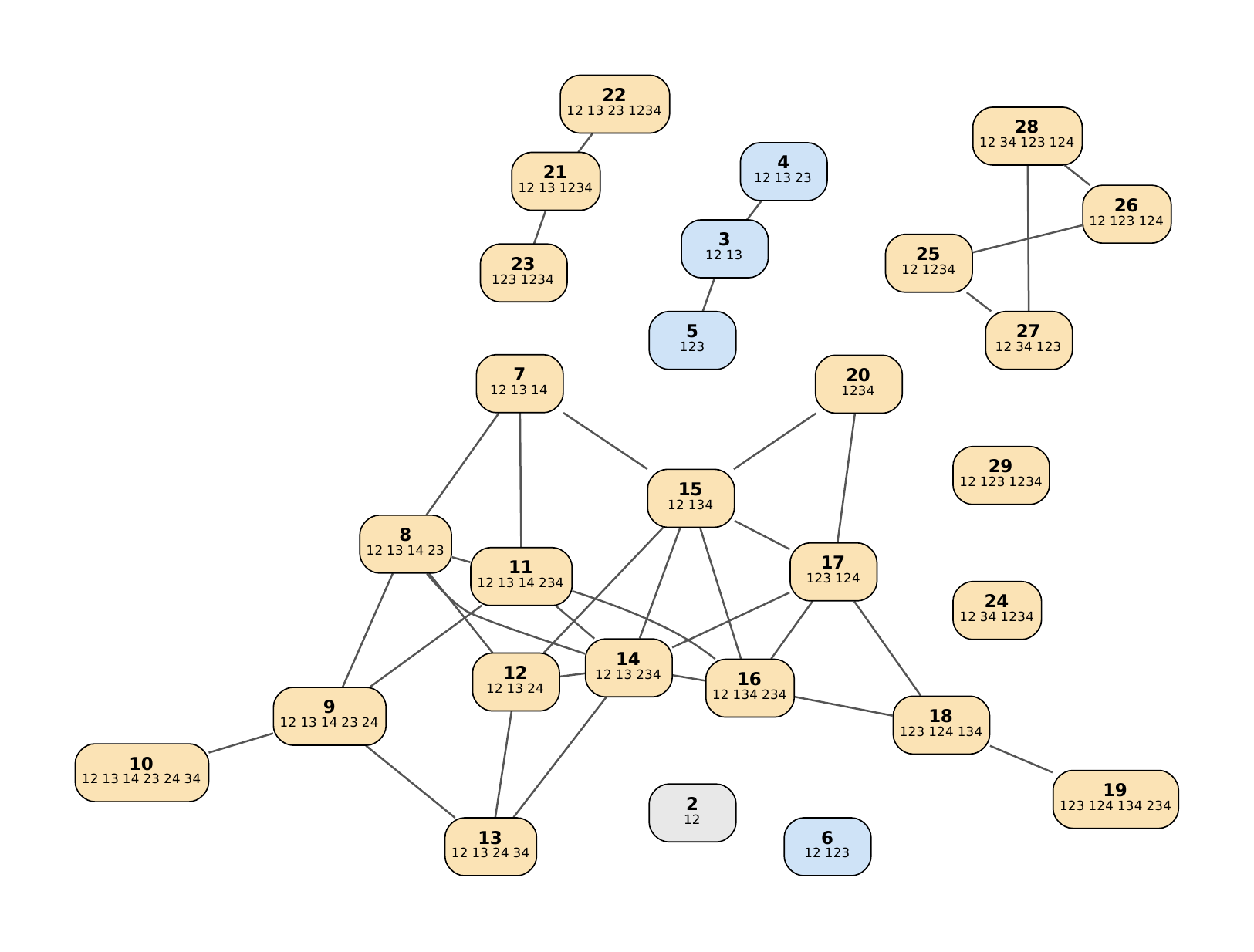}
\caption{Transition graph of $\Th_1^{\pm *}$ on the classes of FHTs with $v=2,3,4$ (grey: $v=2$; blue: $v=3$; orange: $v=4$). Each vertex shows its number $F_k$ and its realization.}\label{fig:th1}
\end{figure}

\begin{figure}[p]
\centering
\includegraphics[width=\textwidth,height=0.82\textheight,keepaspectratio]{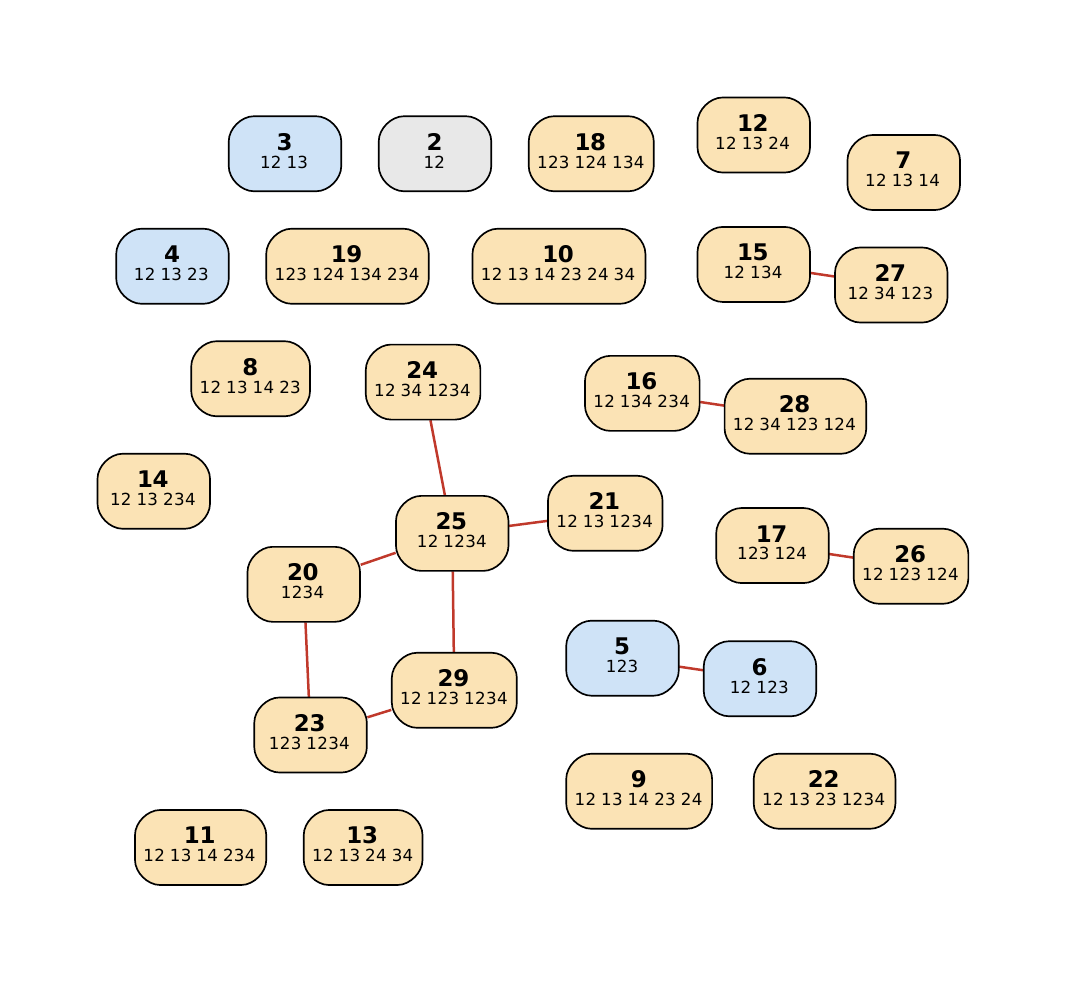}
\caption{Transition graph of $\Th_2^{\pm *}$ on the same classes (red edges). The isolated vertices are the FHTs on which no explosion or aggregation is possible.}\label{fig:th2}
\end{figure}

\begin{figure}[p]
\centering
\includegraphics[width=\textwidth,height=0.88\textheight,keepaspectratio]{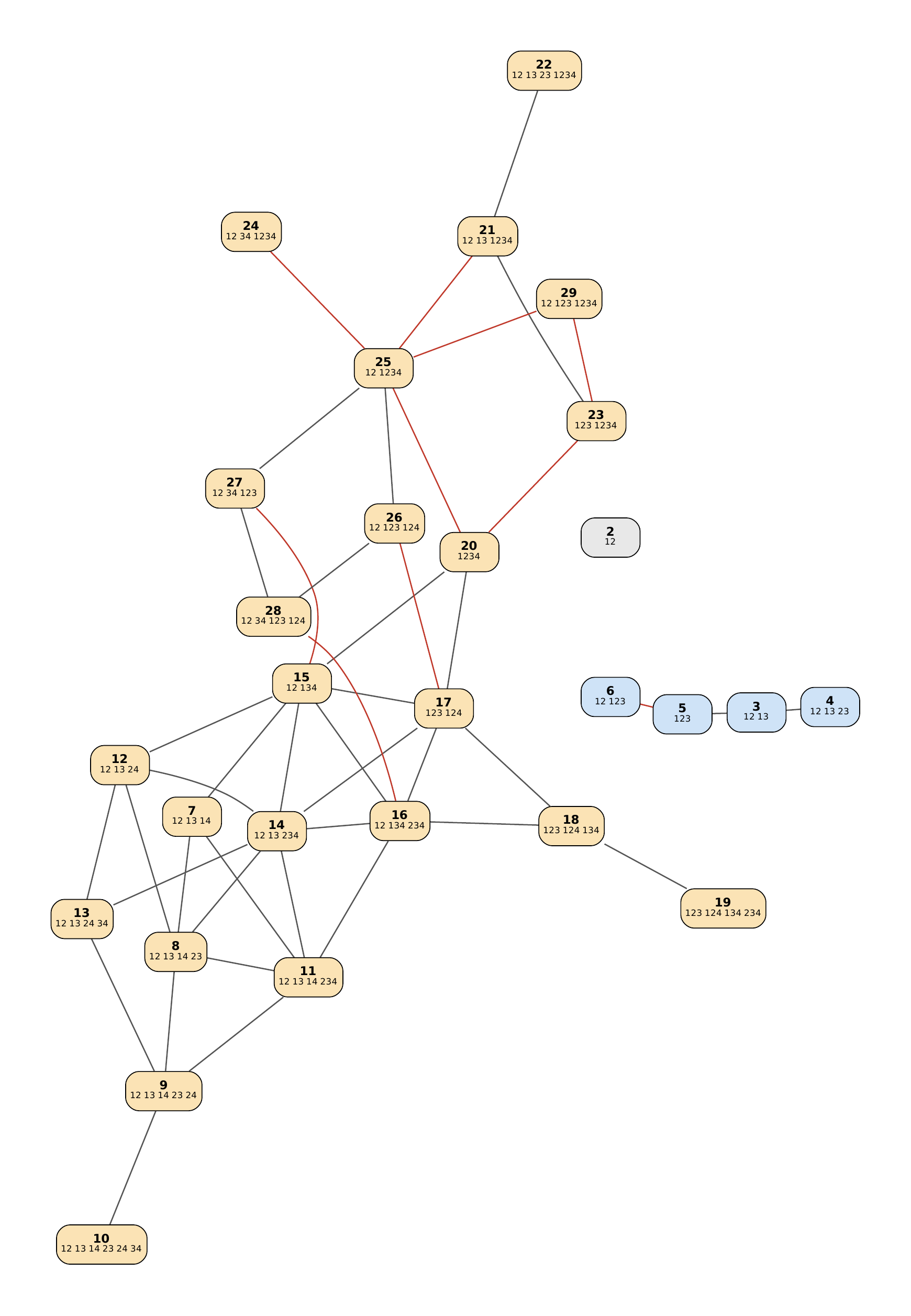}
\caption{Transition graph of $\Th_1^{\pm *}$ (grey) and $\Th_2^{\pm *}$ (red) on the same classes: three connected components, one for each $v$.}\label{fig:th12}
\end{figure}

\begin{theorem}[Connectivity across sizes]\label{thm:global}
Let $U$ be an infinite set and $\TG^{\Omega}$ the graph whose vertices are the pairs $(V,\F)$ with $V\subseteq U$ finite, $\card V\ge1$, and $\F\in\FHT(V)$, and whose edges are the moves $\Th_1^{\pm},\Th_2^{\pm}$ (on the same $V$) and $\Omega_1^{\pm}$, regular or degenerate. Then $\TG^{\Omega}$ is connected, and it is enough to use the moves $\Th$ located at distance at most one from the root and the moves $\Omega_1$ at the root. In particular, any two FHTs, with any number $\ge1$ of vertices, are related by a finite sequence of interior and exterior morphisms. Without the degenerate morphisms, the same holds for the FHTs with at least two vertices, and each FHT on a one-element set is isolated.
\end{theorem}

\begin{proof}
By Theorem~\ref{thm:connect}, every $(V,\F)$ with $\card V\ge2$ is connected to $(V,\mathrm{St}_V)$ by interior morphisms located at distance at most one from the root; for $\card V=1$, $\F=\F_1=\St_V$. For $\card V\ge2$ and $x\in U\setminus V$, $\Omega_1^{+}$ at the root with the unique block $V$ transforms $\mathrm{St}_V$ into $\mathrm{St}_{V\sqcup\{x\}}$, and for $\card V=1$ the degenerate $\Omega_1^{+,0}$ does the same; conversely, $\Omega_1^{-}$ removes a leaf $x$ of $\mathrm{St}_W$ if $\card W\ge3$, and $\Omega_1^{-,0}$ if $\card W=2$. Given $(V,\mathrm{St}_V)$ and $(V',\mathrm{St}_{V'})$, add one by one the elements of $V'\setminus V$ to reach $\mathrm{St}_{V\cup V'}$, then remove one by one the elements of $V\setminus V'$, which is always possible down to $V'$ ($\card{V'}\ge1$), the last step being degenerate when $\card{V'}=1$. Without the degenerate morphisms, the argument stops at supports with at least two elements.
\end{proof}

\begin{remark}[The trivial FHT]
The FHT $\F_1$ with $v=1$ is attached to the other FHTs only by the degenerate morphisms $\Omega_1^{\pm,0}$ of Definition~\ref{def:degenerate}, which join it to $\St_2$. Other possible degenerate rules, for instance the collapse of a component with two children after the removal of one of them (the other child then takes the place of the component in the host of the parent), are not considered here.
\end{remark}

\subsection{The morphic groupoid}\label{sec:groupoid}

The moves of Section~\ref{sec:moves} come in inverse pairs (Proposition~\ref{prop:inverse}), so that $\TG(V)$ is an undirected graph, connected by Corollary~\ref{cor:poset} and Theorem~\ref{thm:connect}. We now organize it into groupoids, at three levels: the free groupoid keeps all the paths; the thin groupoid keeps a single arrow between two objects: it identifies parallel arrows by construction, and Proposition~\ref{prop:thin} relates it to the poset $\mathcal P(V)$, whereas the free groupoid keeps the non-trivial loops; and the action of the symmetric group adds the symmetries. The complexity gives a \emph{signed grading}: the order $\sqsubseteq$ describes the ascending paths, but in a groupoid every arrow is invertible, and the complexity defines a functor $\Delta\kappa$ to the additive group $\mathbb Z$ (Proposition~\ref{prop:thin}) rather than an orientation of the arrows.

\subsubsection{The groupoids}\label{sec:groupoids}

\begin{definition}[Free and thin groupoids]\label{def:groupoids}
The \emph{free groupoid} $\mathfrak F(V)$ has the FHTs on $V$ as objects, and as arrows the homotopy classes, with fixed end points, of the paths in the geometric realization of the graph $\TG(V)$ (a space of dimension one). Its vertex groups are free of rank $\card{E(\TG(V))}-\card{\FHT(V)}+1$. The \emph{thin groupoid} $\mathfrak T(V)$ has the same objects and exactly one arrow $\F\to\F'$ for every pair of objects (the pair groupoid); it is the quotient of $\mathfrak F(V)$ which identifies parallel arrows. Both are connected, by Theorem~\ref{thm:connect}.
\end{definition}

\begin{remark}[Commuting moves in the free groupoid]\label{rem:squares}
Two moves that commute as transformations (Proposition~\ref{prop:locality}) give a square in $\TG(V)$, but this square is not trivial in $\mathfrak F(V)$: the graph has no cells of dimension two that would impose commutation relations, and the free groupoid keeps all the paths before the passage to the thin quotient. An intermediate groupoid is obtained by gluing a $2$-cell on every square formed by two moves with disjoint footprints (Definition~\ref{def:footprint}) and taking the fundamental groupoid of the resulting $2$-complex. It lies between $\mathfrak F(V)$ and $\mathfrak T(V)$, and it is the natural setting for comparing the weights of paths that differ only by the order of independent moves. It is not used below.
\end{remark}

\begin{proposition}[The thin groupoid is the groupoid of fractions of $\mathcal P(V)$]\label{prop:thin}
Let $\card V\ge2$. The groupoid obtained from the poset $\mathcal P(V)$, regarded as a category, by formally inverting all its arrows is the thin groupoid $\mathfrak T(V)$. Let $B\mathbb Z$ be the groupoid with one object associated with the additive group $\mathbb Z$. The functor
\[
\Delta\kappa:\mathfrak T(V)\to B\mathbb Z,\qquad \Delta\kappa(\F\to\F')=\kappa(\F')-\kappa(\F),
\]
is well defined, and it is the only functor $\mathfrak T(V)\to B\mathbb Z$ sending each ascending move to $1$.
\end{proposition}

\begin{proof}
Write $O=\St_V$, the minimum of $\mathcal P(V)$ (Corollary~\ref{cor:poset}): it is an initial object of $\mathcal P(V)$. For every object $\F$, let $u_{\F}:O\to\F$ be the unique arrow. For every arrow $f:\F\to\F'$ of $\mathcal P(V)$ one has $f\circ u_{\F}=u_{\F'}$, by uniqueness of the arrows out of $O$. In the groupoid of fractions, whose arrows are the composites of the images of the arrows of $\mathcal P(V)$ and of their formal inverses, this gives $f=u_{\F'}\circ u_{\F}^{-1}$ and $f^{-1}=u_{\F}\circ u_{\F'}^{-1}$. Every composite of such arrows from $\F$ to $\F'$ telescopes to $u_{\F'}\circ u_{\F}^{-1}$; so there is exactly one arrow from $\F$ to $\F'$ (at least one, namely $u_{\F'}\circ u_{\F}^{-1}$), and the groupoid of fractions is the pair groupoid $\mathfrak T(V)$. The argument only uses that $\mathcal P(V)$ has an initial object.

The map $\Delta\kappa$ is a coboundary on the pair groupoid, hence a functor, and it sends each ascending move to $1$ (Proposition~\ref{prop:kappa}). Conversely, let $\Phi:\mathfrak T(V)\to B\mathbb Z$ send each ascending move to $1$. Every ascending chain from $\St_V$ to $\F$ has length $\kappa(\F)-1$, so $\Phi(u_{\F})=\kappa(\F)-1$, and $\Phi(\F\to\F')=\Phi(u_{\F'})-\Phi(u_{\F})=\Delta\kappa(\F\to\F')$.
\end{proof}

Alternatively, the groupoid of fractions of a category is the fundamental groupoid of its nerve (Gabriel--Zisman~\cite{gz}), and the nerve of $\mathcal P(V)$ is a cone with apex $\St_V$, hence contractible. For $\card V=1$ the statement holds trivially: there is one object, no ascending move, and $\Delta\kappa=0$.

Thus the thin groupoid is, as a category, equivalent to a point: its content is the signed grading $\Delta\kappa$ and the symmetry. The relabelling group $S_V$ acts on $\FHT(V)$, preserves the moves (Theorem~\ref{thm:natural}) and $\kappa$, hence acts on $\mathcal P(V)$ by order automorphisms, and on $\mathfrak T(V)$.

\begin{proposition}[Isotropy]\label{prop:isotropy}
\leavevmode
\begin{enumerate}
\item The action groupoid $\mathfrak A(V)=S_V\ltimes\FHT(V)$, with objects the FHTs on $V$ and arrows $\sigma:\F\to\sigma\F$, has the isotropy group $\Aut(\F)$ at $\F$ and its connected components are the isomorphism classes; it is equivalent to the disjoint union $\bigsqcup_{[\F]}\Aut(\F)$ over the isomorphism classes.
\item The semidirect product $\mathfrak T(V)\rtimes S_V$ (arrows $\F\to\F'$ are the pairs $(\sigma,t)$, $\sigma\in S_V$, $t:\sigma\F\to\F'$ in $\mathfrak T(V)$) is connected with $\mathrm{Hom}(\F,\F')\cong S_V$ for all $\F,\F'$, so it is equivalent to the group $S_V$: its isotropy group is $S_V$ at every object, not $\Aut(\F)$. (The thin groupoid $\mathfrak T(V)$ itself has trivial isotropy; only the action groupoid of (i) has the groups $\Aut(\F)$ as isotropy groups.)
\item The quotient $\overline{\mathcal P}(v)=S_V\backslash\mathcal P(V)$, ordered by $[\F]\le[\F']$ if $\F\sqsubseteq\sigma\F'$ for some $\sigma$, is a graded poset with minimum $[\St_V]$, graded by $\kappa$; its rank function is $r([\F])=\kappa(\F)-1$, which equals the distance $d_{\TG(V)}(\F,\St_V)$ (Corollary~\ref{cor:diststar}; here $\card V\ge2$, and for $\card V=1$ there is a single class, of rank $0$); its covering relations are exactly the images of the ascending moves, and its underlying set is the set of isomorphism classes of FHTs with $v$ vertices.
\end{enumerate}
\end{proposition}

\begin{proof}
(i) is the definition of the action groupoid. (ii) Since $\mathfrak T(V)$ is thin, $t$ is unique for each $\sigma$. (iii) The relation is reflexive and transitive because $S_V$ acts by order automorphisms; if $[\F]\le[\F']\le[\F]$, then $\kappa(\F)=\kappa(\F')$ and $\F\sqsubseteq\sigma\F'$ forces $\F=\sigma\F'$. The complexity is constant on the orbits, since it is invariant under relabelling. If $[\F]<[\F']$ with $\kappa(\F')-\kappa(\F)\ge2$, an ascending chain from $\F$ to $\sigma\F'$ passes through an FHT $\F''$ with $\kappa(\F)<\kappa(\F'')<\kappa(\F')$, and $[\F]<[\F'']<[\F']$: this is not a covering. Conversely, if $[\F]\le[\F']$ and $\kappa(\F')=\kappa(\F)+1$, then $\F\sqsubseteq\sigma\F'$ is a single ascending move (Corollary~\ref{cor:poset}(i)), and no class lies strictly between, since $\kappa$ increases strictly along strict relations: this is a covering. So the coverings are the images of the ascending moves; they increase $r$ by one, and $r$ is a rank function vanishing at the minimum. For $\card V=1$ the poset is a point and $r=0$.
\end{proof}

By (ii), the thin groupoid, even equipped with the symmetry, forgets the groups $\Aut(\F)$: the information on the symmetry of an FHT is contained in the $S_V$-set $\FHT(V)$, that is, in the stabilizers of (i). This is the object of the next section.

\section{Classes \texorpdfstring{$(v,\alpha)$}{(v,alpha)} and the associated algebras}\label{sec:valpha}

\subsection{The action \texorpdfstring{$\alpha$}{alpha} and the classes}\label{sec:alpha}

The group $S_V$ acts on $\FHT(V)$ by relabelling (Theorem~\ref{thm:natural}); the stabilizer of $\F$ is its automorphism group $\Aut(\F)\subseteq S_V$, a permutation group on the vertices.

\begin{definition}[The class $(v,\alpha)$]\label{def:valpha}
The \emph{action} of $\F\in\FHT(V)$ is the subgroup $\Aut(\F)\subseteq S_V$ \emph{up to conjugacy in $S_V$}; we denote it $\alpha(\F)$. Two FHTs $\F$ on $V$ and $\F'$ on $V'$ are \emph{in the same class $(v,\alpha)$} if $\card V=\card{V'}=v$ and there is a bijection $\varphi:V\to V'$ with $\varphi\Aut(\F)\varphi^{-1}=\Aut(\F')$, that is, if the permutation groups $\Aut(\F)$ and $\Aut(\F')$ are conjugate in $S_v$. Fixed points of $\Aut(\F)$ are \emph{kept}: they are vertices on which the action is trivial, and are part of $\alpha$.
\end{definition}

Since $\Aut(\sigma\F)=\sigma\Aut(\F)\sigma^{-1}$, isomorphic FHTs are in the same class: a class $(v,\alpha)$ is a union of isomorphism classes, called its \emph{fibre}, whose cardinality is denoted
\[
n(v,\alpha)=\#\{[\F]:\F\text{ belongs to the class }(v,\alpha)\},
\]
a number of isomorphism classes (the set of the labelled FHTs of the class is denoted $\Phi$ in Theorem~\ref{thm:wedderburn}). Thus $(v,\alpha)$ is coarser than the isomorphism class, and finer than $v$ alone. The classes are described by the following exact sequence.

\begin{proposition}[Splitting of the sequence of Theorem~\ref{thm:autrec}]\label{prop:aut}
Let $H$ be the root of $\F$, $\card V\ge2$, and let $\Gamma(\F)=\Aut(\widehat H^{\tau})$ be the group of the automorphisms of the host that preserve the types $\tau(c)$ of the children (Definition~\ref{def:gamma}). The exact sequence of Theorem~\ref{thm:autrec},
\[
1\longrightarrow\prod_{c\in\Ch(H)}\Aut(\F_c)\longrightarrow\Aut(\F)\longrightarrow\Gamma(\F)\longrightarrow1,
\]
splits, non-canonically:
\[
\Aut(\F)\cong\Big(\prod_{c\in\Ch(H)}\Aut(\F_c)\Big)\rtimes\Gamma(\F),
\]
the semidirect product depending on the choice of identifications between the children of the same type.
\end{proposition}

\begin{proof}
The exactness is Theorem~\ref{thm:autrec}. For the splitting, choose for each type a model and isomorphisms $\varphi_c$ from the children of this type to the model; then $\pi\mapsto(\varphi_{\pi(c)}^{-1}\circ\varphi_c)_c$, glued on the children, is a homomorphism $\Gamma(\F)\to\Aut(\F)$, since $\varphi_{\pi\pi'(c)}^{-1}\varphi_{\pi'(c)}\circ\varphi_{\pi'(c)}^{-1}\varphi_c=\varphi_{\pi\pi'(c)}^{-1}\varphi_c$, and it is a section of the projection.
\end{proof}

For instance $\Aut(\St_V)=S_V$, and the twin children of one block with the same type can be permuted freely. Table~\ref{tab:valpha} lists all classes $(v,\alpha)$ for $v\le6$.

{\footnotesize\setlength{\tabcolsep}{3pt}
\begin{longtable}{@{}rlrp{3.7cm}lrrll@{}}
\caption{The classes $(v,\alpha)$ for $1\le v\le 6$. $G=\Aut(\F)$ up to conjugacy in $S_v$, given by its isomorphism type, its order, generators of one representative, the sizes of its orbits on the $v$ vertices, the number $n$ of isomorphism classes of the fibre, and the number $n\,v!/\card G$ of labelled FHTs of the class; the next column gives the dimensions of the irreducible complex representations of $G$ ($d^m$ means $m$ of dimension $d$), and the last the numbers $(r,c,h)$ of real-type irreducible representations, of pairs of complex-conjugate irreducible representations, and of quaternionic-type ones.\label{tab:valpha}}\\
\toprule
$v$ & $G$ & $\card G$ & generators & orbits & $n$ & labelled & $\mathrm{Irr}(G)$ & $(r,c,h)$\\
\midrule
\endfirsthead
\toprule
$v$ & $G$ & $\card G$ & generators & orbits & $n$ & labelled & $\mathrm{Irr}(G)$ & $(r,c,h)$\\
\midrule
\endhead
\bottomrule
\endfoot
1 & $1$ & 1 & $e$ & $1$ & 1 & 1 & $1$ & $(1,0,0)$ \\
\midrule
2 & $C_2$ & 2 & $(1\,2)$ & $2$ & 1 & 1 & $1^{2}$ & $(2,0,0)$ \\
\midrule
3 & $C_2$ & 2 & $(2\,3)$ & $2+1$ & 2 & 6 & $1^{2}$ & $(2,0,0)$ \\
3 & $S_3$ & 6 & $(1\,2\,3),\allowbreak\ (2\,3)$ & $3$ & 2 & 2 & $2,\,1^{2}$ & $(3,0,0)$ \\
\midrule
4 & $C_2$ & 2 & $(1\,2)(3\,4)$ & $2+2$ & 1 & 12 & $1^{2}$ & $(2,0,0)$ \\
4 & $C_2$ & 2 & $(3\,4)$ & $2+1+1$ & 6 & 72 & $1^{2}$ & $(2,0,0)$ \\
4 & $C_2^2$ & 4 & $(3\,4),\allowbreak\ (1\,2)$ & $2+2$ & 6 & 36 & $1^{4}$ & $(4,0,0)$ \\
4 & $S_3$ & 6 & $(2\,3\,4),\allowbreak\ (3\,4)$ & $3+1$ & 5 & 20 & $2,\,1^{2}$ & $(3,0,0)$ \\
4 & $D_4$ & 8 & $(1\,3\,2\,4),\allowbreak\ (3\,4)$ & $4$ & 2 & 6 & $2,\,1^{4}$ & $(5,0,0)$ \\
4 & $S_4$ & 24 & $(1\,2\,3\,4),\allowbreak\ (1\,2\,4\,3)$ & $4$ & 3 & 3 & $3^{2},\,2,\,1^{2}$ & $(5,0,0)$ \\
\midrule
5 & $1$ & 1 & $e$ & $1+1+1+1+1$ & 7 & 840 & $1$ & $(1,0,0)$ \\
5 & $C_2$ & 2 & $(2\,3)(4\,5)$ & $2+2+1$ & 22 & 1320 & $1^{2}$ & $(2,0,0)$ \\
5 & $C_2$ & 2 & $(4\,5)$ & $2+1+1+1$ & 57 & 3420 & $1^{2}$ & $(2,0,0)$ \\
5 & $C_2^2$ & 4 & $(2\,3)(4\,5),\allowbreak\ (2\,4)(3\,5)$ & $4+1$ & 1 & 30 & $1^{4}$ & $(4,0,0)$ \\
5 & $C_2^2$ & 4 & $(4\,5),\allowbreak\ (2\,3)$ & $2+2+1$ & 62 & 1860 & $1^{4}$ & $(4,0,0)$ \\
5 & $S_3$ & 6 & $(3\,4\,5),\allowbreak\ (4\,5)$ & $3+1+1$ & 20 & 400 & $2,\,1^{2}$ & $(3,0,0)$ \\
5 & $D_4$ & 8 & $(2\,4\,3\,5),\allowbreak\ (4\,5)$ & $4+1$ & 15 & 225 & $2,\,1^{4}$ & $(5,0,0)$ \\
5 & $D_5$ & 10 & $(1\,2\,4\,5\,3),\allowbreak\ (2\,3)(4\,5)$ & $5$ & 2 & 24 & $2^{2},\,1^{2}$ & $(4,0,0)$ \\
5 & $D_6$ & 12 & $(1\,2)(3\,4\,5),\allowbreak\ (4\,5)$ & $3+2$ & 32 & 320 & $2^{2},\,1^{4}$ & $(6,0,0)$ \\
5 & $S_4$ & 24 & $(2\,3\,4\,5),\allowbreak\ (2\,3\,5\,4)$ & $4+1$ & 9 & 45 & $3^{2},\,2,\,1^{2}$ & $(5,0,0)$ \\
5 & $S_5$ & 120 & $(1\,2)(3\,4\,5),\allowbreak\ (1\,2\,3)(4\,5)$ & $5$ & 4 & 4 & $6,\,5^{2},\,4^{2},\,1^{2}$ & $(7,0,0)$ \\
\midrule
6 & $1$ & 1 & $e$ & $1+1+1+1+1+1$ & 7274 & 5237280 & $1$ & $(1,0,0)$ \\
6 & $C_2$ & 2 & $(1\,2)(3\,4)(5\,6)$ & $2+2+2$ & 120 & 43200 & $1^{2}$ & $(2,0,0)$ \\
6 & $C_2$ & 2 & $(3\,4)(5\,6)$ & $2+2+1+1$ & 1528 & 550080 & $1^{2}$ & $(2,0,0)$ \\
6 & $C_2$ & 2 & $(5\,6)$ & $2+1+1+1+1$ & 4455 & 1603800 & $1^{2}$ & $(2,0,0)$ \\
6 & $C_3$ & 3 & $(1\,2\,3)(4\,5\,6)$ & $3+3$ & 6 & 1440 & $1^{3}$ & $(1,1,0)$ \\
6 & $C_2^2$ & 4 & $(3\,4)(5\,6),\allowbreak\ (1\,2)(3\,5)(4\,6)$ & $4+2$ & 22 & 3960 & $1^{4}$ & $(4,0,0)$ \\
6 & $C_2^2$ & 4 & $(3\,4)(5\,6),\allowbreak\ (3\,5)(4\,6)$ & $4+1+1$ & 27 & 4860 & $1^{4}$ & $(4,0,0)$ \\
6 & $C_2^2$ & 4 & $(5\,6),\allowbreak\ (1\,2)(3\,4)$ & $2+2+2$ & 259 & 46620 & $1^{4}$ & $(4,0,0)$ \\
6 & $C_2^2$ & 4 & $(3\,4)(5\,6),\allowbreak\ (1\,2)(5\,6)$ & $2+2+2$ & 22 & 3960 & $1^{4}$ & $(4,0,0)$ \\
6 & $C_2^2$ & 4 & $(5\,6),\allowbreak\ (3\,4)$ & $2+2+1+1$ & 1731 & 311580 & $1^{4}$ & $(4,0,0)$ \\
6 & $C_4$ & 4 & $(1\,2)(3\,5\,4\,6)$ & $4+2$ & 9 & 1620 & $1^{4}$ & $(2,1,0)$ \\
6 & $C_6$ & 6 & $(1\,4\,5\,2\,3\,6)$ & $6$ & 2 & 240 & $1^{6}$ & $(2,2,0)$ \\
6 & $S_3$ & 6 & $(1\,3\,4)(2\,5\,6),\allowbreak\ (3\,4)(5\,6)$ & $3+3$ & 51 & 6120 & $2,\,1^{2}$ & $(3,0,0)$ \\
6 & $S_3$ & 6 & $(4\,5\,6),\allowbreak\ (5\,6)$ & $3+1+1+1$ & 316 & 37920 & $2,\,1^{2}$ & $(3,0,0)$ \\
6 & $C_2^3$ & 8 & $(5\,6),\allowbreak\ (1\,2)(3\,4),\allowbreak\ (1\,3)(2\,4)$ & $4+2$ & 9 & 810 & $1^{8}$ & $(8,0,0)$ \\
6 & $C_2^3$ & 8 & $(5\,6),\allowbreak\ (3\,4),\allowbreak\ (1\,2)$ & $2+2+2$ & 181 & 16290 & $1^{8}$ & $(8,0,0)$ \\
6 & $D_4$ & 8 & $(1\,2)(3\,5\,4\,6),\allowbreak\ (5\,6)$ & $4+2$ & 60 & 5400 & $2,\,1^{4}$ & $(5,0,0)$ \\
6 & $D_4$ & 8 & $(1\,2)(3\,5\,4\,6),\allowbreak\ (3\,5)(4\,6)$ & $4+2$ & 20 & 1800 & $2,\,1^{4}$ & $(5,0,0)$ \\
6 & $D_4$ & 8 & $(3\,5\,4\,6),\allowbreak\ (5\,6)$ & $4+1+1$ & 122 & 10980 & $2,\,1^{4}$ & $(5,0,0)$ \\
6 & $D_5$ & 10 & $(2\,3\,5\,6\,4),\allowbreak\ (3\,4)(5\,6)$ & $5+1$ & 15 & 1080 & $2^{2},\,1^{2}$ & $(4,0,0)$ \\
6 & $D_6$ & 12 & $(1\,5\,4\,2\,3\,6),\allowbreak\ (3\,4)(5\,6)$ & $6$ & 12 & 720 & $2^{2},\,1^{4}$ & $(6,0,0)$ \\
6 & $D_6$ & 12 & $(2\,3)(4\,5\,6),\allowbreak\ (5\,6)$ & $3+2+1$ & 453 & 27180 & $2^{2},\,1^{4}$ & $(6,0,0)$ \\
6 & $D_4\times C_2$ & 16 & $(3\,5\,4\,6),\allowbreak\ (1\,2)(3\,5\,4\,6),\allowbreak\ (5\,6)$ & $4+2$ & 72 & 3240 & $2^{2},\,1^{8}$ & $(10,0,0)$ \\
6 & $A_4\times C_2$ & 24 & $(1\,3\,5\,2\,4\,6),\allowbreak\ (1\,3\,6\,2\,4\,5)$ & $6$ & 2 & 60 & $3^{2},\,1^{6}$ & $(4,2,0)$ \\
6 & $S_4$ & 24 & $(1\,2)(3\,5\,4\,6),\allowbreak\ (1\,3\,2\,4)(5\,6)$ & $6$ & 4 & 120 & $3^{2},\,2,\,1^{2}$ & $(5,0,0)$ \\
6 & $S_4$ & 24 & $(3\,4\,5\,6),\allowbreak\ (3\,4\,6\,5)$ & $4+1+1$ & 47 & 1410 & $3^{2},\,2,\,1^{2}$ & $(5,0,0)$ \\
6 & $S_3\times S_3$ & 36 & $(2\,3)(4\,5\,6),\allowbreak\ (1\,2)(4\,5\,6),\allowbreak\ (1\,2\,3)(5\,6)$ & $3+3$ & 34 & 680 & $4,\,2^{4},\,1^{4}$ & $(9,0,0)$ \\
6 & $S_4\times C_2$ & 48 & $(1\,3\,5\,2\,4\,6),\allowbreak\ (1\,3\,6\,2\,4\,5),\allowbreak\ (3\,5\,4\,6)$ & $6$ & 9 & 135 & $3^{4},\,2^{2},\,1^{4}$ & $(10,0,0)$ \\
6 & $S_4\times C_2$ & 48 & $(1\,2)(4\,5\,6),\allowbreak\ (1\,2)(3\,4\,5),\allowbreak\ (3\,4\,5\,6)$ & $4+2$ & 57 & 855 & $3^{4},\,2^{2},\,1^{4}$ & $(10,0,0)$ \\
6 & $A_5$ & 60 & $(2\,3\,5\,6\,4),\allowbreak\ (1\,2\,3\,6\,4)$ & $6$ & 1 & 12 & $5,\,4,\,3^{2},\,1$ & $(5,0,0)$ \\
6 & $S_3\wr C_2$ & 72 & $(2\,3)(4\,5\,6),\allowbreak\ (1\,2)(4\,5\,6),\allowbreak\ (1\,2\,3)(5\,6),\allowbreak\ (1\,4\,2\,5\,3\,6)$ & $6$ & 8 & 80 & $4^{4},\,2,\,1^{4}$ & $(9,0,0)$ \\
6 & $S_5$ & 120 & $(2\,3)(4\,5\,6),\allowbreak\ (2\,3\,4)(5\,6)$ & $5+1$ & 14 & 84 & $6,\,5^{2},\,4^{2},\,1^{2}$ & $(7,0,0)$ \\
6 & $S_6$ & 720 & $(2\,3)(4\,5\,6),\allowbreak\ (2\,3\,4)(5\,6),\allowbreak\ (1\,2)(4\,5\,6)$ & $6$ & 5 & 5 & $16,\,10^{2},\,9^{2},\,5^{4},\,1^{2}$ & $(11,0,0)$ \\
\end{longtable}}

\begin{table}[ht]
\centering
\begin{tabular}{@{}lrrrrrr@{}}
\toprule
$v$ & $1$ & $2$ & $3$ & $4$ & $5$ & $6$\\
\midrule
FHTs on $V$ (labelled) & 1 & 1 & 8 & 149 & 8\,488 & 7\,923\,621\\
isomorphism classes & 1 & 1 & 4 & 23 & 231 & 16\,947\\
classes $(v,\alpha)$ & 1 & 1 & 2 & 6 & 11 & 33\\
conjugacy classes of subgroups of $S_v$ & 1 & 2 & 4 & 11 & 19 & 56\\
\bottomrule
\end{tabular}
\caption{Numbers of FHTs, of isomorphism classes, of classes $(v,\alpha)$, and of conjugacy classes of subgroups of $S_v$ (OEIS A000638~\cite{oeis}).\label{tab:counts}}
\end{table}

The following facts are read from Table~\ref{tab:valpha} (they are computed, not proved).
\begin{itemize}
\item The classes $(v,\alpha)$ are very few compared to the isomorphism classes: $33$ against $16\,947$ for $v=6$. The groups which occur are far from all subgroups of $S_v$.
\item For $v\le5$ the groups that occur are $1$, $C_2$, $C_2^2$, $S_3$, $D_4$, $D_5$, $D_6\cong S_3\times C_2$, $S_4$ and $S_5$ (in all, $11$ of the $19$ conjugacy classes of subgroups of $S_5$ for $v=5$), and all their irreducible representations are of real type. No cyclic group of order $\ge3$ occurs.
\item The first cyclic groups of order $\ge3$ occur at $v=6$: $C_3$ ($6$ isomorphism classes), $C_4$, $C_6$, and also $A_4\times C_2$ and $A_5$ (acting transitively on $6$ points). A representative with $\Aut(\F)=C_3$ is the flat FHT with edges $123,\,124,\,135,\,345,\,236,\,146,\,256$, whose automorphism group is generated by $(1\,2\,3)(4\,6\,5)$.
\item The flat FHT with the $10$ edges $123,124,135,245,345,236,146,346,156,256$ (a $2$--$(6,3,2)$ design: every pair of vertices lies in exactly two edges) has $\Aut=A_5$ acting transitively on the six vertices; it is the only class with this group.
\end{itemize}

\subsection{The algebras}\label{sec:algebras}

Let $\alpha=\alpha(\F)$ and $G=\Aut(\F)\subseteq S_V$. All algebras are finite dimensional and semisimple; we use the standard facts on the representations of finite groups~\cite{serre}.

\begin{definition}[Algebras of a class]\label{def:algebras}
\leavevmode
\begin{enumerate}
\item The \emph{group algebras} $A^{\mathbb C}_\alpha=\mathbb C[G]$ and $A^{\mathbb R}_\alpha=\mathbb R[G]$.
\item The \emph{leaf algebra} $\mathbb C^V$ (functions on the vertices), with the invariant subalgebra $(\mathbb C^V)^G=\mathbb C^{V/G}$ and the \emph{crossed product} $B_\alpha=\mathbb C^V\rtimes G$.
\item The \emph{fibre algebra} $\mathfrak A_{(v,\alpha)}=\mathbb C[\,S_V\ltimes\Phi\,]$, the groupoid algebra of the restriction of $\mathfrak A(V)$ to the set $\Phi\subseteq\FHT(V)$ of all FHTs on $V$ of the class $(v,\alpha)$.
\end{enumerate}
\end{definition}

\begin{theorem}[Wedderburn decompositions; classical, recalled]\label{thm:wedderburn}
\leavevmode
\begin{enumerate}
\item $\mathbb C[G]\cong\bigoplus_{\rho\in\mathrm{Irr}(G)}M_{d_\rho}(\mathbb C)$; the number of factors is the number $k(G)$ of conjugacy classes of $G$, and $K_0(\mathbb C[G])\cong\mathbb Z^{k(G)}$.
\item $\mathbb R[G]\cong\bigoplus_{\rho\ \mathrm{real}}M_{d_\rho}(\mathbb R)\ \oplus\bigoplus_{\{\rho,\bar\rho\}\ \mathrm{complex}}M_{d_\rho}(\mathbb C)\ \oplus\bigoplus_{\rho\ \mathrm{quaternionic}}M_{d_\rho/2}(\mathbb H)$, where the type of $\rho$ is given by the Frobenius--Schur indicator $\frac1{\card G}\sum_{g}\chi_\rho(g^2)\in\{1,0,-1\}$. The number of simple factors is $r+c+h$, the triple $(r,c,h)$ of Table~\ref{tab:valpha} being the ``nature of the real structure''.
\item $\mathbb C^V\rtimes G\cong\bigoplus_{O\in V/G}M_{\card O}\big(\mathbb C[G_{x_O}]\big)$, where $x_O\in O$ and $G_{x_O}$ is its stabilizer; its dimension is $v\card G$.
\item Let $\Phi$ be the set of FHTs of the class $(v,\alpha)$, a union of $n=n(v,\alpha)$ orbits of $S_V$, each isomorphic to $S_V/G$ as an $S_V$-set. Then
$\mathfrak A_{(v,\alpha)}\cong\bigoplus_{i=1}^{n}M_{v!/\card G}\big(\mathbb C[G]\big)$, which has dimension $n\,(v!)^2/\card G$ and is Morita equivalent to $\mathbb C[G]^{\,n}$. The algebra $\mathbb C[S_V\ltimes\FHT(V)]$ of all labelled FHTs on $V$ is $\bigoplus_{[\F]}M_{v!/\card{\Aut\F}}(\mathbb C[\Aut\F])$.
\end{enumerate}
\end{theorem}

\begin{proof}
The matrix isomorphisms below are in general not canonical. (i) and (ii) are Maschke's and Wedderburn's theorems, with the classification of real simple modules by the Frobenius--Schur indicator~\cite{serre}. (iii) The algebra $\mathbb C^V\rtimes G$ is the direct sum over the orbits $O$ of $\mathbb C^O\rtimes G$; for a transitive $G$-set $O\cong G/H$, the crossed product $\mathbb C^{G/H}\rtimes G$ is isomorphic to $M_{\card{G/H}}(\mathbb C[H])$ (imprimitivity theorem), and the dimensions agree: $\card O\cdot\card G=\card O^2\card H$. (iv) The groupoid algebra of a transitive groupoid with $N$ objects and vertex group $A$ is $M_N(\mathbb C[A])$; here the components of $S_V\ltimes\Phi$ are the $n$ orbits, with $N=v!/\card G$ objects and vertex group $G$ (Proposition~\ref{prop:isotropy}(i)).
\end{proof}

\begin{example}
$\mathbb R[C_3]=\mathbb R\oplus\mathbb C$, $\mathbb R[C_4]=\mathbb R^2\oplus\mathbb C$, $\mathbb R[C_6]=\mathbb R^2\oplus\mathbb C^2$ and $\mathbb R[A_4\times C_2]=\mathbb R^2\oplus\mathbb C^2\oplus M_3(\mathbb R)^2$. For the star, $\mathbb R[S_V]=\bigoplus_{\lambda\vdash v}M_{f^\lambda}(\mathbb R)$ ($f^\lambda$ the number of standard Young tableaux of shape $\lambda$), all of real type, and $\mathbb C^V\rtimes S_V=M_v(\mathbb C[S_{v-1}])$. For the path $\{12,13\}$ ($v=3$), $G=\langle(2\,3)\rangle\cong C_2$ has the orbits $\{1\}$ and $\{2,3\}$, so $\mathbb C^V\rtimes G=\mathbb C[C_2]\oplus M_2(\mathbb C)$, of dimension $2+4=3\cdot2$.
\end{example}

\begin{remark}[What the table shows about the untwisted algebras]\label{rem:connes}
For all $54$ classes with $v\le6$ (computed, Table~\ref{tab:valpha}), no real group algebra $\mathbb R[\Aut\F]$ has a quaternionic factor ($h=0$ throughout), and the only complex factors are $\mathbb C$ (of complex dimension $1$): they appear at $v=6$ only, for $C_3$, $C_4$, $C_6$ and $A_4\times C_2$. In particular a factor $M_n(\mathbb C)$ with $n\ge2$ or $\mathbb H$ does not occur in the group algebras of automorphism groups of FHTs with at most six vertices. Hence the algebra $A_F=\mathbb C\oplus\mathbb H\oplus M_3(\mathbb C)$ of Connes' description of the Standard Model is not obtained as $\mathbb R[\Aut\F]$ at this size, and in fact at no size: by Proposition~\ref{prop:B1} (with $\eta=1$), every real group algebra $\mathbb R[G]$ has a simple factor $\mathbb R$, which $A_F$ has not. The restriction to small sizes only concerns the observations on $\mathbb H$ and on $M_n(\mathbb C)$, $n\ge2$. It is obtained as the \emph{twisted} group algebra $\mathbb R^{\mathrm{sgn}}[S_4]$ of the next subsection, and this does occur for $v\ge4$. The rank of $K_0$ and the nature of the real structure, invariants of $(v,\alpha)$ read from the last two columns of Table~\ref{tab:valpha}, are class invariants in the sense of the programme.
\end{remark}

\subsection{Twisted group algebras and the algebra \texorpdfstring{$A_F$}{A\_F}}\label{sec:twisted}

The real group algebra $\mathbb R[G]$ is not the only real algebra attached to the pair formed by a finite group and a sign character. Let $\eta:G\to\{\pm1\}$ be a homomorphism and write $\bar e(g)\in\{0,1\}$ for $\eta(g)=(-1)^{\bar e(g)}$.

\begin{definition}[Twisted group algebra]\label{def:twisted}
The \emph{twisted group algebra} $\mathbb R^{\eta}[G]$ is the real algebra with basis $(u_g)_{g\in G}$, product $u_gu_h=(-1)^{\bar e(g)\bar e(h)}u_{gh}$, unit $u_1$ and involution $u_g^{*}=(-1)^{\bar e(g)}u_{g^{-1}}$. For $\eta=1$ it is $\mathbb R[G]$.
\end{definition}

The twisted algebra is attached to the pair $(G,\eta)$, not to the group $G$ alone; the algebra attached to an FHT is $\mathbb R^{\eta}[\Aut\F]$ for a chosen character $\eta$ (Remark~\ref{rem:whicheta}). The product is associative because $(g,h)\mapsto(-1)^{\bar e(g)\bar e(h)}$ is a $2$-cocycle: both ways of bracketing $u_gu_hu_k$ give the sign $(-1)^{\bar e(g)\bar e(h)+\bar e(g)\bar e(k)+\bar e(h)\bar e(k)}$, since $\bar e(gh)\equiv\bar e(g)+\bar e(h)\pmod 2$; the same congruence shows $(u_gu_h)^*=u_h^*u_g^*$.

\begin{lemma}[Central extension]\label{lem:central}
Let $\hat G=\{(g,c)\in G\times\mathbb Z/4:\ c\equiv\bar e(g)\bmod 2\}$ and $z=(1,2)$, a central element of order $2$. Then $\hat G/\langle z\rangle\cong G$, and $\mathbb R[\hat G]=\mathbb R[\hat G]\tfrac{1+z}{2}\oplus\mathbb R[\hat G]\tfrac{1-z}{2}$ with $\mathbb R[\hat G]\tfrac{1+z}{2}\cong\mathbb R[G]$ and $\mathbb R[\hat G]\tfrac{1-z}{2}\cong\mathbb R^{\eta}[G]$ via $u_g\mapsto(g,\bar e(g))\tfrac{1-z}{2}$. $\mathbb C^{\eta}[G]\cong\mathbb C[G]$ via $u_g\mapsto i^{\bar e(g)}g$.
\end{lemma}
\begin{proof}
In $\hat G$, $(g,\bar e(g))(h,\bar e(h))=(gh,\bar e(g)+\bar e(h))$, which equals $(gh,\bar e(gh))$ if $\bar e(g)\bar e(h)=0$ and $z\,(gh,\bar e(gh))$ if $\bar e(g)=\bar e(h)=1$; on $\frac{1-z}2$ we have $z=-1$. The last map is an algebra isomorphism because $i^{\bar e(g)}i^{\bar e(h)}=(-1)^{\bar e(g)\bar e(h)}i^{\bar e(gh)}$.
\end{proof}

For $\rho\in\mathrm{Irr}(G)$ put $\tilde\rho(u_g)=i^{\bar e(g)}\rho(g)$; this is a complex irreducible representation of $\mathbb R^{\eta}[G]$ (Lemma~\ref{lem:central}), and every complex irreducible representation is of this form. Define
\[
\nu_\eta(\rho)=\frac1{\card G}\sum_{g\in G}\eta(g)\,\chi_\rho(g^2)\in\mathbb C .
\]

\begin{theorem}[{Wedderburn decomposition of $\mathbb R^{\eta}[G]$; classical~\cite{karpilovsky}, recalled}]\label{thm:twistedWedderburn}
The conjugate representation of $\tilde\rho$ is $\widetilde{\bar\rho\otimes\eta}$. Hence, with $d_\rho=\dim\rho$:
\begin{enumerate}
\item if $\bar\rho\otimes\eta\not\cong\rho$, then $\nu_\eta(\rho)=0$ and the pair $\{\rho,\bar\rho\otimes\eta\}$ contributes one factor $M_{d_\rho}(\mathbb C)$;
\item if $\bar\rho\otimes\eta\cong\rho$, then $\nu_\eta(\rho)=\pm1$, and $\rho$ contributes $M_{d_\rho}(\mathbb R)$ if $\nu_\eta(\rho)=+1$ and $M_{d_\rho/2}(\mathbb H)$ if $\nu_\eta(\rho)=-1$.
\end{enumerate}
Thus $\mathbb R^{\eta}[G]\cong\bigoplus M_{d}(\mathbb R)\oplus\bigoplus M_d(\mathbb C)\oplus\bigoplus M_{d/2}(\mathbb H)$ over these three kinds of orbits, and the dimensions add up to $\card G$. For $\eta=1$ this is Theorem~\ref{thm:wedderburn}(ii).
\end{theorem}
\begin{proof}
By Lemma~\ref{lem:central}, $\mathbb R^{\eta}[G]$ is the sum of the simple factors of $\mathbb R[\hat G]$ on which $z$ acts as $-1$, and the complex irreducible representations of $\hat G$ with $z\mapsto-1$ are the $\tilde\rho(g,c)=i^{c}\rho(g)$. Since $\overline{i^{c}}=i^{c}(-1)^{c}=i^{c}\eta(g)$ for $c\equiv\bar e(g)$, the conjugate of $\tilde\rho$ is $\widetilde{\bar\rho\otimes\eta}$. The Frobenius--Schur indicator of $\tilde\rho$ is $\frac1{\card{\hat G}}\sum_{(g,c)}\chi_{\tilde\rho}((g,c)^2)$; as $(g,c)^2=(g^2,2c)$ and $z$ acts as $-1$, $\chi_{\tilde\rho}(g^2,2c)=(-1)^{c}\chi_\rho(g^2)=\eta(g)\chi_\rho(g^2)$, each $g$ occurring twice, so this indicator is $\nu_\eta(\rho)$. The real Wedderburn theorem for $\mathbb R[\hat G]$ (\cite{serre}, \S13.2) gives the three cases.
\end{proof}

\begin{example}\label{ex:S3}
For $G=S_3$ and $\eta=\mathrm{sgn}$: the characters $1$ and $\mathrm{sgn}$ are exchanged by $\rho\mapsto\bar\rho\otimes\eta$, which gives a factor $\mathbb C$; the $2$-dimensional representation is fixed and $\nu_\eta=\frac16(2-3\cdot2-2)=-1$, so it gives $\mathbb H$. So $\mathbb R^{\mathrm{sgn}}[S_3]\cong\mathbb C\oplus\mathbb H$, whereas $\mathbb R[S_3]=\mathbb R^2\oplus M_2(\mathbb R)$.
\end{example}

The twist has a geometric meaning: it is an orientation character, and the three types of factors are the types of Hodge duality.

\begin{theorem}[Twisted group algebras and Hodge duality; {[T]}, D]\label{thm:hodge}
Let $W$ be a real orthogonal representation of $G$ of dimension $n$, $w=\det W$, $\Lambda=\Lambda^\bullet(W)\otimes\mathbb C$ with the action $\Lambda(g)$, $\star$ the Hodge star of an orientation of $W$, $\mathrm{conj}$ the complex conjugation of $\Lambda$, and $C=\mathrm{conj}\star$, an antilinear map.
\begin{enumerate}
\item $\star\,\Lambda(g)=w(g)\,\Lambda(g)\,\star$ for every $g\in G$.
\item Under $u_g\mapsto i^{\bar e(g)}g$, $\mathbb R^{w}[G]$ is the real form of $\mathbb C[G]$ fixed by $\sum a_gg\mapsto\sum\bar a_g\,w(g)\,g$, and its image on $\Lambda$ is the set of elements of the image of $\mathbb C[G]$ that commute with $C$. If $\Lambda$ is a faithful $\mathbb C[G]$-module (all the irreducible representations of $G$ occur in it, which is stronger than the faithfulness of the action of $G$), for instance for $W=w\oplus\mathbb R[G]^2$ (of determinant $w$), this image is $\mathbb R^w[G]$ itself.
\item On $\Lambda^k$, $C^2=(-1)^{k(n-k)}$. Hence an irreducible constituent of multiplicity one of $\Lambda^k$, $2k\ne n$, which $C$ maps onto a non-isomorphic constituent of $\Lambda^{n-k}$, carries a factor $M_d(\mathbb C)$ of $\mathbb R^w[G]$; an irreducible constituent of multiplicity one of $\Lambda^{n/2}$ stable under $C$ carries a factor of real type if $n\equiv0\pmod 4$ and of quaternionic type if $n\equiv2\pmod4$.
\end{enumerate}
\end{theorem}
\begin{proof}
(i) is the covariance of the Hodge star under $O(n)$. (ii) $\Lambda(g)$ is real, so for $x=\sum a_g\Lambda(g)$ one has $CxC^{-1}=\sum\bar a_g\,w(g)\,\Lambda(g)$ by (i). If $x$ commutes with $C$ it is the image of $\frac12\sum(a_g+\bar a_gw(g))g$, whose coefficients are real for $g$ even and imaginary for $g$ odd, that is of an element of the image of $\mathbb R^w[G]$; conversely such images commute with $C$. (iii) $\mathrm{conj}$ commutes with $\star$ and $\star^2=(-1)^{k(n-k)}$ on $\Lambda^k$, hence $C^2=(-1)^{k(n-k)}$. Let $U$ be an irreducible constituent of multiplicity one. The image of $\mathbb C[G]$ on $U$ is the full algebra $\mathrm{End}(U)$ (Burnside), and by (ii) the image of $\mathbb R^w[G]$ consists of the elements commuting with $C$. If $C(U)=U$, this is the real form of $\mathrm{End}(U)$ defined by the antilinear map $C|_U$, of real type if $C^2=+1$ and of quaternionic type if $C^2=-1$: this is the Wedderburn factor carried by $U$. If $C$ maps $U$ onto a non-isomorphic constituent $U'$, an element commuting with $C$ is determined by its restriction to $U$, which is arbitrary, and the factor is $\mathrm{End}(U)\cong M_d(\mathbb C)$. Multiplicity one is needed: on an isotypic component of higher multiplicity, the sign of $C^2$ alone does not determine the type of the underlying irreducible representation.
\end{proof}

Theorem~\ref{thm:hodge} interprets a chosen character as an orientation character; it does not select it. Every character $\eta$ is the determinant of the one-dimensional representation $\eta$ itself, so the existence of a representation $W$ with $\det W=\eta$ does not make $\eta$ obligatory for a given FHT.

\begin{corollary}[The factors of $A_F$ as Hodge types; {[T]}, D]\label{cor:hodgeAF}
Let $S_4$ act on the plane $W_2$ of the three pairings of four points (through $S_4\to S_3$) and on the space $W_3$ of the tetrahedron ($S_4=T_d\subset O(3)$); both have determinant $\mathrm{sgn}$. On $\Lambda^0W_2\oplus\Lambda^2W_2=\mathbf 1\oplus\mathrm{sgn}$ the Hodge star exchanges the two summands: the factor $\mathbb C$. On $\Lambda^1W_2=\mathbf 2$, $\star^2=-1$: the factor $\mathbb H$, and $\nu_{\mathrm{sgn}}(\mathbf 2)=-1$ in Theorem~\ref{thm:AF} is this sign. On $\Lambda^1W_3\oplus\Lambda^2W_3=\mathbf 3\oplus\mathbf 3'$ the Hodge star exchanges the two summands: the factor $M_3(\mathbb C)$. Hence $A_F=\mathbb R^{\mathrm{sgn}}[S_4]$ is the algebra of the operators of $\mathbb C[S_4]$ on $\Lambda^\bullet W_2\oplus\Lambda^\bullet W_3$ that commute with Hodge duality composed with complex conjugation.
\end{corollary}
\begin{proof}
Theorem~\ref{thm:hodge} with $n=2$ and $n=3$. All five irreducible representations of $S_4$ occur, so that the image is faithful.
\end{proof}

The algebra $A_F=\mathbb C\oplus\mathbb H\oplus M_3(\mathbb C)$ is the finite algebra of the noncommutative-geometric description of the Standard Model by Connes and by Chamseddine, Connes and Marcolli~\cite{connes96,ccm07,cc08,vs15}. The next theorem recovers this known algebra (its nature is \emph{reconstructed}, Section~\ref{sec:status}); what is new is its origin, as a twisted group algebra of an automorphism group of FHTs, and the uniqueness of the pair $(S_4,w_1^2)$.

\begin{lemma}[Real cocycles]\label{lem:realcocycle}
Let $G$ be a finite group and $c:G\times G\to\mathbb R^\times$ a $2$-cocycle for the trivial action of $G$ on $\mathbb R^\times$ (its values are non-zero). Then $c$ is cohomologous to the $\{\pm1\}$-valued cocycle $\operatorname{sgn}\circ c$, and $\mathbb R^{c}[G]\cong\mathbb R^{\operatorname{sgn}\circ c}[G]$.
\end{lemma}
\begin{proof}
The maps $x\mapsto\operatorname{sgn}x$ and $x\mapsto\lvert x\rvert$ are homomorphisms, so $s=\operatorname{sgn}\circ c$ and $\lvert c\rvert$ are cocycles, and $c=s\lvert c\rvert$. The function $f=\log\lvert c\rvert$ is a $2$-cocycle with values in the trivial module $\mathbb R$: $f(h,k)-f(gh,k)+f(g,hk)-f(g,h)=0$. Summing over $k$ and putting $b(g)=\card G^{-1}\sum_kf(g,k)$ gives $f(g,h)=b(g)+b(h)-b(gh)$, so $\lvert c\rvert$ is the coboundary of $e^b$ and $c$ is cohomologous to $s$. Finally, if $c'(g,h)=c(g,h)\lambda(g)\lambda(h)\lambda(gh)^{-1}$, then $u_g\mapsto\lambda(g)u_g$ is an isomorphism $\mathbb R^{c'}[G]\to\mathbb R^{c}[G]$.
\end{proof}

\begin{theorem}[The algebra $A_F$ as a twisted group algebra]\label{thm:AF}
$\mathbb R^{\mathrm{sgn}}[S_4]\cong\mathbb C\oplus\mathbb H\oplus M_3(\mathbb C)$, the finite algebra $A_F$ of Connes' description of the Standard Model. Conversely, if $\mathbb R^{c}[G]\cong A_F$ for a finite group $G$ and a $2$-cocycle $c:G\times G\to\{\pm1\}$ (the real algebra with basis $(u_g)$ and product $u_gu_h=c(g,h)u_{gh}$), then $G\cong S_4$ and $c$ is cohomologous to $\mathrm{sgn}\cup\mathrm{sgn}$, the cocycle of Definition~\ref{def:twisted}: the class $w_1^2$. (The cocycles may be assumed normalized; the conclusion holds after transport by an isomorphism $G\cong S_4$; the uniqueness concerns the abstract group and the cohomology class, not a permutation action, nor an FHT, nor a representative cocycle. By Lemma~\ref{lem:realcocycle}, the same holds for real cocycles $c:G\times G\to\mathbb R^\times$.)
\end{theorem}
\begin{proof}
The irreducible representations of $S_4$ are $1,\mathrm{sgn},\mathbf 2,\mathbf 3,\mathbf 3'=\mathbf 3\otimes\mathrm{sgn}$, all real. With $\eta=\mathrm{sgn}$: $\rho\mapsto\bar\rho\otimes\eta$ exchanges $1\leftrightarrow\mathrm{sgn}$ and $\mathbf 3\leftrightarrow\mathbf 3'$, which gives $\mathbb C$ and $M_3(\mathbb C)$, and fixes $\mathbf 2$ since $\mathbf 2\otimes\mathrm{sgn}\cong\mathbf 2$. Summing $\mathrm{sgn}(g)\chi_{\mathbf 2}(g^2)$ over the classes of $S_4$ (identity: $2$; $6$ transpositions: $-12$; $3$ double transpositions: $+6$; $8$ three-cycles: $-8$; $6$ four-cycles: $-12$) gives $-24$, so $\nu_\eta(\mathbf 2)=-1$ and $\mathbf 2$ gives $\mathbb H$. The dimensions are $2+4+18=24$.

Conversely, $\dim_{\mathbb R}\mathbb R^{c}[G]=\card G=24$ and $\mathbb C^c[G]\cong A_F\otimes_{\mathbb R}\mathbb C\cong\mathbb C^2\oplus M_2(\mathbb C)\oplus M_3(\mathbb C)^2$. A factor of degree $1$ is a one-dimensional projective representation with cocycle $c$, so $c$ is a coboundary with values in $\mathbb C^\times$ and $\mathbb C^c[G]\cong\mathbb C[G]$: $G$ has five irreducible representations, of degrees $1,1,2,3,3$. Then $G/G'$ has order $2$ and $\card{G'}=12$; an irreducible representation of odd degree stays irreducible on a subgroup of index $2$ (by Clifford's theorem, its restriction to a normal subgroup of index $2$ is irreducible or the sum of two conjugate irreducible representations of the same degree, which is impossible in odd degree), so $G'$ has an irreducible representation of degree $3$, and $G'\cong A_4$, the only group of order $12$ with one~\cite{dummitfoote} (the other groups of order $12$, $C_{12}$, $C_2\times C_6$, $D_6$ and $\mathrm{Dic}_3$, have an abelian subgroup of index $2$, hence irreducible representations of degree at most $2$). The centraliser of $A_4$ in $G$ meets $A_4$ in its trivial centre; if it were non-trivial, $G$ would be $A_4\times C_2$, whose derived subgroup is $V_4$. Hence $G$ embeds in $\Aut(A_4)$, and $\Aut(A_4)\cong S_4$: conjugation gives an injective map $S_4\to\Aut(A_4)$, since the centraliser of $A_4$ in $S_4$ is trivial, and an automorphism is determined by the images of $a=(123)$ and $b=(12)(34)$, with $b\mapsto V_4\setminus\{1\}$ ($V_4$ is characteristic) and $a\mapsto$ an element of order $3$, so $\card{\Aut(A_4)}\le3\cdot8=24$. So $G\cong S_4$. Finally, by exactness of the long sequence of $1\to\{\pm1\}\to\mathbb C^\times\xrightarrow{z\mapsto z^2}\mathbb C^\times\to1$, the classes of $H^2(S_4;\{\pm1\})$ that vanish in $H^2(S_4;\mathbb C^\times)$ are the images under the connecting map of the characters $S_4\to\mathbb C^\times$, that is of $1$ and $\mathrm{sgn}$; the character $\mathrm{sgn}$, with square root $g\mapsto i^{\bar e(g)}$, gives the class of $(-1)^{\bar e(g)\bar e(h)}$, that is $w_1^2$. The class $0$ gives $\mathbb R[S_4]=\mathbb R^2\oplus M_2(\mathbb R)\oplus M_3(\mathbb R)^2\not\cong A_F$.
\end{proof}

\begin{remark}[Algebras and $*$-algebras]\label{rem:star}
The isomorphisms of this section are isomorphisms of real algebras. For a normalized $\{\pm1\}$-valued cocycle $c$, the algebra $\mathbb R^{c}[G]$ carries the involution $u_g^*=u_g^{-1}$, extended $\mathbb R$-linearly and antimultiplicatively; in the left regular representation the $u_g$ act by signed permutation matrices, which are orthogonal, so $*$ is the transpose and is positive. For a cocycle with values in $\mathbb R^\times$, the involution is transported from this signed model by the isomorphism of Lemma~\ref{lem:realcocycle}. A finite-dimensional real algebra with a positive involution is a direct sum of algebras $M_n(\mathbb R)$, $M_n(\mathbb C)$, $M_n(\mathbb H)$ with the conjugate-transpose involution, up to $*$-isomorphism. Two such algebras that are isomorphic as real algebras are $*$-isomorphic, and the identification of Theorem~\ref{thm:AF} holds with the usual involution of $A_F$.
\end{remark}

\begin{remark}[Choice of the character]\label{rem:whicheta}
A flat FHT carries several canonical characters of its automorphism group: the sign $\mathrm{sgn}_V$ of the permutation of the vertices (the orientation character of the ambient simplex), and the orientation characters $\det\tilde H_k(\Delta(\F);\mathbb R)$ of its homology groups (Theorem~\ref{thm:hodge}). They need not agree. In this article the twist is fixed once for all as $\eta_V=\mathrm{sgn}_V$ (Corollary~\ref{cor:etaV}); the other characters appear only in computations where they are named (Section~\ref{sec:status}). Other natural candidates are the sign $\mathrm{sgn}_E$ of the induced permutation of the edges of the realization, and the sign on the children of the root. For $\Aut(\F)\cong S_4$ each of $\mathrm{sgn}_V$, $\mathrm{sgn}_E$ and the sign on the children is either trivial or equal to $\mathrm{sgn}$. For the $63$ isomorphism classes of FHTs with $\Aut(\F)\cong S_4$ and $v\le6$ (Proposition~\ref{prop:census6}), $\mathrm{sgn}_V$ is non-trivial for $59$ of them and trivial for the $4$ classes in which $S_4$ acts transitively on the six vertices. In the census below it is non-trivial for $1\,216$ of the $1\,343$ classes with $v=7$. For $S_4$, whichever of these is used, the algebra is the same once the character is non-trivial: Theorem~\ref{thm:AF} determines it from the abstract group. For larger groups the non-trivial character is no longer unique, and the sector $A_F$ depends on it: for $S_4\times S_3$ acting on $4+3$ points, $\eta=\mathrm{sgn}\otimes1$ and $\eta=1\otimes\mathrm{sgn}$ both give a sector $A_F$, the second without twist on $S_4$, whereas $\eta=\mathrm{sgn}_V=\mathrm{sgn}\otimes\mathrm{sgn}$ gives none (Section~\ref{sec:verif}). The abstract sign of $S_4$ and the sign of its action on the vertices must be distinguished explicitly.
\end{remark}

\begin{corollary}[A uniform rule for the twist; {[T]}, D]\label{cor:etaV}
For an FHT $\F$ on $V$, let $\eta_V=\det R^V$, where $R^V$ is the real permutation representation of $\Aut\F$ on $\mathbb R^V$; so $\eta_V(\sigma)$ is the sign of the permutation $\sigma$ of the vertices. Then
\[
\mathbb R^{\eta_V}[\Aut\F]\cong A_F\iff\Aut\F\cong S_4\ \text{and some automorphism of $\F$ induces an odd permutation of $V$.}
\]
\end{corollary}
\begin{proof}
If $\Aut\F\cong S_4$ and $\eta_V$ is non-trivial, then $\eta_V=\mathrm{sgn}$, the only non-trivial character of $S_4$, and Theorem~\ref{thm:AF} gives $A_F$. Conversely, if $\mathbb R^{\eta_V}[\Aut\F]\cong A_F$, Theorem~\ref{thm:AF} gives $\Aut\F\cong S_4$; if $\eta_V$ were trivial, the algebra would be $\mathbb R[S_4]$, which has the simple factor $\mathbb R$ of the trivial representation, so $\eta_V$ is non-trivial.
\end{proof}

The rule is computed on the FHT; it is not chosen after the group is known. The four classes with $\Aut\F\cong S_4$ acting transitively on six vertices, where $\eta_V$ is trivial, give $\mathbb R[S_4]$ and not $A_F$: they are counterexamples to the reading ``$\Aut\F\cong S_4$ gives $A_F$'', not exceptions to the rule. At $v=7$ the rule gives $A_F$ for $1\,216$ of the $1\,343$ classes with $\Aut\F\cong S_4$ (Table~\ref{tab:census}).

\begin{definition}[Simple factors, sectors, whole algebra]\label{def:sector}
Let $B$ be a finite-dimensional semisimple real algebra. A \emph{simple factor} of $B$ is a simple two-sided ideal which is a direct factor of $B$. A \emph{sector} of $B$ is an algebra $eB$, where $e$ is a central idempotent of $B$; equivalently, a direct sum of simple factors. Thus $B$ has a sector isomorphic to an algebra $A$ if and only if $B\cong A\oplus B'$ for an algebra $B'$, and $B$ is \emph{whole} $A$ if $B\cong A$ (the case $e=1$, $B'=0$). In particular $B$ has a sector isomorphic to $A_F$ if and only if $\mathbb C$, $\mathbb H$ and $M_3(\mathbb C)$ occur among its simple factors.
\end{definition}

A subalgebra isomorphic to $A_F$ need not be a sector. For $H\le G$, the span of the $u_h$, $h\in H$, is a subalgebra of $\mathbb R^\eta[G]$ isomorphic to $\mathbb R^{\eta|_H}[H]$ (Definition~\ref{def:twisted}); so $\mathbb R^{\mathrm{sgn}}[S_5]$ contains $\mathbb R^{\mathrm{sgn}}[S_4]=A_F$ as a subalgebra, but has no sector $A_F$, since it has no quaternionic factor (Proposition~\ref{prop:spinsector}(iii)). Conversely, a sector $A_F$ may occur without a subgroup $S_4$: $GL(2,3)$ has no subgroup isomorphic to $S_4$ (its abelianization is $\{\pm1\}$, through $\det$, so its only subgroup of index $2$ is $SL(2,3)\not\cong S_4$), yet $\mathbb R^{\det}[GL(2,3)]$ has the sector $A_F$ (Proposition~\ref{prop:spinsector}(i)). Containing $S_4$ is neither sufficient nor necessary for a sector $A_F$.

\subsection{A census of the classes with an \texorpdfstring{$A_F$}{A\_F} sector}\label{sec:census}

For each $v\le6$ we know all classes $(v,\alpha)$ (Table~\ref{tab:valpha}), and we computed $\mathbb R^{\eta}[\Aut\F]$ for every $\eta\in\mathrm{Hom}(G,\{\pm1\})$, $G$ running over the representative groups of the $54$ classes $(v,\alpha)$ by the formula of Theorem~\ref{thm:twistedWedderburn} (the decomposition was also compared with that of the central extension $\hat G$, with $\mathbb R[\hat G]=\mathbb R[G]\oplus\mathbb R^{\eta}[G]$; this uses the same character tables and is a check of type (S), Section~\ref{sec:verification}). The result is the following.

\begin{proposition}[Computed, $v\le6$]\label{prop:census6}
Among the $54$ classes with $v\le6$ and all $\eta\in\mathrm{Hom}(\Aut\F,\{\pm1\})$, $\mathbb R^{\eta}[\Aut\F]$ has a sector isomorphic to $A_F$ (Definition~\ref{def:sector}) only in the following cases; at these sizes the search covers all automorphism groups. (i) $\Aut\F\cong S_4$ (the four rows $S_4$ of Table~\ref{tab:valpha}; $3,9,4,47$ isomorphism classes, $63$ in all) and $\eta=\mathrm{sgn}$: $\mathbb R^{\eta}[S_4]=A_F$ exactly. (ii) $\Aut\F\cong S_4\times C_2$ ($9$ and $57$ isomorphism classes, $66$ in all) and $\eta=\mathrm{sgn}\otimes1$: $\mathbb R^{\eta}[S_4\times C_2]=A_F\otimes\mathbb R[C_2]=A_F\oplus A_F$. In no other class do $\mathbb H$ and $M_3(\mathbb C)$ occur together as simple factors.
\end{proposition}

For $v\ge7$ the number of isomorphism classes ($490\,083\,358$ for $v=7$, Table~\ref{tab:fht_counts}) is too large for an exhaustive enumeration with the tools used here. We restricted the search to the FHTs whose automorphism group \emph{contains} a subgroup $S_4$. By the uniqueness in Theorem~\ref{thm:AF}, this search finds all the FHTs whose \emph{whole} twisted algebra is $A_F$, as far as the generation is exhaustive at the size considered; it only gives a restricted census of the \emph{sectors} $A_F$ of larger algebras, since containing $S_4$ is neither sufficient nor necessary for such a sector. Size by size: at $v=7$ all the faithful actions of $S_4$ were processed, but the algebras of the $214$ classes whose automorphism group has an order other than $24$, $48$ and $144$ were not analysed (orders $120$, $168$, $240$, $720$, $5\,040$); at $v=8$, nine of the ten actions were processed; at $v=9$, one action of thirteen. A faithful action of $S_4$ on $v$ points is a multiset of orbits $S_4/K$, $K$ running over the eleven conjugacy classes of subgroups, with trivial total kernel; there are $1,1,4,5,10,13$ of them for $v=4,\dots,9$. For each action the $S_4$-invariant FHTs are generated by the recursive decomposition of Proposition~\ref{prop:recursive} (an invariant partition into blocks and an invariant connected clutter on the blocks; then, for each orbit of blocks, a representative block $B$ and an FHT on $B$ invariant under the stabilizer of $B$, transported to the other blocks of the orbit; the invariance under the stabilizer makes the transport independent of the choice of the transporting elements). $\Aut(\F)$ is computed by backtracking. The same FHT can be reached from several subgroups, or several embeddings, isomorphic to $S_4$; the final deduplication identifies these occurrences, isomorphism classes being decided by a canonical form under all the relabellings of $V$, not only under the permutations that normalize a given copy of $S_4$. The method is validated by $v\le6$: it returns exactly the $3,9,4+47$ classes with $\Aut\cong S_4$, the $9+57$ with $S_4\times C_2$, the $4+14$ with $S_5$ and the $5$ with $S_6$ of Table~\ref{tab:valpha}.

\begin{table}[ht]
\centering
\begin{tabular}{@{}lrrrr@{}}
\toprule
$v$ & $4$ & $5$ & $6$ & $7$ \\
\midrule
actions of $S_4$ on $v$ points & 1 & 1 & 4 & 5 \\
FHTs with $\Aut\F\cong S_4$ & 3 & 9 & 51 & 1\,343 \\
\quad of which $\mathrm{sgn}_V\ne1$ & 3 & 9 & 47 & 1\,216 \\
FHTs with $\Aut\F$ of order $48$ and $\mathbb R^{\eta}[\Aut\F]\cong A_F\oplus A_F$ & 0 & 0 & 66 & 1\,227 \\
FHTs with $\Aut\F$ of order $144$ and $\mathbb R^{\eta}[\Aut\F]\cong A_F\oplus A_F\oplus M_2(A_F)$ & 0 & 0 & 0 & 143 \\
\bottomrule
\end{tabular}
\caption{Classes with an $A_F$ sector: isomorphism classes of FHTs whose automorphism group contains $S_4$, by the twisted algebra. The lines on a twisted algebra count the FHTs admitting at least one character $\eta$ that gives the stated decomposition, here $\eta=\mathrm{sgn}$ on the factor $S_4$ and trivial on the other factor; for $\eta=\mathrm{sgn}_V$ the answer depends on the parity of the action of the other factor on the vertices (there is no sector $A_F$ when it is odd, Section~\ref{sec:verif}) and was not counted. At $v=7$ the groups of order $48$ are $S_4\times C_2$ and the group of order $144$ is of type $S_4\times S_3$ acting on $4+3$ points. The $214$ classes at $v=7$ whose automorphism group has an order other than $24$, $48$, $144$ were not analysed.\label{tab:census}}
\end{table}

The groups larger than $S_4$ in Table~\ref{tab:census} are determined abstractly. A group of order $48$ containing $S_4$ contains it as a normal subgroup of index $2$; since $S_4$ has trivial centre and no outer automorphism, the group is $S_4\times C_2$. A subgroup of order $144$ of $S_7$ containing $S_4$ is $S_4\times S_3$ acting on $4+3$ points: no transitive group of degree $7$ has order $144$; $S_6$ has no subgroup of index $5$ (the kernel of the action on the five cosets is a normal subgroup contained in the subgroup, of order $144$; the normal subgroups of $S_6$ are $1$, $A_6$ and $S_6$, so the kernel is trivial and $S_6$ would embed in $S_5$, of order $120<720$); and for every other partition of $7$ into orbits, the order of the group of permutations preserving it is not divisible by $144$, except for $4+3$, where it equals $144$. So for these groups the character $\eta=\mathrm{sgn}$ on the factor $S_4$, trivial on the other factor, gives $\mathbb R^\eta[S_4\times C_2]=A_F\otimes\mathbb R[C_2]=A_F\oplus A_F$ and $\mathbb R^\eta[S_4\times S_3]=A_F\otimes\mathbb R[S_3]=A_F\oplus A_F\oplus M_2(A_F)$, the cocycle being the product of that of $S_4$ and of the trivial one.

\begin{proposition}[Computed, $v=7$ and partial $v=8$]\label{prop:census7}
(i) For $v=7$ there are exactly $1\,343$ isomorphism classes of FHTs with $\Aut\F\cong S_4$ ($51$ for the action $6+1$, $76$ for $4+3$, $1\,216$ for $4+1+1+1$), each with $\mathbb R^{\mathrm{sgn}}[\Aut\F]=A_F$; and $1\,227$ with an automorphism group of order $48$ whose twisted algebra, for $\eta=\mathrm{sgn}\otimes1$, is $A_F\oplus A_F$; the $214$ classes whose automorphism group has an order other than $24$, $48$ and $144$ (orders $120$, $168$, $240$, $720$, $5\,040$) were not analysed. (ii) For $v=8$, nine of the ten faithful actions of $S_4$ on eight points were processed (the action $4+1+1+1+1$ is not included in this count). They give $10\,618$ isomorphism classes of FHTs with $\Aut\F\cong S_4$: $26$ for the transitive action $8$, $88$ and $20$ for the two actions $6+2$, $902$ for $4+4$, $2\,573$ for one of the two actions $6+1+1$ and none for the other, and $7\,009$ for $4+3+1$ (the actions $4+2+2$ and $4+2+1+1$ give none, their automorphism groups being larger). This is a lower bound for the number of such classes at $v=8$. For $v=9$ the single action $8+1$ gives $3\,104$ classes (one action of thirteen). By Theorem~\ref{thm:AF}, $\mathbb R^{\mathrm{sgn}}[\Aut\F]=A_F$ for each of them.
\end{proposition}

\begin{remark}[Reading of the census]\label{rem:census}
The algebra $A_F$ is \emph{not} specific to the first FHT where it appears. It occurs at every $v\ge4$: starting from the star $\St_4$, add recursively a new root whose two children are the previous FHT and a new leaf, with a host made of a single edge; the two children are not isomorphic, so the automorphism group stays $S_4$ (Proposition~\ref{prop:aut}), and $\mathrm{sgn}_V$ stays non-trivial since the new leaf is fixed. The number of classes grows quickly ($3,9,51,1\,343$ for $v=4,\dots,7$); the twisted algebra only depends on the abstract group $\Aut\F\cong S_4$. Hence the algebra alone does not select an FHT: a selection needs a module on which $A_F$ acts.
\end{remark}

\subsection{Moves and classes}\label{sec:movesclasses}

A move changes $\Aut(\F)$ in an uncontrolled way: for instance $\Th_2^{\pm}$ replaces a child by its pieces and may break or create symmetries, and $\Omega_1$ changes $v$. Three kinds of quantities must be distinguished. \emph{Quantities preserved by moves}: the interior moves preserve the number $v$ of vertices, but in general neither $\Aut(\F)$ nor $\kappa$; $\kappa$ is preserved exactly by the level moves (fusions with $\card{\Abs(e)}=1$, fissions with $\card B=1$) and by the regular $\Omega_1^{\pm}$, which change $v$ (Proposition~\ref{prop:kappa}). \emph{Invariants under relabelling}: the pair $(v,\alpha)$ and the function $\kappa$ are invariant under $S_V$, hence descend to the quotient graph $S_V\backslash\TG(V)$, which is connected. \emph{Gradings}: $\kappa$ grades the poset $\mathcal P(V)$ and measures the descent to the star (Sections~\ref{sec:kappa} and~\ref{sec:descent}); it is not an invariant of the moves. Hence the quotient poset $\overline{\mathcal P}(v)$ of Proposition~\ref{prop:isotropy}(iii) carries at the same time the order, the class $(v,\alpha)$ and the fibre multiplicity $n(v,\alpha)$.

\subsection{Whole twisted algebras and sectors}\label{sec:wholetwisted}

Theorem~\ref{thm:AF} has a closure counterpart: among the algebras explicitly listed in Corollary~\ref{cor:B2} (the algebra of the noncommutative-geometric Standard Model and some of its usual extensions), $A_F$ is the only one that is a \emph{whole} twisted group algebra $\mathbb R^\eta[G]$ for a sign character $\eta$ (Remark~\ref{rem:closurescope} for general cocycles).

\begin{proposition}[A commutative factor; {[T]}, D]\label{prop:B1}
For every finite group $G$ and every homomorphism $\eta:G\to\{\pm1\}$, the algebra $\mathbb R^{\eta}[G]$ has a commutative simple factor, $\mathbb R$ if $\eta=1$ and $\mathbb C$ if $\eta\ne1$.
\end{proposition}
\begin{proof}
Apply Theorem~\ref{thm:twistedWedderburn} to the trivial representation $\rho=1$. If $\eta=1$, then $\bar\rho\otimes\eta=1$ and $\nu_\eta(1)=1$, which gives a factor $\mathbb R$. If $\eta\ne1$, then $\bar\rho\otimes\eta=\eta\not\cong1$, and the pair $\{1,\eta\}$ gives a factor $M_1(\mathbb C)=\mathbb C$.
\end{proof}

\begin{corollary}[Closure of the ladder; {[T]}, D]\label{cor:B2}
None of the algebras $\mathbb H\oplus\mathbb H\oplus M_4(\mathbb C)$ (Pati--Salam~\cite{ps74,ccvs13ps}), $M_2(\mathbb H)\oplus M_4(\mathbb C)$ (the case $k=2$ of the classification of Chamseddine and Connes~\cite{cc08}, starting point of the grand symmetry), the grand algebra $M_4(\mathbb H)\oplus M_8(\mathbb C)$ of~\cite{dlm14} and $\mathbb C\oplus\mathbb H\oplus\mathbb H\oplus M_3(\mathbb C)$ (left--right~\cite{ccm07}) is isomorphic to $\mathbb R^{\eta}[G]$ for a finite group $G$ and a character $\eta$. Among these algebras and $A_F$, only $A_F$ is a whole twisted group algebra, and only for $(S_4,\mathrm{sgn})$.
\end{corollary}
\begin{proof}
The first three have no commutative factor (Proposition~\ref{prop:B1}). For the last, $\dim_{\mathbb R}=28$ would be $\card G$, and the factor $M_3(\mathbb C)$ requires an irreducible complex representation of degree $3$ (Lemma~\ref{lem:central}), whose degree divides $\card G$; $3$ does not divide $28$. The last statement is Theorem~\ref{thm:AF}.
\end{proof}

\begin{remark}[Scope of the closure]\label{rem:closurescope}
Corollary~\ref{cor:B2} concerns the twisted algebras of sign characters, those of Definition~\ref{def:twisted}, which are the algebras attached to FHTs. With a general $2$-cocycle the closure fails at the top of the ladder: the Clifford algebra of four anticommuting generators of square $-1$ is $M_2(\mathbb H)$, a twisted algebra of $C_2^4$, and $M_2(\mathbb H)\otimes\mathbb R[C_3]=M_2(\mathbb H)\oplus M_4(\mathbb C)$ is a whole twisted algebra of $C_2^4\times C_3$, of dimension $48$. In the same way the Clifford algebra of six anticommuting generators of square $+1$ is $M_4(\mathbb H)$, and $M_4(\mathbb H)\otimes\mathbb R[C_3]=M_4(\mathbb H)\oplus M_8(\mathbb C)$ is a whole twisted algebra of $C_2^6\times C_3$, of dimension $192$.  The exclusion of the left--right algebra, on the other hand, holds for every cocycle with values in $\{\pm1\}$: after complexification its factor $\mathbb C$ gives a one-dimensional projective representation of cocycle $c$, so that $c$ is a complex coboundary and $\mathbb C^c[G]\cong\mathbb C[G]$ has an irreducible representation of degree $3$, whose degree must divide $\card G=28$. The uniqueness of Theorem~\ref{thm:AF}, by contrast, holds for all real cocycles (Lemma~\ref{lem:realcocycle}), up to isomorphism and cohomology.
\end{remark}

Some simple factors of the larger algebras do occur, as \emph{sectors} (Definition~\ref{def:sector}) of twisted algebras of larger groups. We consider two of them. The group $GL(2,3)$ is a double cover of $PGL(2,3)\cong S_4$ (the action on the four points of the projective line over $\mathbb F_3$), with centre $\{\pm I\}$; it is isoclinic to the binary octahedral group $2O\subset SU(2)$. The group $S_5$ contains $S_4$ as the stabiliser of a point.

\begin{proposition}[Spinorial sector and the colour--lepton factor; {[T]}]\label{prop:spinsector}
(i) The element $u_{-I}$ of the algebra $\mathbb R^{\det}[GL(2,3)]$ is central, and the summand on which it acts by $+1$ is $\mathbb R^{\mathrm{sgn}}[S_4]=A_F$.
(ii) The summand on which $u_{-I}$ acts by $-1$ (the \emph{spinorial sector}) is $\mathbb H\oplus\mathbb H\oplus M_2(\mathbb H)$; the factor $M_2(\mathbb H)$ acts on the irreducible representation of degree $4$; after identification of the semisimple sectors considered, it corresponds to the factor carried by the restriction to $2O$ of the spin-$\tfrac32$ representation. This is an identification of algebraic sectors; it does not result from a general correspondence between representations of isoclinic groups. Hence $\mathbb R^{\det}[GL(2,3)]\cong A_F\oplus\mathbb H\oplus\mathbb H\oplus M_2(\mathbb H)$.
(iii) $\mathbb R^{\mathrm{sgn}}[S_5]\cong\mathbb C\oplus M_4(\mathbb C)\oplus M_5(\mathbb C)\oplus M_6(\mathbb R)$; the factor $M_4(\mathbb C)$ acts on the standard representation $\mathbf 4$, whose restriction to $S_4$ is $\mathbf 1\oplus\mathbf 3$.
\end{proposition}
\begin{proof}
(i) $\det(-I)=1$, so $u_{-I}u_g=u_{-g}=u_gu_{-I}$ and $u_{-I}^2=u_I$: the elements $\tfrac12(1\pm u_{-I})$ are central idempotents. On the summand $u_{-I}=1$ the algebra is the twisted algebra of $GL(2,3)/\{\pm I\}=PGL(2,3)$ for the character induced by $\det$, which is well defined since $\det(-g)=\det g$. Under $PGL(2,3)\cong S_4$ this character is the sign: $\mathrm{diag}(1,-1)$ fixes $[1:0]$ and $[0:1]$ and exchanges $[1:1]$ and $[1:2]$. Theorem~\ref{thm:AF} concludes. (ii) Put $f(x)=\sum_{g^2=x}\det(g)$; then $\langle f,\chi_\rho\rangle=\nu_{\det}(\rho)$, so $f=\sum_\rho\nu_{\det}(\rho)\chi_\rho$. In $GL(2,3)$ the solutions of $g^2=I$ are $\pm I$ and the $12$ conjugates of $\mathrm{diag}(1,-1)$, and those of $g^2=-I$ are the $6$ matrices with characteristic polynomial $x^2+1$; hence $f(I)=-10$, $f(-I)=6$ and $\sum_{\rho(-I)=-1}\nu_{\det}(\rho)d_\rho=\tfrac12(f(I)-f(-I))=-8$. The group $GL(2,3)$ has $8$ conjugacy classes (by the rational canonical form: $I$, $-I$, the $12$ conjugates of $\mathrm{diag}(1,-1)$, the $6$ elements of characteristic polynomial $x^2+1$, two classes of $8$ elements of orders $3$ and $6$, and two classes of $6$ elements of order $8$); the $5$ irreducible representations trivial on $-I$ are those of $PGL(2,3)\cong S_4$, of degrees $1,1,2,3,3$, so the squares of the degrees of the $3$ faithful ones add up to $48-24=24$, and the only decomposition of $24$ into three squares is $4+4+16$: they have degrees $2,2,4$. As $\lvert\nu_{\det}\rvert\le1$, $\nu_{\det}=-1$ for all three. The same argument applies to the binary octahedral group $2O\subset SU(2)$ ($8$ classes, $2O/\{\pm1\}\cong S_4$): its faithful irreducible representations have degrees $2,2,4$. They are of quaternionic type: the natural representation on $\mathbb C^2=\mathbb H$, its product with the character of order $2$ of $2O$, and the restriction of spin $\tfrac32$, $\mathrm{Sym}^3\mathbb C^2$, which preserves a symplectic form and is irreducible on $2O$: for an element with eigenvalues $e^{\pm i\theta}$ its character is $2\cos3\theta+2\cos\theta$, which takes the values $4,-4,0,0,0,0,-1,1$ on the eight classes of $2O$, of sizes $1,1,6,12,6,6,8,8$ (orders $1,2,4,4,8,8,6,3$), so its norm is $(16+16+8+8)/48=1$; the spinorial sector of $\mathbb R^{\det}[GL(2,3)]$ is isomorphic to that of the untwisted algebra $\mathbb R[2O]$, the twist by $\det$ compensating the difference between the two double covers of $S_4$ (also checked by computer, Section~\ref{sec:verif}). The reading of the module of degree $4$ as spin $\tfrac32$ uses this identification of sectors, not a general correspondence between isoclinic groups. (iii) The pairs $\{1,\mathrm{sgn}\}$, $\{\mathbf 4,\mathbf 4'\}$, $\{\mathbf 5,\mathbf 5'\}$ are exchanged by $\rho\mapsto\bar\rho\otimes\mathrm{sgn}$ (complex factors), and $\mathbf 6\cong\mathbf 6\otimes\mathrm{sgn}$ has $\nu_{\mathrm{sgn}}(\mathbf 6)=+1$ (real factor): by Theorem~\ref{thm:hodge}(iii), with $W$ the standard representation of $S_5$, orthogonal of dimension $4$ and determinant $\mathrm{sgn}$, the representation $\Lambda^2W\cong\mathbf 6$ is irreducible: its character $\tfrac12(\chi_W(g)^2-\chi_W(g^2))$ takes the values $6,0,-2,0,0,0,1$ on the classes of $1$, $(12)$, $(12)(34)$, $(123)$, $(123)(45)$, $(1234)$, $(12345)$, of sizes $1,10,15,20,20,30,24$, so its norm is $(36+60+24)/120=1$; hence of multiplicity one and stable under $C$, and $n=4\equiv0\pmod4$ gives the real type (likewise $\Lambda^0W\oplus\Lambda^4W=\mathbf 1\oplus\mathrm{sgn}$ and $\Lambda^1W\oplus\Lambda^3W=\mathbf 4\oplus\mathbf 4'$ give the factors $\mathbb C$ and $M_4(\mathbb C)$); $2+32+50+36=120$. The restriction of $\mathbf 4$ is the permutation module of four points minus the trivial module plus the trivial module of the fixed point, $\mathbf 1\oplus\mathbf 3$.
\end{proof}

\begin{remark}
Some of the simple factors of the algebras of Corollary~\ref{cor:B2} already occur as sectors in the two examples above: $\mathbb H$ and $M_2(\mathbb H)$ in $\mathbb R^{\det}[GL(2,3)]$ ($GL(2,3)$ acting linearly on $\mathbb F_3^2$; on the nine points of the plane, $\det$ is the sign of the permutation of the vertices), $M_4(\mathbb C)$ in $\mathbb R^{\mathrm{sgn}}[S_5]$. Together with $A_F$ this covers the simple factors of the Pati--Salam, left--right and $M_2(\mathbb H)\oplus M_4(\mathbb C)$ algebras, but not the factors $M_4(\mathbb H)$ and $M_8(\mathbb C)$ of the grand algebra, for which only the general-cocycle construction of Remark~\ref{rem:closurescope} is given here.  Corollary~\ref{cor:B2} is a statement about twisted group algebras of sign characters; it bears on FHTs because the algebra attached to an FHT, equipped with the chosen character $\eta$, is of this form, $\mathbb R^{\eta}[\Aut\F]$. It implies that none of the extensions listed there can be the whole twisted algebra of the automorphism group of an FHT. The corollary excludes these algebras as whole twisted algebras; it does not exclude other forms of occurrence, such as sectors or subalgebras. This is a restriction, not an existence statement.
\end{remark}

\section{Representability}\label{sec:representability}

This section asks which symmetries and which homologies flat FHTs can carry. Every finite group is the automorphism group of a flat FHT. A condition of Euler type, the Burnside condition, characterizes the graded rational modules that are the homology of a finite chain complex of permutation modules, whose differentials are arbitrary module maps. A faithful action on a finite connected simplicial complex can be rigidified into a flat FHT with this automorphism group and the same homology. The Burnside condition is not the only constraint on the homology of a flat FHT: a flat FHT is connected, so its reduced homology vanishes in degree $0$. We do not prove that the algebraic condition suffices for a finite simplicial realisation. For a finite group $G$ write $R_{\mathbb Q}(G)$ for the Grothendieck group of finite-dimensional $\mathbb Q G$-modules, $A(G)$ for the Burnside ring of finite $G$-sets and $\beta:A(G)\to R_{\mathbb Q}(G)$, $[X]\mapsto[\mathbb Q X]$. For a graded module $M=\bigoplus_dM_d$ put $\chi_G(M)=\sum_d(-1)^d[M_d]\in R_{\mathbb Q}(G)$.

\begin{theorem}[Representability; {[T]}, D; (iv) B]\label{thm:repr}
Let $G$ be a finite group.
\begin{enumerate}
\item {[T]} There is a flat FHT $\F$ with $\Aut(\F)\cong G$.
\item {[T]} If $G$ acts on a finite simplicial complex $K$, then $\chi_G(\tilde H_*(K;\mathbb Q))\in\beta(A(G))$.
\item {[T]} Conversely, if a finite-dimensional graded $\mathbb QG$-module $M$ satisfies $\chi_G(M)\in\beta(A(G))$, there is a finite chain complex of permutation modules whose homology is $M$. Its differentials are arbitrary $\mathbb QG$-linear maps; it is not claimed to be the simplicial or cellular chain complex of a $G$-complex.
\item {[T], B} Let $G$ act faithfully on a finite connected simplicial complex $K$. Then there are a flat FHT $\F$, whose complex of facets $L$ contains $K$, and an action of $G$ on $L$ that extends the given action on $K$ and is identified with $\Aut(\F)\cong G$, such that the inclusion $K\subset L$ is a homotopy equivalence. In particular $\tilde H_*(K;\mathbb Z)\cong\tilde H_*(\F;\mathbb Z)$ as graded $\mathbb ZG$-modules, and the same holds with rational coefficients.
\end{enumerate}
\end{theorem}
\begin{proof}
(i) By Frucht's theorem~\cite{frucht39} there is a finite connected graph $\Gamma$ with $\Aut(\Gamma)\cong G$, which can be taken without isolated vertices. Its edges form a connected clutter covering its vertices, that is a flat FHT, and the automorphisms of this clutter are those of $\Gamma$.
(ii) Pass to the barycentric subdivision $K'$ of $K$, on which $G$ acts simplicially, with $|K'|\cong|K|$ equivariantly. This action has no inversion: an element that stabilizes a simplex of $K'$, that is a chain $\sigma_0\subsetneq\dots\subsetneq\sigma_d$ of simplices of $K$ of distinct dimensions, fixes each $\sigma_i$, hence fixes the simplex vertex by vertex. So no orientation character occurs: the chain module of degree $d$ of $K'$ is the permutation module $\mathbb Q[K'_d]$, $K'_d$ being the set of $d$-simplices, and the augmentation module is $\mathbb Q=\beta[G/G]$ in degree $-1$. The equivariant Euler characteristic of the augmented chain complex equals that of its homology (Hopf trace formula over the semisimple algebra $\mathbb QG$), which is the reduced homology of $K$: $\chi_G(\tilde H_*(K;\mathbb Q))=\sum_{d\ge-1}(-1)^d\beta[K'_d]\in\beta(A(G))$, with $K'_{-1}$ a point.
(iii) $\mathbb QG$ is semisimple. Let $M$ live in degrees $a\le d\le b$. Choose inductively $C_d=M_d\oplus N_d\oplus N_{d-1}$ with $N_{a-1}=0$, the differential mapping $N_{d-1}\subset C_d$ isomorphically onto $N_{d-1}\subset C_{d-1}$ and vanishing elsewhere; for $d<b$ choose $N_d$ such that $C_d$ is a free module. At $d=b$, put $N_b=0$; then $[C_b]=\pm\big(\chi_G(M)-\sum_{d<b}(-1)^d[C_d]\big)\in\beta(A(G))$, so $[C_b]=\beta(x)-\beta(y)$ with $x,y$ finite $G$-sets. Adding the acyclic complex $\mathbb Qy\xrightarrow{=}\mathbb Qy$ in degrees $b+1,b$ puts the permutation module $\mathbb Qy$ in degree $b+1$ and makes $[C_b\oplus\mathbb Qy]=[\mathbb Qx]$, hence $C_b\oplus\mathbb Qy\cong\mathbb Qx$ by semisimplicity. (Degrees $b,b-1$ would also do when $b>a$; degrees $b+1,b$ keep all the degrees $\ge a$, in particular non-negative when $a\ge0$, also for $b=a$.)
(iv) We use the rigidification of Costoya, Gomes and Viruel~\cite{cgv25} (Theorem~C of the introduction; Definition~3.1, Proposition~3.3 and Theorem~3.4 in the body). For $G\ne1$, choose a path $P$ through all the vertices and put $L=R_G(K^+;P)=K^+\cup\bigcup_{g\in G}B_g$, where $B_g$ is a copy of a fixed contractible complex $B$ whose spine, an abstract path $P_B=(x_0,\dots,x_m)\subset B$, is identified with the path $g\cdot P$ of $K^+$. Two assertions must be separated. (a) \emph{The automorphism group}, taken from the source: by~\cite[Prop.~3.3]{cgv25}, restriction to $K^+$ (the complex on which the construction is performed, see below) is an injective homomorphism $\operatorname{aut}(L)\to\operatorname{aut}(K^+)$ with image $G$, and $h\in G$ acts on $L$ by its action on $K^+$ and by $B_g\mapsto B_{hg}$: the action of $G$ on $L$ extends the given action on $K$, and it is identified with $\operatorname{aut}(L)$. (b) \emph{The homotopy type}: the inclusion $K\subset L$ is a homotopy equivalence. The source describes $\lvert L\rvert$ as a pushout~\cite[Prop.~3.3]{cgv25} but does not treat the identifications produced by a path with repeated vertices; we give the argument here.

\emph{Choice of the path.} Replace first $K$ by the complex $K^+$ obtained by adding, for every edge $e=\{a,b\}$ of $K$, a new vertex $m_e$ and the triangle $\{a,b,m_e\}$, with $g\cdot m_e=m_{g\cdot e}$. The action on $K^+$ is simplicial and faithful, the triangles collapse onto their edges, and $K\subset K^+$ is a $G$-equivariant homotopy equivalence; the construction is applied to $K^+$, which contains $K$. In $K^+$ there is a path through all the vertices in which any three consecutive points are pairwise distinct. Take a spanning tree $T$ of the $1$-skeleton of $K$ and a depth-first traversal of $T$: go down a tree edge $\{a,b\}$ through $m_{\{a,b\}}$ (that is, $a,m_{\{a,b\}},b$), come back up directly ($b,a$), and, at the first visit of a vertex $a$, insert for every edge $e=\{a,b\}\notin T$ chosen at this endpoint the detour $a,m_e,b,a$. Each $m_e$ occurs once, a return $b,a$ is preceded by a vertex different from $a$ and followed by a vertex different from $b$, so no pattern $x,y,x$ occurs.

\emph{The pushout.} In the band of~\cite[Def.~2.11 and~2.12]{cgv25}, two points of the spine lie in a common simplex only if they are consecutive, and a vertex outside the spine is adjacent to at most three consecutive points $x_{i-1},x_i,x_{i+1}$ of the spine: in $U_i$ the vertices $y_i$, $w_i$, $v_j$ are adjacent to $x_{i-1}$ and $x_i$ only, the points of the attached paths to no point of the spine, and the band adds the $2$-simplices $\{x_i,w_i,x_{i+1}\}$ and the edges $\{w_i,w_{i+1}\}$, so that $w_i$ is also adjacent to $x_{i+1}$. Hence the spine is a full subcomplex of $B$: a simplex with all its vertices on the spine is a vertex or an edge of $P_B$. A simplex $\sigma$ of $B$ not in $P_B$ has a vertex $u$ outside the spine, which is identified with nothing, and its vertices on the spine lie among three consecutive points, whose images in $K^+$ are pairwise distinct by the choice of the path. So $\sigma$ is not collapsed, and it is identified with no other simplex: a simplex with the same image has the same vertices outside the spine, hence its vertices on the spine lie among the same three points, and their images determine them. The only identifications are those of the vertices and edges of $P_B$ with those of the path $g\cdot P$, consecutive points of a path being joined by an edge. So $\lvert L\rvert$ is the adjunction space $K^+\cup_{P_B}\bigcup_gB_g$, attached along the maps $P_B\to K^+$, which need not be injective.

\emph{Proof of (b).} $L$ is obtained from $K^+$ by attaching the bands one after the other, each as the pushout of $K^+\leftarrow P_B\hookrightarrow B_g$. The inclusion $P_B\hookrightarrow B_g$ is the inclusion of a subcomplex, a cofibration, and a homotopy equivalence since both are contractible. The pushout of a cofibration that is a homotopy equivalence is again one (gluing lemma), so each attachment preserves the homotopy type; no contractibility of the image of the spine is used. With $K\subset K^+$, the inclusion $K\subset L$ is a homotopy equivalence. For $G=1$ the bands are replaced by the contractible complexes of~\cite{cgv25}, attached at vertices, with the same argument.

Since the inclusion is $G$-equivariant and a homotopy equivalence, the induced isomorphism $\tilde H_*(K;\mathbb Z)\to\tilde H_*(L;\mathbb Z)$ is $G$-equivariant. The inclusion is a $G$-equivariant map and an ordinary homotopy equivalence; it is not claimed to be a $G$-homotopy equivalence, and no $G$-equivariant homotopy inverse is needed for the isomorphism of $\mathbb ZG$-modules. No restriction to rational coefficients is needed either. Finally $L$ is connected, and its facets form a connected clutter covering its vertices, that is a flat FHT $\F$, whose automorphisms are those of $L$: $\Aut(\F)=\operatorname{aut}(L)\cong G$.
\end{proof}

\begin{remark}[Faithfulness]\label{rem:faithful}
The hypothesis of faithfulness in (iv) can be removed for a non-empty connected $K$. Choose a vertex $x$ of $K$, add for each $g\in G$ a new vertex $a_g$ and an edge $\{a_g,g\cdot x\}$, and let $h\in G$ act by $h\cdot a_g=a_{hg}$. The action on the enlarged complex $K'$ is simplicial and faithful ($h\ne1$ moves $a_1$), the pendant edges retract onto $K$, and the inclusion $K\subset K'$ is $G$-equivariant and a homotopy equivalence. Applying (iv) to $K'$ gives a flat FHT with $\Aut(\F)\cong G$ and $\tilde H_*(\F;\mathbb Z)\cong\tilde H_*(K;\mathbb Z)$ as $\mathbb ZG$-modules. This extension is not used elsewhere in this article.
\end{remark}

\begin{remark}[Scope]\label{rem:repr}
(a) The condition of (ii) is classical; the cokernel of $\beta$ is finite, of exponent dividing $\card G$ (Artin's induction theorem~\cite{serre}), trivial for $p$-groups and non-trivial in general~\cite{bartel_dokchitser16}; its kernel is described in~\cite{bartel_dokchitser}. It is satisfied by all irreducible rational modules of $S_4$, $S_5$, $GL(2,3)$, $Q_8$, $SL(2,3)$ and $\mathrm{Dic}_3$, and it fails for the irreducible rational module of dimension $4$ of $Q_8\times C_3$ placed in a single degree (computed; repository \FHTrepo). (b) The passage from the chain complex of (iii) to a finite simplicial complex is not addressed here. In (iii) the module $M$ is arbitrary: a non-zero trivial module placed in degree $0$ satisfies the condition, but is not the reduced homology of a flat FHT, which is connected. For a connected complex the reduced homology vanishes in degree $0$, so the question is the following, which we leave open: is every finite-dimensional graded $\mathbb QG$-module $M$ concentrated in degrees $\ge1$ with $\chi_G(M)\in\beta(A(G))$ the reduced rational homology of a faithful action of $G$ on a finite connected simplicial complex? Rationally, Moore $G$-spaces exist without finiteness~\cite{kahn86}, over the integers the corresponding problem of Steenrod fails~\cite{carlsson81}, and the finite obstruction lives on the orbit category~\cite{hpy13}. 
\end{remark}

The consequence is that the existence of an FHT carrying a given symmetry and a given homology says little: the content lies in principles that select carriers, such as the whole twisted algebra of Corollary~\ref{cor:B2}.

\section{Methodology and computer verification}\label{sec:verification}

\subsection{Method}\label{sec:method}
This article follows the method below.
\begin{enumerate}
\item \emph{Proof first, computation as a check or as data.} A statement labelled [T] is proved in the text; computation then serves as a check. A statement labelled [E] does not become [T] when it is quoted later. A statement labelled [E] is the output of a program; its scope (sizes, families, fields) is stated with it.
\item \emph{Arithmetic.} Counts are integers. Homology is computed over $\mathbb F_2$, over $\mathbb Q$ and, for large complexes, over a prime field. A comparison between rational coefficients and a finite field $\mathbb F_p$ can reveal $p$-torsion (equality of the Betti numbers in all degrees is equivalent to the absence of $p$-torsion), but agreement for the primes tested is no proof of the absence of torsion in general; for a complex computed over a single prime field, the Betti numbers over this field are reported as such, the arithmetic being exact without its result being automatically the rational one. Group orders are computed by refinement of colourings followed by backtracking. Character tables are computed in floating point (Section~\ref{sec:verif}).
\item \emph{Three kinds of control.} We distinguish (I) a reproduction by an independent method (brute force and recurrence, two codes written independently, an external source); (S) a comparison of two formulations that use common data (for instance the recurrence and the species equation, which use the same tables of fixed hosts); and (C) an internal consistency check (a dimension count, a global identity such as a Burnside sum, a second run of the same program, the verification of a proved identity). A check of type (S) or (C) detects some errors but does not reproduce the result: a dimension count does not reproduce a decomposition nor certify the real, complex or quaternionic types, and a correct Burnside sum checks a global identity, not each of its terms. Table~\ref{tab:trace} gives, for each computed result, its method, its control, the kind of the control and its scope.
\item \emph{Criteria written before computation.} When a computation is meant to decide a question (whether an object has a property, whether a prediction holds), the criterion of reading is written before the computation, and no parameter is adjusted afterwards; in this article no criterion was changed after a first result.
\item \emph{Status.} Every result carries its status and its nature (Section~\ref{sec:intro}); Section~\ref{sec:status} gives the table of the results of this article, with the programs of the repository that produce or check the computed results.
\item \emph{Closures.} Negative statements that condition what follows (a structure cannot exist, a family contains no carrier, an algebra cannot be the whole twisted algebra) are stated as lemmas or propositions, with their proof or their computation, since they fix the shape of the constructions.
\end{enumerate}
The programs are in the repository \FHTrepo, archived at \FHTdoi (release v9.2). All the programs cited below are in the folder \texttt{V9/code}; the independent audits are in \texttt{V9/code/AUDIT\_FABLE}, and the checks added for the present version in \texttt{V9/code/verifications\_A1}. The enumeration program is \texttt{FHT\_Enumeration.py} (folder \texttt{V9/code}), in which the values of $\operatorname{Fix}(\lambda)$ for $m=7,8$ obtained with C routines are stored. The rest of this section gathers the computer checks of the enumeration (Section~\ref{sec:enumeration}), of the invariants (Section~\ref{sec:invariants}) and of the morphisms and algebras (Sections~\ref{sec:morphisms} and~\ref{sec:valpha}).

\subsection{Enumeration}\label{sec:verif_enum}

\paragraph{Implementation.}
The counts of Section~\ref{sec:effective} were obtained with a Python program written for the component-wise axiom (A5), \texttt{FHT\_\allowbreak Enumeration.py}, which provides three methods. The code is available on GitHub\footnote{\FHTrepo}
and archived on Zenodo (\FHTdoi).

\begin{enumerate}
 \item The \emph{brute-force} method constructs the FHTs themselves, following
       Proposition~\ref{prop:recursive}. For each $v$, it runs over the partitions of
       $V$ into $m\ge2$ blocks, over the multisets of already known classes carried by
       blocks of equal size, and over all labelled connected clutters on the blocks.
       Duplicates are removed with canonical certificates of the
       realization computed with \texttt{pynauty}~\cite{mckay_nauty}, which is legitimate
       by Lemma~\ref{lem:grading}. The program checks at run time that the recursive
       decomposition is unique. This method also produces the lists of structures
       (Table~\ref{tab:fht_list}). It is practical up to $v=6$.
 \item The \emph{recurrence} method implements Corollary~\ref{cor:explicit}
       without constructing any FHT. The numbers $\operatorname{Fix}(\lambda)$ of Appendix~\ref{sec:fix} are obtained as follows.
       For $m\le6$ points, by an exhaustive search of the invariant connected clutters; for $m=7$, by a C routine of the same kind (a few seconds per type). The identity term
       is the number of labelled connected clutters on $m$ points, computed from the Dedekind numbers $D(m)$ (OEIS A000372~\cite{oeis_A000372}; $D(7)$, $D(8)$~\cite{wiedemann_D8}) by inclusion--exclusion and the exponential formula.
       At $m=8$ no host is enumerated for the types that matter most: the involutive types $(2^j,1^{8-2j})$ ($j\ge1$) and $(3,1^5)$, which have between $2\cdot10^{11}$ and $10^{17}$ invariant hosts, are computed by a recurrence of Dedekind type (Appendix~\ref{sec:recfix}); the remaining sixteen types have fewer than $10^{9}$ invariant hosts and are counted by the C routine in a few seconds each. The recurrence reproduces every directly enumerated value for $m\le7$.
       Hosts on $m$ points occur for $v\ge m+1$ in the imbricated case, so these tables give $f(v)$ up to $v=9$.
       The flat numbers $S(v)$ are computed by Burnside's lemma up to $v=6$; for $v\ge7$ they are derived from the exact numbers A003182 (see Section~\ref{sec:effective}), and the Burnside computation of $S(7)$ and $S(8)$ from the tables of $\operatorname{Fix}$ is used as a check.
 \item The \emph{species equation} (Corollary~\ref{cor:species}) is solved by iteration on power series; it only gives the totals $f(v)$ and organizes the computation differently from the recurrence by profiles. At the largest order, $v=9$, no table of hosts on $9$ points is needed beyond $S(9)$: a host on $9$ points contributes to the coefficient of $x^9$ only when all its points carry the trivial FHT, since any non-trivial substitution increases the number of vertices. This is the flat contribution $S(9)$, taken from A003182; the imbricated contributions use hosts on at most $8$ points, whose tables $\operatorname{Fix}$ are known.
\end{enumerate}

\paragraph{Validation.}
\begin{itemize}
 \item The brute-force and recurrence methods agree for $1\le v\le5$, and for the imbricated structures at $v=6$ ($I(6)=1005$, with the same distribution by depth). The flat part at $v=6$ is the value $S(6)=15942$ of A261006, which the host counting also reproduces.
 \item An exhaustive test of the axioms (A1)--(A5), over all hypergraphs on $v\le4$ vertices and all assignments of levels $1,\dots,v-1$ to their edges (no FHT on $v$ vertices has depth larger than $v-1$, Lemma~\ref{lem:depthbound}), gives $1,\,1,\,4,\,23$ FHTs for $v=1,\dots,4$ (it needs neither the decomposition nor the recurrence).
 \item The values $S(v)$ computed by Burnside's lemma from the tables $\operatorname{Fix}(\lambda)$ coincide with the values $S(v)$ above for $v\le7$ (and for $v=8$, see below). For $v=7$ this checks the aggregate Burnside identity involving all fifteen cycle types of $\mathfrak{S}_7$, the C routine and the Dedekind number $D(7)$; it does not check each term separately. The identity terms obtained by exhaustive search for $m\le6$ coincide with the Dedekind formula.
 \item The recurrence and the species equation give the same $f(v)$ for $v\le9$.
 \item At $m=8$ the Burnside sum over $\mathfrak{S}_8$ of the whole table (the involutive types and $(3,1^5)$ by the recurrence of Appendix~\ref{sec:recfix}, the identity from $D(8)$, the other sixteen types by the C routine) is $8!\,S(8)$ with $S(8)=1392195547909966876$, which is the exact value derived from A003182 (see above): no term of the table is deduced from this value, so the check is independent. For $v=9$, the flat term is $S(9)=789204635842035037743349748520248426$ (A003182, exact).
\end{itemize}

\subsection{Invariants}\label{sec:verif_inv}

\paragraph{Verification.} The data were computed by exhaustive programs (\texttt{FHT\_Invariants.py}, folder \texttt{V9/code} of the repository, which reads the table written by \texttt{FHT\_Enumeration.py}, Section~\ref{sec:verif_enum}): the groups by search among the permutations compatible with a colour refinement, the Betti numbers by ranks of boundary matrices, the distances by breadth-first search, the tree by the construction used in the proof of the Realization Uniqueness lemma (Lemma~\ref{lem:grading}) (top edges = maximal edges, children of the root = connected components of the other edges). For all $260$ FHTs we checked that the tree and the hosts rebuild the realization; that the order of $\Aut(\F)$ computed directly equals $\prod_H\lvert \Gamma(H)\rvert{}$ (Corollary~\ref{cor:order}); that the Betti numbers of $\F$ and of its root host coincide (Proposition~\ref{prop:bettihost}) and that the Betti numbers over $\mathbb{F}_2$ and $\mathbb{Q}$ coincide (which excludes $2$-torsion, and only $2$-torsion); and, on the $1261$ vertices of the $259$ FHTs with at least two vertices, that the eccentricities of $\PG(\F)$ are those given by Proposition~\ref{prop:blowup}.

\subsection{Morphisms and twisted algebras}\label{sec:verif}

All theorems above are proved. We checked them, and quantities on which no proof is given, by computation on the isomorphism classes of FHTs with $v\le6$ vertices ($1,1,4,23,231,16947$ classes for $v=1,\dots,6$, in agreement with Table~\ref{tab:fht_counts}). The checks of this subsection on the moves were recomputed for the present version by \texttt{verifications\_\allowbreak A1/\allowbreak check\_\allowbreak moves.py}, written from the definitions of Section~\ref{sec:moves} alone (the program of the first version was not kept). They are exhaustive for $v\le6$; the only random choices are those of the descents, with the seed $2026$. The classes are generated by closure of the interior moves from the star, which gives $1,1,4,23,231,16\,947$ classes for $v=1,\dots,6$. For each class and for each non-leaf component we enumerated all admissible choices of $\Th_1^{+}$, $\Th_1^{-}$, $\Th_2^{-}$, $\Th_2^{+}$ as in Section~\ref{sec:moves}.

\paragraph{Validity and inverses.} For $v\le5$ (all $259$ classes with $v\ge2$) we enumerated $1\,807$ moves $\Th_1^{+}$ (of which $715$, $502$, $590$ with $\card{\Abs}=0,1,2$), $2\,351$ moves $\Th_1^{-}$ ($892$, $622$, $837$ with $\card B=0,1,2$), $110$ moves $\Th_2^{-}$ and $218$ moves $\Th_2^{+}$. In each case the result satisfies the axioms (A1)--(A5) read literally, and its realization decomposes back to the same FHT; the inverse move exists (and is unique for $\Th_2^{\pm}$); $\Th_1$ preserves the hierarchy and the levels; $\Th_2^{-}$ never increases the depth ($110/110$, strictly in $64$ cases) and $\Th_2^{+}$ never decreases it ($218/218$, strictly in $140$); the reduced host is preserved. For $v=6$ (all $16\,947$ classes): $351\,759$ moves $\Th_1^{+}$ ($136\,023$, $111\,110$, $104\,626$ with $\card{\Abs}=0,1,2$), $367\,491$ moves $\Th_1^{-}$ ($138\,219$, $113\,230$, $116\,042$ with $\card B=0,1,2$), $1\,227$ moves $\Th_2^{-}$ (the depth decreasing strictly in $754$ cases) and $2\,048$ moves $\Th_2^{+}$ (strictly increasing in $1\,184$), with the same checks and no exception. Each move also changes $\kappa$ as stated in Proposition~\ref{prop:kappa}.

\paragraph{The morphism $\Omega_1$.} From all classes with $2\le v\le6$ we enumerated the leaf additions: $1$, $9$, $80$, $1\,089$ and $100\,947$ moves $\Omega_1^{+}$ from the classes with $v=2,\dots,6$ (choices of $H$ and of a block); all results are FHTs (axioms read literally) and all are exactly inverted by $\Omega_1^{-}$. The leaf removals allowed by (a), (b) are $0$, $3$, $15$, $117$, $1\,589$ for $v=2,\dots,6$; all results are FHTs and are exactly inverted by $\Omega_1^{+}$. The removals blocked by the condition are $2$, $9$, $77$, $1\,038$, $100\,093$, and are not tested. The two removals blocked at $v=2$ are the two degenerate $\Omega_1^{-,0}$, which give $\F_1$; the degenerate addition $\Omega_1^{+,0}$ from $\F_1$ is unique. Using $\Th_1^{\pm},\Th_2^{\pm}$ and $\Omega_1^{+}$ from $v\le5$ to $v\le6$, the $17\,206$ classes with $2\le v\le6$ form a single connected component (Theorem~\ref{thm:global}), and with the degenerate morphisms the $17\,207$ classes with $1\le v\le6$ (including $\F_1$) do.

\paragraph{Edge toggles.} For all $259$ classes with $2\le v\le5$, the sets of FHTs obtained from $\Rea(\F)$ by adding one set ($933$) or deleting one set ($1\,002$) such that the result is the realization of an FHT coincide exactly with the results of the fusions with $\Abs(e)=\varnothing$, the fissions with $B=\varnothing$ and the moves $\Th_2^{\pm}$, as predicted by Theorems~\ref{thm:deletion} and~\ref{thm:addition}. For all $16\,947$ classes with $v=6$ the same check gives $138\,071$ additions and $139\,446$ deletions, again in exact agreement.

\paragraph{Connectivity.} The number of connected components of the quotient transition graph on isomorphism classes, by $v$ and by the set of moves used, is the following (the number of classes is $f(v)$).

\begin{table}[ht]
\centering
\begin{tabular}{@{}lrrrrr@{}}
\toprule
 & $v=2$ & $3$ & $4$ & $5$ & $6$ \\
\midrule
classes $f(v)$ & 1 & 4 & 23 & 231 & 16947 \\
\midrule
all components, $\Th_1^{\pm},\Th_2^{\pm}$ & 1 & 1 & 1 & 1 & 1 \\
moves at distance $\le1$ from the root & 1 & 1 & 1 & 1 & 1 \\
root only & 1 & 1 & 2 & 10 & 131 \\
$\Th_1^{\pm}$ only (all components) & 1 & 2 & 5 & 12 & 33 \\
$\Th_2^{\pm}$ only (all components) & 1 & 3 & 15 & 168 & 16103 \\
\bottomrule
\end{tabular}
\caption{Connected components of the transition graph on isomorphism classes. The values for $\Th_1$ are the numbers of unlabelled hierarchies (Proposition~\ref{prop:orbits}); the first nonzero excess of the root-only graph is at $v=4$ (Proposition~\ref{prop:necessity}).}
\end{table}

For $v=6$ the enumeration gave $723\,408$ moves $\Th_1^{\pm}$, $2\,266$ moves $\Th_2^{+}$ and $1\,337$ moves $\Th_2^{-}$ from the $17\,206$ classes with $2\le v\le6$ (counted from the class representatives), and none of them leaves the set of classes. the theorem of Section~\ref{sec:connect} is proved for all $v$; these figures are consistency checks and give the numbers of the components of the restricted graphs.

\paragraph{Descent and complexity.} From each of the $17\,206$ classes with $2\le v\le6$ we performed three restricted descents, following the rule of the proof of Theorem~\ref{thm:descent}: a descending $\Th_1$ at the root when its host has at least two edges, otherwise a descending $\Th_1$ at a non-leaf child of the root whose host has at least two edges, otherwise the explosion on (root, $c$) with $\varepsilon=\Ch(c)$; the admissible choices are drawn at random with the seed $2026$. All the moves are located at distance at most one from the root. In the $51\,618$ descents ($377\,844$ steps) such a move always exists, the result satisfies the axioms read literally, $\kappa$ decreases by exactly one at each step, and the sequence ends at the star after exactly $\kappa-1$ steps. For $v\le5$ the first version reported $777$ descents of an unrestricted protocol; they are replaced by these. The distribution of $\kappa$ and $\kappa_{\max}=\binom{v}{\lfloor v/2\rfloor}$ ($3\le v\le6$) of Section~\ref{sec:descent} are read from the $16\,947$ classes with $v=6$.

\paragraph{Automorphism groups and classes $(v,\alpha)$.} For every class with $v\le6$ we computed $\Aut(\F)$ by testing the $v!$ permutations, the canonical form of its conjugacy class in $S_v$ (each $\sigma\in S_v$ conjugates simultaneously all the elements of the subgroup, the conjugated elements are sorted, and the minimum over $S_v$ of the sorted lists is taken), its orbits on the vertices, and the characters of the group. The program that produced the characters and the real types of Table~\ref{tab:valpha} is not part of the repository; the table was recomputed independently, by a numerical Wedderburn decomposition in floating point (centre of the algebra, blocks of a generic central element, type of each block by the signature of the trace form) of $\mathbb R[G]$ and of $\mathbb R^{\eta}[G]$ $G$ running over the representative groups of the $54$ classes and $\eta$ over $\mathrm{Hom}(G,\{\pm1\})$ ($S_6$ being treated by its degrees, none of which is $3$): the types $(r,c,h)$ agree for the $54$ classes, and a sector $A_F$ occurs exactly in the cases of Proposition~\ref{prop:census6} (program \texttt{verifications\_A1/valpha54\_check.py}). The twisted algebras of Proposition~\ref{prop:spinsector} were computed in floating point by Burnside's algorithm (eigenvectors of a generic linear combination of the class matrices), the degrees and the twisted Frobenius--Schur indicators being rounded to integers and checked by integrality and by $\sum d^2=\card G$. None of these computations is an exact certification. As a check, $\sum_{[\F]}v!/\card{\Aut\F}$ gives $1,1,8,149,8\,488,7\,923\,621$ labelled FHTs for $v=1,\dots,6$, which is also what an independent recursion gives (set partitions of the vertices, a connected clutter on the blocks, and the numbers of labelled connected clutters computed in \texttt{FHT\_Enumeration.py} from the Dedekind numbers); this checks a global identity, not each group nor each conjugacy class; for $v=3,4$, a direct count of all families of subsets of $V$ satisfying the axioms read literally also gives $8$ and $149$.

\paragraph{Twisted algebras and the census of $S_4$.} For the representative groups of the $54$ classes $(v,\alpha)$ with $v\le6$ and every $\eta\in\mathrm{Hom}(\Aut\F,\{\pm1\})$, the decomposition of $\mathbb R^{\eta}[\Aut\F]$ by Theorem~\ref{thm:twistedWedderburn} was compared with the decomposition of $\mathbb R[\hat G]=\mathbb R[G]\oplus\mathbb R^{\eta}[G]$ (a comparison of type (S), the same character tables being used) and checked by the dimension count $\sum d^2=\card G$ (type (C)). The generator of the $S_4$-invariant FHTs (Section~\ref{sec:census}) was validated on $v\le6$ against the table of classes: it returns the $3,9,4+47$ classes with $\Aut\cong S_4$, the $9+57$ with $S_4\times C_2$, the $4+14$ with $S_5$ and the $5$ with $S_6$, the automorphism groups being recomputed independently by backtracking. For $v=7$ the numbers of isomorphism classes of $S_4$-invariant FHTs per action are the ones of Table~\ref{tab:census}; isomorphism classes are decided by a canonical form under all the relabellings of $V$. The figures for $v\le7$ were recomputed by \texttt{verifications\_\allowbreak A1/\allowbreak census\_\allowbreak s4.py}, written independently of the program of the first version (which was not kept): it finds the five faithful actions at $v=7$ (and $1,1,4,5,10,13$ actions for $v=4,\dots,9$), $2\,927$ classes whose automorphism group contains $S_4$, and the orders of the automorphism groups $24$ ($1\,343$ classes: $51$, $76$, $1\,216$ for the actions $6+1$, $4+3$, $4+1+1+1$; $\mathrm{sgn}_V\ne1$ for $1\,216$), $48$ ($1\,227$), $144$ ($143$), $120$ ($93$), $168$ ($5$), $240$ ($90$), $720$ ($20$) and $5\,040$ ($6$); it also reproduces the counts for $v\le6$. For $v=8$ and $v=9$ the same program, with $\card{\Aut}$ computed by the exact sequence of Theorem~\ref{thm:autrec} (checked against the direct count on all classes for $v=6,7$), reproduces the figures of Proposition~\ref{prop:census7}: $10\,618$ classes for the nine actions at $v=8$ (by action $26$, $88$, $2\,573$, $20$, $902$, $7\,009$, and none for $4+2+2$ and $4+2+1+1$), and $3\,104$ for the action $8+1$ at $v=9$. The action $4+1+1+1+1$ at $v=8$ exceeded the memory available ($8$\,GB) and remains untreated. For $v=8$ the action $4+1+1+1+1$ is not included, and for $v=9$ only the action $8+1$ was processed.

\paragraph{Sectors of larger groups.} The decompositions of $\mathbb R^\eta[G]$ for $G=S_4\times C_2$ and $G=S_4\times S_3$ with their sign characters, and of $\mathbb R^{\det}[GL(2,3)]$, were recomputed independently by a numerical Wedderburn decomposition in floating point (centre, blocks of a generic central element, type of each block by the signature of the trace form; not an exact certification), validated on $\mathbb R^{\mathrm{sgn}}[S_4]$ and $\mathbb R^{\mathrm{sgn}}[S_5]$. For $S_4\times S_3$ acting on $4+3$ points, $\eta=\mathrm{sgn}\otimes1$ gives $A_F\oplus A_F\oplus M_2(A_F)$, $\eta=1\otimes\mathrm{sgn}$ gives $2\mathbb C\oplus2\mathbb H\oplus M_2(\mathbb C)\oplus M_2(\mathbb H)\oplus2M_3(\mathbb C)\oplus2M_3(\mathbb H)$, and $\eta=\mathrm{sgn}_V=\mathrm{sgn}\otimes\mathrm{sgn}$ gives $2\mathbb C\oplus2M_2(\mathbb C)\oplus2M_3(\mathbb C)\oplus M_6(\mathbb C)\oplus M_4(\mathbb R)$, with no sector $A_F$. For $S_4\times C_2$ acting on $4+2$ points, only $\eta=\mathrm{sgn}\otimes1$ gives a sector $A_F$ (namely $A_F\oplus A_F$). For $GL(2,3)$ the decomposition of Proposition~\ref{prop:spinsector} is recovered, together with the counts used in its proof ($14$ solutions of $g^2=I$, $6$ of $g^2=-I$), and $\det$ is the sign of the permutation of the nine points of $\mathbb F_3^2$.

\subsection{Traceability}\label{sec:trace}

Table~\ref{tab:trace} summarizes the computed results of this article with the kind of their control: (I) independent method, (S) comparison of formulations with common data, (C) internal consistency (Section~\ref{sec:method}).

\begin{table}[ht]
\centering\footnotesize
\renewcommand{\arraystretch}{1.15}
\begin{tabularx}{\textwidth}{@{}>{\raggedright\arraybackslash}p{2.6cm}>{\raggedright\arraybackslash}X>{\raggedright\arraybackslash}X c >{\raggedright\arraybackslash}p{2.4cm}@{}}
\toprule
result & main method & control (what it covers) & kind & scope \\
\midrule
$f(v)$, $v\le5$ & brute force (recursive decomposition) & recurrence; test of the axioms ($v\le4$); independent code \texttt{AUDIT\_\allowbreak FABLE/\allowbreak AF1\_\allowbreak enumeration\_\allowbreak independante.py}, written from Definition~\ref{def:fht} alone (classes and labelled counts) & I & exhaustive \\
$f(6)$ & recurrence & brute force for the imbricated part $I(6)$ and its distribution by depth; flat part $S(6)$ from A261006 and host counting & I & exhaustive \\
$f(7)$, $f(8)$, $f(9)$ & recurrence (tables $\operatorname{Fix}$) & species equation (totals $f(v)$ only, not the distribution by depth) & S & all cycle types \\
$\operatorname{Fix}(\lambda)$, $m\le7$ & exhaustive search ($m\le6$), C routine ($m=7$) & recurrence of Dedekind type (each involutive type); identity term from $D(m)$; Burnside sum (aggregate) & I; C & exhaustive \\
$\operatorname{Fix}(\lambda)$, $m=8$ & recurrence (involutive types, $(3,1^5)$), C routine, $D(8)$ & Burnside sum $=8!\,S(8)$, with $S(8)$ from A003182 (aggregate identity only) & C & all $22$ types \\
invariants of the $260$ FHTs & \texttt{FHT\_\allowbreak Invariants.py} & rebuilding of the realization; $\card{\Aut}=\prod\card{\Gamma}$; Betti numbers of the root host; $\mathbb F_2$ against $\mathbb Q$ & C & exhaustive, $v\le5$ \\
separation (vii) & tables of Appendices~\ref{sec:FHT_list}--\ref{sec:annexB} & recomputed by \texttt{verifications\_\allowbreak A1/\allowbreak sep\_\allowbreak signature.py} & S & $v\le5$ \\
validity and inverses of $\Th_1^{\pm},\Th_2^{\pm}$ & enumeration of all moves (\texttt{verifications\_\allowbreak A1/\allowbreak check\_\allowbreak moves.py}) & axioms read literally; decomposition back to the FHT; inverse moves; variation of $\kappa$ & C & exhaustive $v\le6$ \\
$\Omega_1$ & enumeration of all moves (idem) & axioms; inverse moves & C & exhaustive $2\le v\le6$ \\
edge toggles & additions and deletions of one set (idem) & comparison with the moves (Theorems~\ref{thm:deletion}, \ref{thm:addition}) & I & exhaustive $v\le6$ \\
components of the transition graphs & enumeration on the classes (idem) & connectivity is proved (Theorem~\ref{thm:connect}) & C & exhaustive $v\le6$ \\
descent & restricted descents, rule of Theorem~\ref{thm:descent}, seed $2026$ (idem) & moves at distance $\le1$; $\kappa$ decreases by one; end at the star & C & $3$ per class, $2\le v\le6$ \\
$\Aut(\F)$, classes $(v,\alpha)$ & test of the $v!$ permutations & labelled totals against an independent recursion (global identity only); direct count for $v=3,4$ & I (totals) & exhaustive $v\le6$ \\
twisted algebras, $v\le6$ & character tables and Theorem~\ref{thm:twistedWedderburn} & numerical Wedderburn decomposition of every $\mathbb R^\eta[G]$, floating point (\texttt{verifications\_\allowbreak A1/\allowbreak valpha54\_\allowbreak check.py}; $S_6$ treated by its degrees, not by this decomposition); $\mathbb R[\hat G]=\mathbb R[G]\oplus\mathbb R^\eta[G]$; $\sum d^2=\card G$ & I; S; C & the $54$ classes \\
census of $S_4$, $v=7,8,9$ & $S_4$-invariant generation (program of the first version, not kept) & validation against the table for $v\le6$; recomputation for $v\le9$ (actions processed) by \texttt{verifications\_\allowbreak A1/\allowbreak census\_\allowbreak s4.py} & I & $v=7$ all actions; $v=8$: $9$ of $10$; $v=9$: $1$ of $13$ \\
sectors of $S_4\times C_2$, $S_4\times S_3$, $GL(2,3)$ & characters and Theorem~\ref{thm:twistedWedderburn} & numerical Wedderburn decomposition, floating point (\texttt{verifications\_\allowbreak A1/\allowbreak twisted\_\allowbreak decomp.py}, \texttt{gl23\_\allowbreak check.py}); norms of Proposition~\ref{prop:spinsector} (\texttt{norms\_\allowbreak check.py}) & I & these groups \\
\bottomrule
\end{tabularx}
\caption{Traceability of the computed results (kind of control: I independent method, S formulations with common data, C internal consistency).}\label{tab:trace}
\end{table}

\section{Status of the results}\label{sec:status}

Table~\ref{tab:status} gives the status of the definitions and results of this article (``def.'' marks a definition, of nature \emph{posed}), with the conventions of Section~\ref{sec:intro}. The columns separate the proof, the computation (the programs, given by their path in the folder \texttt{V9/code} of the repository; ``---'' when there is none) and the control, with its kind (I, S, C, Section~\ref{sec:method}); ``proof review'' means a re-reading of the proof, without program.

\begin{table}[p]
\centering\scriptsize
\renewcommand{\arraystretch}{1.05}
\begin{tabularx}{\textwidth}{@{}l>{\raggedright\arraybackslash}X l l >{\raggedright\arraybackslash}p{2.2cm} >{\raggedright\arraybackslash}p{2.9cm} >{\raggedright\arraybackslash}p{2.9cm}@{}}
\toprule
\# & result & status & nature & proof & computation & control \\
\midrule
F1a & Axioms (A1)--(A5) of an FHT & def. & P & Definition~\ref{def:fht} & --- & --- \\
F1b & Grading determined by the hypergraph; unique decomposition into children and host & [T] & D & Sections~\ref{sec:defsec}--\ref{sec:enumeration} & \texttt{FHT\_\allowbreak Enumeration.py} & uniqueness of the decomposition checked at run time (C) \\
F2a & Recurrence and species equation & [T] & D & Section~\ref{sec:enumeration} & --- & --- \\
F2b & Values of $f(v)$ up to $9$ vertices & [E] & D; B (data) & recurrence, with $D(7)$, $D(8)$, $a(8)$, $a(9)$ taken as they stand & \texttt{FHT\_\allowbreak Enumeration.py}; \texttt{AUDIT\_\allowbreak FABLE/\allowbreak AF1\_\allowbreak enumeration\_\allowbreak independante.py} & brute force against recurrence, $v\le6$ (I); independent code, $v\le5$ (I); species equation, totals (S) \\
F3 & Hierarchical hosts; $\card{\Aut}$ is a product; Betti numbers of the root host & [T] & D & Section~\ref{sec:invariants} & \texttt{FHT\_\allowbreak Invariants.py} & C \\
F4 & Numerical invariants (vii) separating the $260$ FHTs with at most $5$ vertices & [E] & D & --- & \texttt{FHT\_\allowbreak Invariants.py}; \texttt{verifications\_\allowbreak A1/\allowbreak sep\_\allowbreak signature.py} & recomputed from the tables (S) \\
F5a & Moves $\Th_1^{\pm}$, $\Th_2^{\pm}$, $\Omega_1$ (with the bound $2$ on absorbed edges and fragments) & def. & P & Section~\ref{sec:moves} & --- & --- \\
F5b & Validity, inverses, connectivity; descent and distance to the star; groupoids & [T] & D & Section~\ref{sec:morphisms} & \texttt{verifications\_\allowbreak A1/\allowbreak check\_\allowbreak moves.py} & C, exhaustive $v\le6$ (Section~\ref{sec:verif}) \\
F6 & $\mathbb R^{\mathrm{sgn}}[S_4]\cong A_F$; uniqueness of $(S_4,[w_1^2])$ among real cocycles, group up to isomorphism, class after transport & [T] & \textbf{R}; D & Theorem~\ref{thm:AF}, Lemma~\ref{lem:realcocycle} & \texttt{AUDIT\_\allowbreak FABLE/\allowbreak AF2\_\allowbreak algebre\_\allowbreak tordue\_\allowbreak S4.py}; \texttt{verifications\_\allowbreak A1/\allowbreak twisted\_\allowbreak decomp.py} & numerical decomposition, floating point (I); proof review \\
F6b & Twisted algebras as Hodge-real forms; the factors of $A_F$ as Hodge types & [T] & D & Theorem~\ref{thm:hodge}, Corollary~\ref{cor:hodgeAF} & --- & proof review \\
F7 & Commutative factor; only $A_F$ is a whole twisted algebra of a sign character in the ladder & [T] & D & Proposition~\ref{prop:B1}, Corollary~\ref{cor:B2} & --- & proof review \\
F8 & $\mathbb R^{\det}[GL(2,3)]$; $\mathbb R^{\mathrm{sgn}}[S_5]$ & [T] & D & Proposition~\ref{prop:spinsector} & \texttt{AF20\_\allowbreak algebres\_\allowbreak tordues/\allowbreak }; \texttt{verifications\_\allowbreak A1/\allowbreak gl23\_\allowbreak check.py}; \texttt{verifications\_\allowbreak A1/\allowbreak norms\_\allowbreak check.py} & numerical decomposition, floating point (I) \\
F9 & Representability (Theorem~\ref{thm:repr}) & [T] & D; (iv)(a) B & Theorem~\ref{thm:repr}; (iv)(a) from~\cite{cgv25} & \texttt{burnside\_\allowbreak signe\_\allowbreak v2.py} & examples of Remark~\ref{rem:repr} only \\
F10 & Census of the classes whose automorphism group contains $S_4$ (Proposition~\ref{prop:census7}, Table~\ref{tab:census}): counts of classes, orders of $\Aut$ & [E] & D & --- & \texttt{verifications\_\allowbreak A1/\allowbreak census\_\allowbreak s4.py}, \texttt{census\_\allowbreak s4\_\allowbreak big.py} & exact combinatorial count; recomputed for the actions processed (I) \\
F11 & Identification of the algebras $\mathbb R^\eta[\Aut\F]$ and of their sectors $A_F$ (Propositions~\ref{prop:census6}, \ref{prop:census7}, Table~\ref{tab:valpha}) & [E] & D & Theorem~\ref{thm:AF} for $\Aut\F\cong S_4$ & \texttt{verifications\_\allowbreak A1/\allowbreak valpha54\_\allowbreak check.py}, \texttt{twisted\_\allowbreak decomp.py} & floating point (I), $S_6$ by its degrees; $\hat G$ (S); not an exact certification \\
\bottomrule
\end{tabularx}
\caption{Status of the results of this article. For F6 the nature is \emph{reconstructed} for the algebra, which is that of Connes, and \emph{derived} for the uniqueness, among all real $2$-cocycles (Lemma~\ref{lem:realcocycle}), of the group up to isomorphism and of the class of the cocycle after transport, and for its origin as a twisted group algebra of automorphisms of FHTs. For F9, the control of the automorphism group in (iv) is taken from~\cite{cgv25}, a preprint; the homotopy equivalence is proved here. F10 is an exact count of classes; F11, the identification of the simple factors, is computed in floating point and is not an exact certification, except where Theorem~\ref{thm:AF} applies.}\label{tab:status}
\end{table}

\paragraph{Choices and borrowings.} No physical postulate is used, and no result is of nature \emph{ad hoc}. The choice of construction is declared, of nature \emph{posed}: the algebra attached to an FHT in Section~\ref{sec:twisted} is a \emph{twisted} group algebra, for the sign $\eta_V$ of the action on the vertices, fixed once for all (Corollary~\ref{cor:etaV}); for $\Aut\F\cong S_4$ it is $\mathrm{sgn}$ or trivial. The theorems about the construction are proved conditionally on this choice; they determine its consequences, not its necessity, which does not follow from the axioms of FHTs. For a general group the choice is among the elements of $\mathrm{Hom}(G,\{\pm1\})$. For $S_4$ it is binary, because $S_4$ has exactly two characters ($\mathrm{Hom}(S_4,\{\pm1\})=\{1,\mathrm{sgn}\}$): to twist or not; Theorem~\ref{thm:hodge} then gives the twist a geometric meaning, as the orientation character of a representation, without selecting it by itself (every character is the determinant of itself). Remark~\ref{rem:connes} records that no untwisted real group algebra is isomorphic to the whole algebra $A_F$, at any size, and that at the sizes of the census ($v\le6$) no untwisted algebra $\mathbb R[\Aut\F]$ has a sector $A_F$ (computed). For larger groups such sectors exist: in $\mathbb R[Q_8\times S_4\times C_3]$ the products of factors $\mathbb H\otimes\mathbb R\otimes\mathbb R$, $\mathbb R\otimes\mathbb R\otimes\mathbb C$ and $\mathbb R\otimes M_3(\mathbb R)\otimes\mathbb C$ give a sector $\mathbb C\oplus\mathbb H\oplus M_3(\mathbb C)$, and every finite group is the automorphism group of a flat FHT. Remark~\ref{rem:whicheta} records that several canonical characters exist; Theorem~\ref{thm:AF} itself is a statement about all pairs (finite group, $\{\pm1\}$-valued $2$-cocycle). The proofs use classical results, cited where they are used (P\'olya theory and species, Frucht's theorem, the representation theory of finite groups and of twisted group algebras, Hopf's trace formula, Artin's induction theorem). Two ingredients are taken as they stand, without being re-derived: the control of the automorphism group by the rigidification theorem in Theorem~\ref{thm:repr}(iv)~\cite{cgv25}, which is the only result of nature B in Table~\ref{tab:status}, and the values of $D(7)$, $D(8)$~\cite{wiedemann_D8} and of $a(8)$, $a(9)$~\cite{pawelski9}, on which the exact values at eight and nine vertices depend. The statements labelled ``Computed'' in the text (for instance Propositions~\ref{prop:census6} and~\ref{prop:census7}, and the tables of Sections~\ref{sec:valpha} and~\ref{sec:verification}) have the status [E]; they are obtained by the programs of the repository and checked as described in Section~\ref{sec:verification}.

\paragraph{The only reconstructed object: $A_F$.} The algebra $A_F=\mathbb C\oplus\mathbb H\oplus M_3(\mathbb C)$ is the finite algebra of the noncommutative-geometric description of the Standard Model by Connes and by Chamseddine, Connes and Marcolli~\cite{connes96,ccm07,cc08,vs15}. Theorem~\ref{thm:AF} does not produce a new algebra: it recovers this one as the twisted group algebra $\mathbb R^{\mathrm{sgn}}[S_4]$. The novelty is the origin of the algebra, as a twisted group algebra of the automorphism group of FHTs, and the uniqueness of the pair $(S_4,[w_1^2])$ among all pairs (finite group, real $2$-cocycle), up to isomorphism of the group and cohomology of the cocycle after transport (Lemma~\ref{lem:realcocycle}); neither a permutation action, nor an FHT, nor a representative cocycle is unique; Remark~\ref{rem:census} shows that it occurs for FHTs at every $v\ge4$, so that the algebra alone does not select an FHT (Remark~\ref{rem:census}).

\section{Perspectives}\label{sec:perspectives}

Four questions are open.
\begin{enumerate}
 \item \emph{Asymptotics.} Table~\ref{tab:fht_counts} suggests, without proving it, that $I(v)/f(v)\to0$ and
       that $I(v)$ is dominated by the hosts with one marked vertex, $H((1);(v-2))$; these are conjectures. Proving these estimates would require an asymptotic control of the host
       numbers $H(\alpha;\beta)$ of Corollary~\ref{cor:explicit}, which count coloured
       connected clutters under symmetry constraints, in the spirit of the asymptotics of the Dedekind numbers.
 \item \emph{Geometric realisation.} The question left open by Theorem~\ref{thm:repr} is rational. For a finite-dimensional graded $\mathbb QG$-module concentrated in degrees $\ge1$ and satisfying the Burnside condition, is there a finite connected simplicial complex with an action of $G$ whose reduced rational homology is this module, and hence, after the faithful enlargement of Remark~\ref{rem:faithful} and rigidification, a flat FHT with automorphism group $G$ (Remark~\ref{rem:repr})? A positive answer would not give an integral realisation, which is a different problem~\cite{carlsson81}. A natural route is a stable version (join with a representation sphere, then equivariant obstruction theory over the orbit category). Stabilization shifts the degrees and may twist the modules by an orientation character, so a stable realisation does not by itself answer the question for the given graded module.
 \item \emph{Projective sectors.} Proposition~\ref{prop:spinsector} is the first instance of a sector of a twisted algebra of a double cover. For a double cover $\hat G$ with central element $z$ of order $2$ and a character $\eta$ with $\eta(z)=1$, the element $u_z$ is central and $u_z^2=1$, and the \emph{spinorial sector} is $e_-\,\mathbb R^{\eta}[\hat G]$ with the central idempotent $e_-=\tfrac12(1-u_z)$; the algebra $\mathbb R^{\eta}[\hat G]$ itself is not a sector. A census of these spinorial sectors for the double covers of the automorphism groups of small FHTs, in the spirit of Section~\ref{sec:census}, would say which extensions of $A_F$ can be carried by FHTs at a given size.
 \item \emph{Species and moves.} The species equation (Corollary~\ref{cor:species}) and the naturality of the morphisms (Section~\ref{sec:functor}, where each morphism is a span of species and $\Omega_1$ an isomorphism onto a subspecies of the derivative species) place FHTs and their moves in the theory of species~\cite{bll_species}. The exponential generating series of labelled FHTs, and a description of the transition graphs at the level of species, remain to be exploited, for instance for the asymptotics above.
\end{enumerate}

\section{Conclusion}\label{sec:conclusion}

\paragraph{Theorems valid for all FHTs.} An FHT is a Schr\"oder hierarchy which carries a connected clutter at each node. It decomposes uniquely into the children of its top component and a host, an arbitrary connected clutter on these children. This gives an exact recurrence and a species equation for the numbers of FHTs. The global invariants reduce to the hosts. The automorphism group is a split extension, non-canonically, of the root host group by the groups of the children. The primal graph is a clique blow-up of that of the root host, with the same radius and diameter, and the Betti numbers are those of the root host. The diameter of the incidence graph is twice the larger diameter of the primal and edge graphs, or one more.

Four interior moves, fusion, fission, aggregation and explosion, produce FHTs and come in inverse pairs. Among them, aggregation and explosion, together with fusion without absorption and fission without fragment, are exactly the additions and deletions of one edge of the realization that preserve the FHT property; the other fusions and fissions change several coordinates of the realization. The interior moves connect all FHTs on a given set of vertices, and the moves located at distance at most one from the root suffice. With the exterior move $\Omega_1$, which adds or removes a twin leaf, including the degenerate moves between the one-vertex FHT and the two-vertex star, they connect the FHTs of all sizes. Each interior move changes the number of edges $\kappa$ by at most one; the level moves leave it unchanged. On at least two vertices, every FHT descends to the star by moves located at distance at most one from the root, and $\kappa(\F)-1$ is its distance to the star. The thin groupoid is the groupoid of fractions of the resulting poset. Only the action groupoid has the automorphism groups as isotropy groups.

\paragraph{Results computed on a finite range.} The numbers of FHTs are exact up to nine vertices. On the $260$ FHTs with at most five vertices, a numerical signature without isomorphism types (symmetry, homology, eccentricities, host lines, edge orbits, sizes of the edges of the root host) separates all FHTs pairwise. The classes $(v,\alpha)$ are determined for $v\le6$. The census of the sectors $A_F$ is computed in floating point, exhaustively for $v\le6$ and in part for $v=7,8,9$.

\paragraph{The algebra $A_F$, for a chosen character.} The algebra attached to an FHT depends on a declared choice of character. We fix it once for all as the sign $\eta_V$ of the action on the vertices; the other canonical characters need not agree with it (Remark~\ref{rem:whicheta}). For the sign character, $\mathbb R^{\mathrm{sgn}}[S_4]\cong A_F=\mathbb C\oplus\mathbb H\oplus M_3(\mathbb C)$. Among all finite groups and real $2$-cocycles, $(S_4,[w_1^2])$ is the only pair that gives $A_F$, the group up to isomorphism and the class of the cocycle after transport. With the rule $\eta_V$, $\mathbb R^{\eta_V}[\Aut\F]\cong A_F$ if and only if $\Aut\F\cong S_4$ and some automorphism is odd on the vertices. This occurs at every $v\ge4$; the four classes with $\Aut\F\cong S_4$ acting transitively on six vertices, where $\eta_V$ is trivial, give $\mathbb R[S_4]$ instead. Theorem~\ref{thm:hodge} reads the twist as an orientation character, without selecting it, and the factors of $A_F$ as types of Hodge duality. Every twisted algebra of a sign character has a commutative factor, and among the algebras of the ladder only $A_F$ is a whole twisted algebra of a sign character. Some simple factors of the larger algebras appear as sectors, in the spinorial sector of $\mathbb R^{\det}[GL(2,3)]$ and in $\mathbb R^{\mathrm{sgn}}[S_5]$.

\paragraph{Representability and selection.} Every finite group is the automorphism group of a flat FHT. Every faithful action of $G$ on a finite connected simplicial complex is rigidified into a flat FHT with automorphism group $G$ and the same reduced integral homology, as $\mathbb ZG$-modules; faithfulness can be obtained by enlarging the complex (Remark~\ref{rem:faithful}). Whether every finite-dimensional graded $\mathbb QG$-module concentrated in degrees $\ge1$ and satisfying the Burnside condition is realised in this way is open. FHTs carry so many symmetries and homologies that the existence of a carrier says little. The content lies in principles that select the carriers, such as the whole twisted algebra.

\section*{Declaration on the use of generative AI}
This work was carried out with substantial assistance from large language models of the Claude family (Anthropic), namely Claude Sonnet~5.5, Claude Opus~5.5, and Claude Fable~5.2.

\emph{Role of the author.} The author conceived the research programme and introduced its mathematical objects (fractal hyper-trees and their main constructions). The AI tools proposed options; the author chose the research orientations. He designed the verification method, with success criteria fixed in advance and an explicit epistemic status for each statement, checked by audits in several stages. He did the main analytical calculations by hand and ran computations on his own machine over tens of thousands of FHT representatives.

\emph{Role of the AI tools.} The models were used to draft and edit the manuscript under the author's direction, to suggest constructions and proof strategies, to search the literature and to write the analysis and verification scripts (deposited in~\FHTrepo). They were also used for automated audits and cross-model checks. These are internal checks; they are not human peer review.

\emph{Verification and responsibility.} The author has reviewed all proofs and the algorithms of the verification scripts, and has checked that the bibliographic references exist and are relevant. In this article, large-scale enumerations and algebraic checks rely on the deposited programs and validation protocols detailed in Section~\ref{sec:verification}, rather than manual recalculation. The AI tools are not authors; the author takes full responsibility for the entire content of this article. This includes any error introduced with their assistance.


\newpage
\appendix

\section{Fixed-host numbers}\label{sec:fix}

Table~\ref{tab:fix} lists the numbers $\operatorname{Fix}(\lambda)$ of Definition~\ref{def:hosts}: the number of connected clutters on $m$ labelled points invariant under a permutation of cycle type $\lambda\vdash m$ (cycle lengths written in decreasing order). The identity term $\lambda=(1^m)$ is the number of labelled connected clutters on $m$ points. For every $m\le7$ one checks that $\frac1{m!}\sum_{\lambda\vdash m}\frac{m!}{z_\lambda}\operatorname{Fix}(\lambda)=S(m)$, the number of unlabelled connected clutters: $1,3,14,157,15942,489980450$ for $m=2,\dots,7$. The entries marked $^\ddagger$ ($m\ge7$, identity) are computed from the Dedekind numbers, and the entries marked $^\dagger$ ($m=8$) by the recurrence of Appendix~\ref{sec:recfix}; none is enumerated. For $m=8$ one checks that $\sum_{\lambda\vdash8}\frac{1}{z_\lambda}\operatorname{Fix}(\lambda)=S(8)=1392195547909966876$, the exact value of Section~\ref{sec:effective}.

\begin{longtable}{c l r}
\caption{Numbers $\operatorname{Fix}(\lambda)$ of hosts invariant under a permutation of cycle type $\lambda$.}\label{tab:fix}\\
\toprule
$m$ & $\lambda$ & $\operatorname{Fix}(\lambda)$ \\
\midrule
\endfirsthead
\toprule
$m$ & $\lambda$ & $\operatorname{Fix}(\lambda)$ \\
\midrule
\endhead
2 & $(2)$ & 1 \\
 & $(1,1)$ & 1 \\
\midrule
3 & $(3)$ & 2 \\
 & $(2,1)$ & 3 \\
 & $(1,1,1)$ & 5 \\
\midrule
4 & $(4)$ & 4 \\
 & $(3,1)$ & 6 \\
 & $(2,2)$ & 16 \\
 & $(2,1,1)$ & 22 \\
 & $(1,1,1,1)$ & 84 \\
\midrule
5 & $(5)$ & 8 \\
 & $(4,1)$ & 20 \\
 & $(3,2)$ & 22 \\
 & $(3,1,1)$ & 60 \\
 & $(2,2,1)$ & 236 \\
 & $(2,1,1,1)$ & 652 \\
 & $(1,1,1,1,1)$ & 6\,348 \\
\midrule
6 & $(6)$ & 38 \\
 & $(5,1)$ & 43 \\
 & $(4,2)$ & 296 \\
 & $(4,1,1)$ & 378 \\
 & $(3,3)$ & 542 \\
 & $(3,2,1)$ & 391 \\
 & $(3,1,1,1)$ & 2\,963 \\
 & $(2,2,2)$ & 8\,390 \\
 & $(2,2,1,1)$ & 23\,000 \\
 & $(2,1,1,1,1)$ & 154\,342 \\
 & $(1,1,1,1,1,1)$ & 7\,743\,728 \\
\midrule
7 & $(7)$ & 98 \\
 & $(6,1)$ & 679 \\
 & $(5,2)$ & 233 \\
 & $(5,1,1)$ & 1\,325 \\
 & $(4,3)$ & 265 \\
 & $(4,2,1)$ & 26\,150 \\
 & $(4,1,1,1)$ & 57\,610 \\
 & $(3,3,1)$ & 68\,093 \\
 & $(3,2,2)$ & 9\,951 \\
 & $(3,2,1,1)$ & 57\,383 \\
 & $(3,1,1,1,1)$ & 2\,041\,595 \\
 & $(2,2,2,1)$ & 11\,996\,266 \\
 & $(2,2,1,1,1)$ & 67\,774\,938 \\
 & $(2,1,1,1,1,1)$ & 2\,206\,396\,178 \\
 & $(1,1,1,1,1,1,1)$ & 2\,414\,572\,893\,530$^\ddagger$ \\
\midrule
8 & $(8)$ & 2\,356 \\
 & $(7,1)$ & 3\,657 \\
 & $(6,2)$ & 69\,954 \\
 & $(6,1,1)$ & 141\,388 \\
 & $(5,3)$ & 829 \\
 & $(5,2,1)$ & 20\,530 \\
 & $(5,1,1,1)$ & 522\,922 \\
 & $(4,4)$ & 3\,211\,190 \\
 & $(4,3,1)$ & 24\,459 \\
 & $(4,2,2)$ & 37\,832\,074 \\
 & $(4,2,1,1)$ & 93\,886\,248 \\
 & $(4,1,1,1,1)$ & 423\,760\,086 \\
 & $(3,3,2)$ & 3\,605\,736 \\
 & $(3,3,1,1)$ & 535\,150\,858 \\
 & $(3,2,2,1)$ & 16\,356\,331 \\
 & $(3,2,1,1,1)$ & 401\,256\,501 \\
 & $(3,1,1,1,1,1)$ & 262\,788\,280\,879$^\dagger$ \\
 & $(2,2,2,2)$ & 2\,038\,188\,149\,802$^\dagger$ \\
 & $(2,2,2,1,1)$ & 7\,377\,622\,645\,180$^\dagger$ \\
 & $(2,2,1,1,1,1)$ & 182\,754\,897\,612\,838$^\dagger$ \\
 & $(2,1,1,1,1,1,1)$ & 101\,627\,841\,297\,079\,440$^\dagger$ \\
 & $(1,1,1,1,1,1,1,1)$ & \begin{tabular}[t]{@{}r@{}}56130\\437190053299918162\end{tabular}$^\ddagger$ \\
\bottomrule
\end{longtable}

\section{Fixed hosts of involutive types by recurrence}\label{sec:recfix}

For $m=8$ the types $(2,1^6)$, $(2,2,1^4)$, $(2,2,2,1,1)$, $(2,2,2,2)$ and $(3,1^5)$ have $2\cdot10^{11}$ to $10^{17}$ invariant hosts, and the exhaustive search is out of reach. They are obtained from the numbers of invariant antichains, which satisfy a recurrence of Dedekind type, and from a plethystic logarithm. Let $g$ be a permutation of a set $X$ with $\card X=m$ and $A(g)$ the number of $g$-invariant \emph{upsets} of the Boolean lattice $2^X$; since an upset is the set of supersets of its antichain of minimal elements, $A(g)$ is also the number of $g$-invariant antichains of subsets of $X$ (the empty antichain and $\{\varnothing\}$ included). For $g=\mathrm{id}$, $A(\mathrm{id})=D(m)$.

\begin{lemma}[Splitting off a $2$-cycle]\label{lem:split2}
Let $g=(a\,b)\sigma$ where $\sigma$ is an involution of $Y=X\setminus\{a,b\}$. For an upset $u$ of $2^Y$ let $N^\sigma_-(w)$ (resp.\ $N^\sigma_+(w)$) be the number of $\sigma$-invariant upsets of $2^Y$ contained in (resp.\ containing) $w$. Then
\[ A(g)=\sum_{u}N^\sigma_-(u\cap\sigma u)\,N^\sigma_+(u\cup\sigma u), \]
the sum being over all upsets $u$ of $2^Y$.
\end{lemma}
\begin{proof}
A family $U$ of subsets of $X$ is described by $U_\varnothing,U_a,U_b,U_{ab}\subseteq2^Y$, where $U_S=\{T\subseteq Y:T\cup S\in U\}$. It is an upset if and only if the four families are upsets and $U_\varnothing\subseteq U_a,U_b\subseteq U_{ab}$. Since $g$ maps $T$ to $\sigma(T)$, $T\cup\{a\}$ to $\sigma(T)\cup\{b\}$ and $T\cup\{a,b\}$ to $\sigma(T)\cup\{a,b\}$, $U$ is $g$-invariant if and only if $U_\varnothing$ and $U_{ab}$ are $\sigma$-invariant and $U_b=\sigma(U_a)$. Put $u=U_a$. The $\sigma$-invariant upset $U_\varnothing$ must lie in $u$ and in $\sigma u$, that is in $u\cap\sigma u$, and conversely; and $\sigma^2=1$ makes $u\cap\sigma u$ and $u\cup\sigma u$ $\sigma$-invariant. Similarly $U_{ab}\supseteq u\cup\sigma u$.
\end{proof}

\begin{lemma}[Splitting off a $3$-cycle]\label{lem:split3}
Let $g=(a\,b\,c)$ be a $3$-cycle times the identity of $Y=X\setminus\{a,b,c\}$. Then $A(g)$ is the number of chains $U_0\subseteq U_1\subseteq U_2\subseteq U_3$ of upsets of $2^Y$.
\end{lemma}
\begin{proof}
With $U_S$ as above, invariance under the rotation of $a,b,c$ means that $U_S$ depends only on $\card S$, and the upset condition is $U_0\subseteq U_1\subseteq U_2\subseteq U_3$.
\end{proof}

For $\sigma=\mathrm{id}$ and $\card Y=6$ the sum of Lemma~\ref{lem:split2} counts the chains $u_0\subseteq u\subseteq u_2$ of upsets of $2^Y$; it is evaluated by writing each of the $7\,828\,354$ upsets of $2^Y$ as a pair of nested upsets of $2^5$ (there are $7\,581$ of them). For the other involutive types $\sigma\ne\mathrm{id}$ has at most $6$ points, and the $\sigma$-invariant upsets of $2^Y$ are listed from the $7\,828\,354$ upsets of $2^Y$.

\begin{proposition}[From invariant antichains to connected clutters]\label{prop:plethlog}
Let $g$ be of cycle type $\lambda\vdash m$ and write $A_{\mathrm{ne}}(\lambda)=A(\lambda)-1$ (antichains of non-empty sets). Then
\[ \sum_{\lambda}\frac{A_{\mathrm{ne}}(\lambda)}{z_\lambda}p_\lambda=\operatorname{Exp}\big(2p_1+Z_{\mathrm{cc}}\big),\qquad
   \operatorname{Exp}(f)=\exp\Big(\sum_{k\ge1}\frac1k f(p_i\leftarrow p_{ik})\Big), \]
and so, for $\lambda\vdash m$,
\[ \operatorname{Fix}(\lambda)=A_{\mathrm{ne}}(\lambda)-z_\lambda\,[p_\lambda]\operatorname{Exp}\big(2p_1+Z_{\mathrm{cc}}^{<m}\big), \]
where $Z_{\mathrm{cc}}^{<m}$ is the part of $Z_{\mathrm{cc}}$ of degree $<m$.
\end{proposition}
\begin{proof}
An antichain of non-empty subsets of $X$ splits uniquely into the set of its connected components, each being either a connected clutter on at least two points, a point $x$ with the edge $\{x\}$, or (as a degenerate component) an uncovered point. As species this is $E\circ(2X+\mathrm{CC})$, whose cycle index is $\operatorname{Exp}(2p_1+Z_{\mathrm{cc}})$. The coefficient of $p_\lambda$ with $\card\lambda=m$ is linear in $\operatorname{Fix}(\lambda)/z_\lambda$, with coefficient $1$, and the rest involves only $Z_{\mathrm{cc}}$ in degrees $<m$.
\end{proof}

\paragraph{Validation.} The recurrence reproduces all the numbers $\operatorname{Fix}(\lambda)$ obtained by exhaustive search for the types $(1^m)$, $(2^j,1^r)$ and $(3,1^r)$ with $m\le7$ (twenty values), and the Burnside sum at $m=8$ then gives the exact $S(8)$ (Section~\ref{sec:effective}).
\newpage
\section{FHT List}\label{sec:FHT_list}

{\setlength{\tabcolsep}{6pt}\small
\begin{longtable}{c p{7cm} c c c}
\caption{Dictionary of the FHTs with at most five vertices: realization, number of vertices $\card{V}$, depth $N$ and number of edges $\card{E}$.}\label{tab:fht_list}\\
\toprule
FHT & Realization & $\lvert V\rvert$ & $N$ & $\lvert E\rvert$ \\
\midrule
\endfirsthead
\toprule
FHT & Realization & $\lvert V\rvert$ & $N$ & $\lvert E\rvert$ \\
\midrule
\endhead
$\F_{1}$ & $\varnothing$ & 1 & 0 & 0 \\
\midrule
$\F_{2}$ & 12 & 2 & 1 & 1 \\
\midrule
$\F_{3}$ & 12 13 & 3 & 1 & 2 \\
$\F_{4}$ & 12 13 23 & 3 & 1 & 3 \\
$\F_{5}$ & 123 & 3 & 1 & 1 \\
\midrule
$\F_{6}$ & 12 123 & 3 & 2 & 2 \\
\midrule
$\F_{7}$ & 12 13 14 & 4 & 1 & 3 \\
$\F_{8}$ & 12 13 14 23 & 4 & 1 & 4 \\
\bottomrule
\end{longtable}}

\newpage
\section{Invariants of the FHTs with at most five vertices}\label{sec:annex}

\begin{landscape}
\setlength{\tabcolsep}{2.2pt}
{\scriptsize
\renewcommand{\arraystretch}{1.05}
\begin{longtable}{@{}l l r r l r l rrr rrr rrr@{}}
\caption{Global parameters (bold lines) and host parameters of the $260$ FHTs with at most five vertices. Groups: $\Aut$ on bold lines, $\Gamma$ on host lines. $k$: number of orbits; orbit-length tuple; $\lvert \Inv\rvert{}$: number of fixed points. Column ``$\boldsymbol\beta$ / host'': Betti numbers on bold lines; on host lines the triple $m\,|\,\lvert \widehat{E}_H\rvert{}\,|\,\F_f$. $\PG,\LG,\IG$: primal, edge and incidence graphs, each with $(\rad,\dia,\mec)$.}\label{tab:inv}\\
\toprule
 & \multicolumn{5}{c}{symmetry} & & \multicolumn{3}{c}{$\PG$} & \multicolumn{3}{c}{$\LG$} & \multicolumn{3}{c}{$\IG$}\\
\cmidrule(lr){2-6}\cmidrule(lr){8-10}\cmidrule(lr){11-13}\cmidrule(lr){14-16}
$\F$ & group & order & $k$ & orbits & $\lvert \Inv\rvert{}$ & $\boldsymbol\beta$ / host & $\rad$ & $\dia$ & $\mec$ & $\rad$ & $\dia$ & $\mec$ & $\rad$ & $\dia$ & $\mec$\\
\midrule
\endfirsthead
\caption[]{(continued)}\\
\toprule
$\F$ & group & order & $k$ & orbits & $\lvert \Inv\rvert{}$ & $\boldsymbol\beta$ / host & $\rad$ & $\dia$ & $\mec$ & $\rad$ & $\dia$ & $\mec$ & $\rad$ & $\dia$ & $\mec$\\
\midrule
\endhead
\midrule
\multicolumn{16}{r}{\emph{continued on the next page}}\\
\endfoot
\bottomrule
\endlastfoot
\textbf{$\F_{1}$} & $\{e\}$ & 1 & 1 & (1) & 1 & (1) & -- & -- & -- & -- & -- & -- & -- & -- & -- \\
\midrule
\textbf{$\F_{2}$} & $C_{2}$ & 2 & 1 & (2) & 0 & (1) & 1 & 1 & 1 & 0 & 0 & 0 & 1 & 2 & 5/3 \\
\midrule
\textbf{$\F_{3}$} & $C_{2}$ & 2 & 2 & (2,1) & 1 & (1) & 1 & 2 & 5/3 & 1 & 1 & 1 & 2 & 4 & 16/5 \\
\addlinespace[3pt]
\textbf{$\F_{4}$} & $S_{3}$ & 6 & 1 & (3) & 0 & (1,1) & 1 & 1 & 1 & 1 & 1 & 1 & 3 & 3 & 3 \\
\addlinespace[3pt]
\textbf{$\F_{5}$} & $S_{3}$ & 6 & 1 & (3) & 0 & (1) & 1 & 1 & 1 & 0 & 0 & 0 & 1 & 2 & 7/4 \\
\addlinespace[3pt]
\textbf{$\F_{6}$} & $C_{2}$ & 2 & 2 & (2,1) & 1 & (1) & 1 & 1 & 1 & 1 & 1 & 1 & 2 & 3 & 12/5 \\
\quad\textit{root} & $\{e\}$ & 1 & 2 & (1,1) & 2 & $2\,|\,1\,|\,\F_{2}$ & 1 & 1 & 1 & 0 & 0 & 0 & 1 & 2 & 5/3 \\
\quad$\F_{2}$ & $C_{2}$ & 2 & 1 & (2) & 0 & $2\,|\,1\,|\,\F_{2}$ & 1 & 1 & 1 & 0 & 0 & 0 & 1 & 2 & 5/3 \\
\end{longtable}}
\end{landscape}

\section{Children and edges of the root host}\label{sec:annexB}

For each FHT $\F_n$ with at least two vertices, the table gives the parameters of Section~\ref{sec:childedge} for the root host; the host of a proper component of type $\F_k$ is the root host of $\F_k$. Columns: the child profile, written $\F_k[s_1,s_2,\dots]{\times}\mu$ for $\mu$ children of type $\F_k$ whose incident host edges have sizes $s_1,s_2,\dots$; the edge-size profile $s^{\mu}$ ($\mu$ host edges of size $s$); the weight profile (omitted, ``--'', for flat FHTs); $\lvert E(\PG)\rvert{}$ and $\lvert E(\LG)\rvert{}$; the edge-orbit tuple of $\Gamma$ on the host edges; and the edge-orbit tuple of $\Aut(\F)$ on the edges of the realization.

\begin{landscape}
\setlength{\tabcolsep}{3pt}
{\scriptsize
\begin{longtable}{@{}l p{8.2cm} l l r r l l@{}}
\caption{Children and edges of the root host of the FHTs with at most five vertices.}\label{tab:childedge}\\
\toprule
$\F$ & child profile & sizes & weights & $\lvert E(\PG)\rvert{}$ & $\lvert E(\LG)\rvert{}$ & $\Gamma$ on host edges & $\Aut$ on edges\\
\midrule
\endfirsthead
\caption[]{(continued)}\\
\toprule
$\F$ & child profile & sizes & weights & $\lvert E(\PG)\rvert{}$ & $\lvert E(\LG)\rvert{}$ & $\Gamma$ on host edges & $\Aut$ on edges\\
\midrule
\endhead
\midrule
\multicolumn{8}{r}{\emph{continued on the next page}}\\
\endfoot
\bottomrule
\endlastfoot
$\F_{2}$ & $\F_{1}[2]{\times}2$ & $2$ & -- & 1 & 0 & (1) & (1) \\
\midrule
$\F_{3}$ & $\F_{1}[2]{\times}2$\; $\F_{1}[2,2]$ & $2^{2}$ & -- & 2 & 1 & (2) & (2) \\
$\F_{4}$ & $\F_{1}[2,2]{\times}3$ & $2^{3}$ & -- & 3 & 3 & (3) & (3) \\
$\F_{5}$ & $\F_{1}[3]{\times}3$ & $3$ & -- & 3 & 0 & (1) & (1) \\
$\F_{6}$ & $\F_{1}[2]$\; $\F_{2}[2]$ & $2$ & $3$ & 1 & 0 & (1) & (1,1) \\
\midrule
$\F_{7}$ & $\F_{1}[2]{\times}3$\; $\F_{1}[2,2,2]$ & $2^{3}$ & -- & 3 & 3 & (3) & (3) \\
$\F_{8}$ & $\F_{1}[2]$\; $\F_{1}[2,2]{\times}2$\; $\F_{1}[2,2,2]$ & $2^{4}$ & -- & 4 & 5 & (2,1,1) & (2,1,1) \\
$\F_{9}$ & $\F_{1}[2,2]{\times}2$\; $\F_{1}[2,2,2]{\times}2$ & $2^{5}$ & -- & 5 & 8 & (4,1) & (4,1) \\
\end{longtable}}
\end{landscape}


\begin{thebibliography}{99}

\bibitem{berge_hypergraphs}
C.~Berge,
\emph{Hypergraphs: Combinatorics of Finite Sets},
North-Holland Mathematical Library, vol.~45, North-Holland, Amsterdam, 1989.

\bibitem{voloshin_hypergraphs}
V.~I.~Voloshin,
\emph{Introduction to Graph and Hypergraph Theory},
Nova Science Publishers, New York, 2009.

\bibitem{bretto_hypergraphs}
A.~Bretto,
\emph{Hypergraph Theory: An Introduction},
Mathematical Engineering, Springer, Cham, 2013.

\bibitem{battiston_beyond}
F.~Battiston, G.~Cencetti, I.~Iacopini, V.~Latora, M.~Lucas, A.~Patania,
J.-G.~Young, G.~Petri,
Networks beyond pairwise interactions: structure and dynamics,
\emph{Phys. Rep.} \textbf{874} (2020), 1--92.

\bibitem{bianconi_higher}
G.~Bianconi,
\emph{Higher-Order Networks: An Introduction to Simplicial Complexes},
Cambridge Elements, Cambridge University Press, 2021.

\bibitem{lambiotte_higher}
R.~Lambiotte, M.~Rosvall, I.~Scholtes,
From networks to optimal higher-order models of complex systems,
\emph{Nature Phys.} \textbf{15} (2019), 313--320.

\bibitem{sperner_1928}
E.~Sperner,
Ein Satz \"uber Untermengen einer endlichen Menge,
\emph{Math. Z.} \textbf{27} (1928), 544--548.

\bibitem{ravasz_hierarchical}
E.~Ravasz, A.-L.~Barab\'asi,
Hierarchical organization in complex networks,
\emph{Phys. Rev. E} \textbf{67} (2003), 026112.

\bibitem{johnson_clustering}
S.~C.~Johnson,
Hierarchical clustering schemes,
\emph{Psychometrika} \textbf{32} (1967), 241--254.

\bibitem{shapley_simple}
L.~S.~Shapley,
Simple games: an outline of the descriptive theory,
\emph{Behavioral Sci.} \textbf{7} (1962), 59--66.

\bibitem{schroder_1870}
E.~Schr\"oder,
Vier combinatorische Probleme,
\emph{Z. Math. Phys.} \textbf{15} (1870), 361--376.

\bibitem{stanley_ec2}
R.~P.~Stanley,
\emph{Enumerative Combinatorics, Vol.~2},
Cambridge Studies in Advanced Mathematics, vol.~62, Cambridge University Press, 1999.

\bibitem{riordan_shannon}
J.~Riordan, C.~E.~Shannon,
The number of two-terminal series-parallel networks,
\emph{J. Math. Phys.} \textbf{21} (1942), 83--93.

\bibitem{gallai_1967}
T.~Gallai,
Transitiv orientierbare Graphen,
\emph{Acta Math. Acad. Sci. Hungar.} \textbf{18} (1967), 25--66.

\bibitem{billera_clutters}
L.~J.~Billera,
On the composition and decomposition of clutters,
\emph{J. Combin. Theory Ser. B} \textbf{11} (1971), 234--245.

\bibitem{mohring_radermacher}
R.~H.~M\"ohring, F.~J.~Radermacher,
Substitution decomposition for discrete structures and connections with
combinatorial optimization,
\emph{Annals of Discrete Mathematics} \textbf{19} (1984), 257--356.

\bibitem{bll_species}
F.~Bergeron, G.~Labelle, P.~Leroux,
\emph{Combinatorial Species and Tree-like Structures},
Encyclopedia of Mathematics and its Applications, vol.~67,
Cambridge University Press, 1998.

\bibitem{harary_palmer}
F.~Harary, E.~M.~Palmer,
\emph{Graphical Enumeration},
Academic Press, New York, 1973.

\bibitem{mckay_nauty}
B.~D.~McKay, A.~Piperno,
Practical graph isomorphism, II,
\emph{J. Symbolic Comput.} \textbf{60} (2014), 94--112.

\bibitem{oeis_A261006}
OEIS Foundation Inc.,
\emph{The On-Line Encyclopedia of Integer Sequences}, sequence A261006
(Number of unlabeled connected simplicial complexes with n nodes),
\url{https://oeis.org/A261006}, accessed 2026.

\bibitem{oeis_A000669}
OEIS Foundation Inc.,
\emph{The On-Line Encyclopedia of Integer Sequences}, sequence A000669
(Number of series-reduced planted trees with n leaves), \url{https://oeis.org/A000669}.

\bibitem{oeis_A003182}
OEIS Foundation Inc.,
\emph{The On-Line Encyclopedia of Integer Sequences}, sequence A003182
(Dedekind numbers: inequivalent monotone Boolean functions of n or fewer variables, or antichains of subsets of an n-set), \url{https://oeis.org/A003182}.

\bibitem{oeis_A000372}
OEIS Foundation Inc.,
\emph{The On-Line Encyclopedia of Integer Sequences}, sequence A000372
(Dedekind numbers: number of monotone Boolean functions of $n$ variables), \url{https://oeis.org/A000372}.

\bibitem{pawelski9}
B.~Pawelski,
On the number of inequivalent monotone Boolean functions of 9 variables,
arXiv:2305.06346 (2023).

\bibitem{wiedemann_D8}
D.~Wiedemann,
A computation of the eighth Dedekind number,
\emph{Order} \textbf{8} (1991), 5--6.

\bibitem{ccm07}
A.~H.~Chamseddine, A.~Connes, M.~Marcolli,
\emph{Gravity and the standard model with neutrino mixing},
Adv.\ Theor.\ Math.\ Phys.\ \textbf{11} (2007), 991--1089.

\bibitem{dowker}
C.~H. Dowker, Homology groups of relations, \textit{Ann. of Math.} \textbf{56} (1952), 84--95.

\bibitem{bjorner}
A.~Bj\"orner, Topological methods, in \textit{Handbook of Combinatorics}, Vol.~II, North-Holland, 1995, 1819--1872.

\bibitem{joyal}
A.~Joyal, Une th\'eorie combinatoire des s\'eries formelles, \textit{Adv. Math.} \textbf{42} (1981), 1--82.

\bibitem{oeis}
OEIS Foundation Inc., \textit{The On-Line Encyclopedia of Integer Sequences}, sequences A000311, A000669 and A000638.

\bibitem{gz}
P.~Gabriel, M.~Zisman, \textit{Calculus of Fractions and Homotopy Theory}, Springer, 1967.

\bibitem{serre}
J.-P.~Serre, \textit{Linear Representations of Finite Groups}, Springer, 1977.

\bibitem{connes96}
A.~Connes,
Gravity coupled with matter and the foundation of non-commutative geometry,
\emph{Comm. Math. Phys.} \textbf{182} (1996), 155--176.

\bibitem{cc08}
A.~H.~Chamseddine, A.~Connes,
Why the Standard Model,
\emph{J. Geom. Phys.} \textbf{58} (2008), 38--47.

\bibitem{vs15}
W.~D.~van Suijlekom,
\emph{Noncommutative Geometry and Particle Physics},
Mathematical Physics Studies, Springer, Dordrecht, 2015.


\bibitem{schur1904}
I.~Schur,
\"Uber die Darstellung der endlichen Gruppen durch gebrochene lineare Substitutionen,
\emph{J. Reine Angew. Math.} \textbf{127} (1904), 20--50.

\bibitem{dummitfoote}
D.~S.~Dummit and R.~M.~Foote,
\emph{Abstract Algebra}, 3rd ed., Wiley, 2004.

\bibitem{isaacs}
I.~M.~Isaacs,
\emph{Character Theory of Finite Groups},
Academic Press, New York, 1976.

\bibitem{karpilovsky}
G.~Karpilovsky,
\emph{Projective Representations of Finite Groups},
Marcel Dekker, New York, 1985.

\bibitem{frucht39}
R.~Frucht,
Herstellung von Graphen mit vorgegebener abstrakter Gruppe,
\emph{Compositio Math.} \textbf{6} (1939), 239--250.

\bibitem{babai74}
L.~Babai,
On the minimum order of graphs with given group,
\emph{Canad. Math. Bull.} \textbf{17} (1974), 467--470.

\bibitem{carlsson81}
G.~Carlsson,
A counterexample to a conjecture of Steenrod,
\emph{Invent. Math.} \textbf{64} (1981), 171--174.

\bibitem{kahn86}
P.~J.~Kahn,
Rational Moore $G$-spaces,
\emph{Trans. Amer. Math. Soc.} \textbf{298} (1986), 245--271.

\bibitem{tomdieck87}
T.~tom Dieck,
\emph{Transformation Groups},
de Gruyter Studies in Mathematics, vol.~8, de Gruyter, Berlin, 1987.

\bibitem{hpy13}
I.~Hambleton, S.~Pamuk, E.~Yal\c{c}{\i}n,
Equivariant CW-complexes and the orbit category,
\emph{Comment. Math. Helv.} \textbf{88} (2013), 369--425.

\bibitem{bartel_dokchitser}
A.~Bartel, T.~Dokchitser,
Brauer relations in finite groups,
\emph{J. Eur. Math. Soc.} \textbf{17} (2015), 2473--2512.

\bibitem{cgv25}
C.~Costoya, R.~Gomes, A.~Viruel,
Rigidifying simplicial complexes and realizing group actions,
preprint, arXiv:2509.09646v2 (2025).

\bibitem{cc97}
A.~H.~Chamseddine, A.~Connes,
The spectral action principle,
\emph{Comm. Math. Phys.} \textbf{186} (1997), 731--750.

\bibitem{ps74}
J.~C.~Pati, A.~Salam,
Lepton number as the fourth ``color'',
\emph{Phys. Rev. D} \textbf{10} (1974), 275--289.

\bibitem{ccvs13ps}
A.~H.~Chamseddine, A.~Connes, W.~D.~van Suijlekom,
Beyond the spectral Standard Model: emergence of Pati--Salam unification,
\emph{J. High Energy Phys.} \textbf{11} (2013), 132.

\bibitem{dlm14}
A.~Devastato, F.~Lizzi, P.~Martinetti,
Grand symmetry, spectral action and the Higgs mass,
\emph{J. High Energy Phys.} \textbf{01} (2014), 042.

\bibitem{bartel_dokchitser16}
A.~Bartel, T.~Dokchitser,
Rational representations and permutation representations of finite groups,
\emph{Math. Ann.} \textbf{364} (2016), 539--558.

\end{thebibliography}
\end{document}